\documentclass[12pt]{article}
\usepackage[T1]{fontenc}
\usepackage[utf8]{inputenc}
\usepackage{lmodern}
\usepackage{amssymb,amsmath,amsthm}
\usepackage{bbm}
\usepackage{graphicx}
\usepackage{float}
\usepackage{xcolor}
\usepackage[colorlinks,linkcolor=blue!70!black]{hyperref} 
\usepackage[a4paper,margin=1.5cm]{geometry}
\usepackage{tikz}

\usepackage{tikz-cd}

\usetikzlibrary{decorations.pathreplacing, patterns,shapes,snakes}

\usepackage{pgfplots}

\newcommand{\ra}{\rightarrow}

\newcommand{\thh}{\text{th}}

\newcommand{\be}{\begin{equation}}
\newcommand{\ee}{\end{equation}}
\newcommand{\bi}{\begin{itemize}}
\newcommand{\ei}{\end{itemize}}
\newcommand{\br}{\begin{eqnarray}}
\newcommand{\er}{\end{eqnarray}}

\newcommand{\commentout}[1]{}

\newcommand{\TV}{\text{TV}}

\newcommand{\Wah}{W}
\newcommand{\ST}{\text{stop}}
\newcommand{\WD}{\text{WD}}

\newcommand{\Rm}{{\mathbb R}}
\newcommand{\Pm}{{\mathbb P}}
\newcommand{\Qm}{{\mathbb Q}}
\newcommand{\expE}{{\mathbb E}}
\newcommand{\Ind}{{\mathbbm{1}}}

\newtheorem{theo}{Theorem}[section]
\newtheorem{lem}[theo]{Lemma}
\newtheorem{defin}[theo]{Definition}
\newtheorem{cond}[theo]{Condition}
\newtheorem{prop}[theo]{Proposition}
\newtheorem{cor}[theo]{Corollary}
\newtheorem{rmk}[theo]{Remark}

\newtheorem{prob}[theo]{Problem}
\newtheorem{example}[theo]{Example}

\newcommand{\Law}{{\mathcal{L}}}

\newtheorem*{rep@theorem}{\rep@title}
\newcommand{\newreptheorem}[2]{%
\newenvironment{rep#1}[1]{%
 \def\rep@title{#2 \ref{##1}}%
 \begin{rep@theorem}}%
 {\end{rep@theorem}}}
\makeatother
\newreptheorem{theorem}{Theorem}
\newreptheorem{lemma}{Lemma}
\newreptheorem{defin}{Definition}
\newreptheorem{cond}{Condition}
\newreptheorem{prop}{Proposition}
\newreptheorem{lem}{Lemma}
\newreptheorem{cor}{Corollary}
\newreptheorem{rmk}{Remark}
\newreptheorem{exer}{Exercise}
\newreptheorem{conj}{Conjecture}
\newreptheorem{assum}{Assumption}
\newreptheorem{equation}{Equation}
\newreptheorem{problem}{Problem}

\long\def\metanote#1#2{{\color{#1}\
\ifmmode\hbox\fi{\sffamily\mdseries\upshape [#2]}\ }}

\begin{document}
\setcounter{page}{1}

\title{The Fleming-Viot Process with McKean-Vlasov Dynamics}
\author{Oliver Tough\footnote{Department of Mathematics, Duke University, Durham, NC, USA. (oliver.kelsey.tough@duke.edu)} and James Nolen\footnote{Department of Mathematics, Duke University, Durham, NC, USA. (nolen@math.duke.edu)}}
\date{November 23, 2020}

\maketitle

\begin{abstract}
The Fleming-Viot particle system consists of $N$ identical particles diffusing in a domain $U \subset \Rm^d$. Whenever a particle hits the boundary $\partial U$, that particle jumps onto another particle in the interior. It is known that this system provides a particle representation for both the Quasi-Stationary Distribution (QSD) and the distribution conditioned on survival for a given diffusion killed at the boundary of its domain. We extend these results to the case of McKean-Vlasov dynamics. We prove that the law conditioned on survival of a given McKean-Vlasov process killed on the boundary of its domain may be obtained from the hydrodynamic limit of the corresponding Fleming-Viot particle system. We then show that if the target killed McKean-Vlasov process converges to a QSD as $t \to \infty$, such a QSD may be obtained from the stationary distributions of the corresponding $N$-particle Fleming-Viot system as $N\ra\infty$.
\end{abstract}


\tableofcontents

\section{Introduction}

The long-term behaviour of Markovian processes with an absorbing boundary has been studied since the work of Yaglom on sub-critical Galton-Watson processes \cite{Yaglom1947}, a review of which can be found in \cite{Meleard2011}. The long-time limits we obtain are quasi-stationary distribution (QSDs). In this paper we study the behavior of a system of interacting diffusion processes, known as a Fleming-Viot particle system, which is known to provide a particle representation for these long time limits \cite{Burdzy2000}, \cite[Section 6]{Meleard2011}.

Given an open set $U \subset \Rm^d$, we consider $N \geq 2$ particles diffusing in the domain $U$. The particle positions are denoted by $X_t^{1},\ldots,X_t^{N} \in U$, so that $\vec{X}^N_t=(X_t^{1},\ldots,X_t^{N})$ is a $U^N$-valued stochastic process. A drift acting on the particles will depend on their empirical measure. Let $\mathcal{P}(U)$ be the set of Borel probability measures on $U$, and let $\vartheta^N:U^N \to \mathcal{P}(U)$ be the map which takes the points $x_1,\dots,x_N \in U$ to their empirical measure,
\begin{equation}
 \vartheta^N(x_1,\ldots,x_N) = \frac{1}{N}\sum_{k=1}^N\delta_{x_k},
\label{eq:empirical map}
\end{equation}
which is invariant under permutation of the indices. Given $\vec{X}^N_t$, $\vartheta^N(\vec{X}^N_t)$ is the empirical measure of the $N$ particles (at time $t$), a random measure supported on $U$. We further define a measurable drift:
\begin{equation}
b: \mathcal{P}(U)\times U \ra \Rm^d. 
\label{eq:drift b}
\end{equation}

We now define the particle system which is the subject of this paper:
\begin{defin}[Fleming-Viot Particle System with McKean-Vlasov Dynamics]
Let $\upsilon^N$ be a probability measure on $U^N$, and let $\{ W^i_t\}_{i=1}^N$ be a collection of independent Brownian motions on $\Rm^d$. Then the particle system $\{X^i\}_{i=1}^N \subset U$ with initial distribution $\upsilon^N$ is defined up to a time $\tau_{\WD}:=\tau_\infty\wedge\tau_\ST\wedge\tau_{\max}$ by:
\begin{equation} \label{eq:N-particle system sde}
\left \{ \begin{split}
(i) & \quad \vec{X}^N_0 \sim \upsilon^N. \\
(ii) & \quad \text{For $t \in [0,\tau_{\WD})$ and between jump times (while $\{X^i_t\}_{i=1}^N \subset U$), the particles}\\
& \quad \text{evolve according to the system:} \\
&  \quad dX^i_t=b(\vartheta^N(\vec{X}^N_t),X^i_t)dt+dW^i_t,\quad i = 1,\dots,N, \\
& \quad \text{where the $W^{i}$ are independent Brownian motions in $\Rm^d$.} \\
(iii) & \quad \text{When a particle $X^i$ hits the boundary $\partial U$, $X^i$ instantly jumps }\\
& \quad \text{to the location of another particle chosen independently and uniformly at random.}
\end{split}  \right.
\end{equation}
\label{defin:Fleming-Viot MKV dynamics}
\end{defin}

We let $U^i_k\in\{1,\ldots,N\}\setminus\{i\}$ be the index of the particle i jumps onto on its $k^{\thh}$ death time, so that $\{\{U^i_k\}_{k=1}^\infty\}_{i=1}^N$ are a family of independent random variables such that for each $i \in \{1,\dots,N\}$ the variables $\{ U^i_k\}_{k= 1}^\infty$ are all uniformly distributed on the set $\{1,\dots,N\} \setminus \{i\}$.

We write $\tau_k$ for the $k^{\text{th}}$ jump time of any particle, and moreover $\tau_{\infty}=\lim_{k\ra\infty}\tau_k$, after which the particle system is not well-defined. Furthermore we write $\tau_\ST=\inf\{t>0:\;\exists\; j\neq k\text{ such that }X^j_{t-},\;X^k_{t-}\in \partial U\}$ after which the particle system is not well-defined. Furthermore if the domain $U$ is unbounded, the particles may "escape to infinity" in finite time, after which time the particle system is not well-defined. We write $\tau_{\max}=\inf\big\{t>0:\sup_{\substack{t'\leq t\\ 1\leq i\leq N}}\lvert X^{i}_{t'}\rvert=\infty\big\}$. Thus the particle system is well-defined only up to the time:
\[
\tau_{\WD}=\tau_{\infty}\wedge\tau_{\ST}\wedge \tau_{\max}.
\]

Although the Brownian motions $\{W_i\}_{i=1}^N$ are independent, the drift $b$ in the motion of the $i^{th}$ particle $X^i_t$ may depend on $X^i$ and on the empirical measure $\vartheta^N(\vec{X}^N_t)$ of all $N$ particles. The particles also depend on each other through the rule for relocating a particle when it hits the boundary $\partial U$. Because we do not make strong regularity assumptions on the drift $b$, we will interpret the SDE in \eqref{eq:N-particle system sde} in the weak sense, which we make precise in Definition \ref{defin:Fleming-Viot MKV dynamics weak}.

This system is a generalisation of the Fleming-Viot system introduced in the foundational papers of Burdzy, Holyst, Ingerman, and March \cite{Burdzy1996, Burdzy2000}. Their work involved the particular case of purely Brownian dynamics (i.e. $b \equiv 0$) on a bounded domain $U$. Even if $b \equiv 0$, it is not clear that the system \eqref{eq:N-particle system sde} should be well-defined for all $t > 0$. In particular, the following problem remains open:
\begin{prob}[\cite{Burdzya}]
Consider the $b\equiv 0$ case. Is it true that $\tau_{\infty}=\infty$, almost surely, for any bounded open connected set $D\subseteq \Rm^d$?
\label{prob:obtaining an infinite time horizon for the particle system in the Brownian case for bounded domains with arbitrary boundary}
\end{prob}

In \cite{Burdzy2000,Grigorescu2007a, Bieniek2009}, conditions for the global well-posedness ($\Pm(\tau_{\WD}= +\infty) = 1$) of this system were established for the case $b \equiv 0$ when $U$ is bounded (and the boundary satisfies various additional conditions). Note that the proof given in \cite[Theorem 1.1]{Burdzy2000} features an irreparable error, however implicit in \cite[Theorem 1.4]{Burdzy2000} is another proof when the domain satisfies an interior ball condition. These are complemented by \cite{Grigorescu2012,Villemonais2011}, providing well-posedness for general diffusions on possibly unbounded domains (satisfying various additional conditions). We provide a similar result (Theorem \ref{theo:Global Well-Posedness of the N-Particle System}) for $b\neq 0$ and $U$ being possibly unbounded.

In \cite{Burdzy2000,Grigorescu2004}, Burdzy, et al. also consider the limits $N \to \infty$ and $t \to \infty$. They established that the empirical measure of the particle converges to the solution of the heat equation renormalised to have constant mass 1, corresponding to the distribution of Brownian motion killed at the boundary of its domain, conditioned on survival. The notion of convergence was later strengthened by Grigorescu and Kang in \cite{Grigorescu2004}. In particular if $\frac{1}{N}\sum_{i=1}^N\delta_{X^{i}_0}\ra \nu$ weakly in probability then:
\[
\frac{1}{N}\sum_{i=1}^N\delta_{X^{i}_t}\ra \frac{u}{\lvert u\rvert_{\ast}}=\mathcal{L}_{\nu}(B_t\lvert \tau>t)\quad\text{weakly in probability}
\]
where $\lvert u\rvert_{\ast}$ is the mass of $u_t$, which is a solution of the Dirichlet heat equation:
\[
\partial_t u=\frac{1}{2}\Delta u, \quad u_{\lvert_{\partial U}},\quad u_0=\nu = 0
\]
and where $(B_t)_{0\leq t\leq\tau}$ is a Brownian motion with initial condition $B_0\sim\nu$ stopped at the time $\tau=\inf\{t>0:B_{t}\in \partial U\}$. Note that, by abuse of notation, we are using functions interchangeably with the measures having their density.

Moreover for fixed $N$, Burdzy, et al. \cite[Theorem 1.4]{Burdzy2000} prove that $\vec{X}^N_t$ has a stationary distribution $\mathbf{M}^N$ on $U^N$ to which the distribution of $\vec{X}^N_t$ converges exponentially fast as $t\ra\infty$. Furthermore the corresponding stationary random empirical measure $\chi^N_{\mathbf{M}}\sim \vartheta^N_{\#}\mathbf{M}^N$ converges weakly in probability as $N \to \infty$ to a function $\phi(x)$ which is the principal Dirichlet eigenfunction of the Laplacian on $U$:
\begin{equation}
\begin{split}
\Delta\phi+\lambda\phi=0,\quad \phi>0 \text{ on U},\quad \phi=0\text{ on }\partial U,
\end{split}
\label{eq:min efn and evalue}
\end{equation}
normalised to have integral 1. This corresponds to the quasi-stationary distribution (QSD) for Brownian motion killed on the boundary of its domain:
\[
\mathcal{L}_{\phi}(B_t\lvert \tau>t)=\phi, \quad 0\leq t<\infty.
\]
This QSD is the unique quasi-limiting distribution (QLD) for Brownian motion killed at the boundary of its domain. That is for any initial condition $\nu$:
\[
\mathcal{L}_{\nu}(B_t\lvert \tau>t)\ra \pi\quad\text{as}\quad t\ra\infty.
\]

Similar results have been established for a variety of other Fleming-Viot particle systems with Markovian dynamics: for instance by Ferrari and Maric \cite{Ferrari2006} in the case of countable state spaces and by Villemonais \cite{Villemonais2011} in the case of general Strong Markov processes. These are complemented by generic long-time convergence criteria for the conditional distribution of killed Markov processes \cite[Theorem 7]{Meleard2011}, \cite{Champagnat2018a}. Campi and Fischer \cite{Campi2017} have also considered a similar mean field game with particles killed at the boundary of their domain (corresponding to bankruptcy) and interacting with the renormalised empirical measure (their setup did not feature branching, so the mass decreases over time). 

\subsection*{Summary of results}
In the present article, we extend the results in the Markovian case to the more general system \eqref{eq:N-particle system sde} which includes dynamics of McKean-Vlasov type whereby the particles interact through the dependence of the drift $b$ on the empirical measure. Throughout the paper, we assume that the open set $U \subset \Rm^d$ satisfies the interior ball condition with radius $r>0$: for every $x\in U$ there exists a point $y\in U$ such that $x\in B(y,r)\subseteq U$. We also assume that the drift
\[
b: \mathcal{P}(U)\times U \ra \Rm^d
\]
is measurable with respect to the Borel sigma algebra on $\mathcal{P}(U)\times U$ and uniformly bounded by $B<\infty$, where $\mathcal{P}(U)$ is endowed with the topology of weak convergence of measures.

We begin by establishing in Theorem \ref{theo:Global Well-Posedness of the N-Particle System} global well-posedness of the system \eqref{eq:N-particle system sde} - that is $\Pm(\tau_{\WD}=\infty) = 1$. At the same time, we establish some estimates on the $N$-particle system which shall be used throughout this paper.

We then seek to characterise the behaviour as $t\ra\infty$ for fixed $N<\infty$. Here we must impose an additional assumption: that the domain $U$ is bounded and path-connected. The reason boundedness becomes necessary is that on unbounded domains we have the possibility of mass escaping to infinity over infinite time horizons. We conjecture that a Lyapunov criterion should exist allowing our large time results to be extended to the setting of unbounded domains. In the Markovian case, such Lyapunov criteria have been established in \cite{Champagnat2018,Champagnat2018a}. We establish in Theorem \ref{theo:Ergodicity of the N-Particle System} that the system \eqref{eq:N-particle system sde} is ergodic, having a unique stationary distribution $\psi^N$ on $U^N$. 

We then consider the behaviour of the system \eqref{eq:N-particle system sde} as $N \to \infty$. We no longer need to impose the assumption that $U$ is bounded and path-connected. We will establish a hydrodynamic limit theorem - Theorem \ref{theo:Hydrodynamic Limit Theorem} - which will be the main result of this paper. As we will show, the limit behavior of $\vec{X}^N$ as $N \to \infty$ can be described in terms of the following conditional McKean-Vlasov system:

\begin{equation} \label{eq:MKV sde}
\left \{ \begin{split}
(i) & \quad (X_t:0\leq t\leq \tau)\text{ is a continuous process defined up to the}\\
& \quad \text{stopping time }\tau=\inf\{t>0:X_t\in\partial U\},\\
(ii) & \quad X_t\in U\text{ for }t<\tau,\;X_{\tau} = \lim_{t \nearrow \tau^-} X_t \in \partial U,\\
(iii) & \quad  \text{Initial condition: $X_0 \sim \nu \in \mathcal{P}(U)$,} \\ 
(iv) & \quad X_t\text{ satisfies }dX_t=b(\mathcal{L}(X_t\lvert \tau>t),X_t)dt+dW_t,\quad 0\leq t \leq \tau,\\ & \text{where $W$ is a Brownian motion},
\end{split}  \right. 
\end{equation}
which gives rise to the flow of conditional laws:
\begin{equation}
(\mathcal{L}(X_t\lvert \tau>t):0\leq t<\infty).
\label{eq:flow of conditional laws}
\end{equation}
\begin{rmk}
Strictly speaking we should define $X_t$ as occupying some cemetary state for all $t\geq \tau$. This could be some point seperate from $\Rm^d$ or it could be the point on the boundary it hits at time $\tau$. Nevertheless it shall be more convenient for our purposes for killed processes to be defined only up the killing time $\tau$. By abuse of notation, we are writing $\mathcal{L}(X_t\lvert \tau>t)$ for $\mathcal{L}(X_{t\wedge \tau}\lvert \tau>t)$ and $\mathcal{L}(X_t)$ for the sub-probability measure $\mathcal{L}(X_t\lvert\tau>t)\Pm(\tau>t)$ - so in particular $\mathcal{L}(X_t)$ assigns mass only to $U$ and not to any "cemetary state".
\label{rmk:convention for L(xt)}
\end{rmk}

Such processes have been studied over finite time horizons for instance by Caines, Ho and Song \cite{Caines2019} and in the context of Mean Field Games by Campi and Fischer \cite{Campi2017}. In the SDE, we use $\mathcal{L}(X_t\lvert \tau>t) \in \mathcal{P}(U)$ to denote the law of $X_t$ conditioned on $\tau > t$, where $\tau$ is the first time $X_t$ hits the boundary $\partial U$. For convenience, we define:
\begin{equation}
m_t = \mathcal{L}(X_t\lvert \tau>t) \in \mathcal{P}(U), \quad t \geq 0 \label{eq:mdef}
\end{equation}
and 
\begin{equation}
J_t = - \ln \Pm (\tau > t), \quad t \geq 0. \label{eq:Jdef}
\end{equation}

These are only well-defined for as long as $\Pm(\tau>t)>0$. We therefore define the following:
\begin{defin}[Global Weak Solution to \eqref{eq:MKV sde}]
If a weak solution to \eqref{eq:MKV sde} satisfies $\Pm(\tau>t)>0$ for all $t\in [0,\infty)$ we say it is a global weak solution. 
\label{defin:Global weak solution to MKV SDE}
\end{defin}

We establish in Proposition \ref{prop:Properties of the McKean-Vlasov Process} that all weak solutions are global weak solutions along with the existence, uniqueness in law and time continuity of such solutions. This allows us to uniquely define the following semigroup:
\begin{equation}\label{eq:MKV SDE semigroup}
\begin{split}
G_t(\nu):=\mathcal{L}_{\nu}(X_t\lvert \tau>t)\text{ where }(X,\tau,W)\text{ is a }\\
\text{global weak solution to \eqref{eq:MKV sde} with initial condition }X_0 \sim \nu \end{split}
\end{equation}
which we later show in Proposition \ref{prop:basic properties of Gt} is jointly continuous in $[0,\infty)\times \mathcal{P}(U)$. The density $u = m_te^{-J_t}=\mathcal{L}(X_t)$ corresponds to a weak solution of the following nonlinear transport equation:
\[
\partial_t u + \nabla \cdot \left(b\left(\frac{u}{\lvert u\rvert_*},x\right) u\right) = \frac{1}{2} \Delta u, \quad \quad u \big \vert_{\partial U} = 0
\]
where $\lvert u\rvert_*$ is the mass of u on $U$. 

Returning to our Fleming-Viot system of $N$ particles, we define the empirical measure of the $N$-particle system:
\begin{equation}
m^N_t=\vartheta^N(\vec{X}^N_t).
\label{eq:defin mN}
\end{equation}
which has initial distribution:
\begin{equation}
m^N_0\sim \xi^N:=\vartheta^N_{\#}\upsilon^N.
\label{eq:initial dist mN0}
\end{equation}
Thus $m_0^N$ is a random probability measure on $U$; $\xi^N$ is the law of this random measure and is the pushforward of $\upsilon^N$ under the map $\vartheta^N$.

We further define
\begin{equation}
J^N_t = \frac{1}{N} \sup \{ k \in \mathbb{N}\;|\; \tau_k \leq t \},
\label{eq:defin JN}
\end{equation}
which is the number of jumps of the $N$-particle process up to time $t$, normalized by $1/N$. In Theorem \ref{theo:Hydrodynamic Limit Theorem} we establish $(m^N_t,J^N_t)_{0\leq t<\infty}$ converges uniformly on compacts in probability to $(m_t,J_t)_{0\leq t<\infty}$.

Having established ergodicity for fixed $N$ and hydrodynamic convergence to the flow of conditional laws \eqref{eq:flow of conditional laws} for the system \eqref{eq:MKV sde}, it is natural to ask whether we might obtain convergence in large time for \eqref{eq:flow of conditional laws}. We recall the semigroup \eqref{eq:MKV SDE semigroup} and ask when the limit
\[
\lim_{t\ra\infty}G_t(\nu)=\lim_{t\ra\infty}\mathcal{L}_{\nu}(X_t\lvert \tau>t)
\]
exists. We extend the definitions given in the Markovian case in \cite[Page 5]{Meleard2011}:

\begin{defin}[McKean-Vlasov QLDs and QSDs]
We take a domain $U$ and drift $b$ satisfying the assumptions of Proposition \ref{prop:Properties of the McKean-Vlasov Process} and take the unique associated semigroup as in \eqref{eq:MKV SDE semigroup}:
\[
G_t:\mathcal{P}(U)\ra \mathcal{P}(U).
\]
We let $\pi$ be a Borel probability measure on $U$. We say $\pi$ is a quasi-limiting distribution (QLD) for $(b,U)$ if there is a probability measure $\nu$ on $U$ such that:
\begin{equation}
G_t(\nu)\ra \pi\quad\text{in}\quad\mathcal{P}(U)\quad\text{as}\quad t\ra\infty.
\label{eq:MKV quasi limit dist eq}
\end{equation}
We define $\pi$ to be a quasi-stationary distribution (QSD) for $(b,U)$ if:
\begin{equation}
G_t(\pi)=\pi,\quad 0\leq t<\infty.
\label{eq:MKV quasi stat dist eq}
\end{equation}
We then define the set of QSDs to be:
\begin{equation}
\Pi=\{\pi\in \mathcal{P}(U):\pi\text{ is a QSD for }(b,U)\}.
\label{eq:set of QLDs}
\end{equation}
\end{defin}

We then ask in Problem \ref{prob:conv to equil} when we have:
\begin{equation}
    \text{the limit }\lim_{t\ra\infty}G_t(\nu)\text{ exists for every }\nu\in\mathcal{P}(U)
\label{eq:convergence to non-unique quasi-equil intro}
\end{equation}
(but we do not require the same limit for different $\nu\in\mathcal{P}(U)$). This is the most significant issue left unresolved in this paper; in our later theorems we assume we are working with a case where \eqref{eq:convergence to non-unique quasi-equil intro} does hold. We would not have \eqref{eq:convergence to non-unique quasi-equil intro} if, for example, $G_t(\nu)$ converges to a limit cycle as $t\ra \infty$ for some $\nu\in\mathcal{P}(U)$.

Whereas we are not able to resolve Problem \ref{prob:conv to equil}, we are able to extend \cite[Proposition 1]{Meleard2011} from the Markovian case to the McKean-Vlasov case: establishing in Proposition \ref{prop:properties of QSD} that $\pi$ is a QSD if and only it is a QLD, that QSDs can be characterised as the solutions of a nonlinear eigenproblem, that $\Pi$ is a non-empty compact set (in particular, at least one QSD exists) and that the killing time $\tau$ at quasi-equilibrium is exponentially distributed with rate given by the corresponding eigenvalue. This and all of our later results require the domain $U$ be bounded. 

We demonstrate in the following example that $\Pi$, the set of QSDs, need not be a singleton:

\begin{example}
We assume $U=[-1,1]$ and the drift is given by the first moment:
\[
dX_t=\gamma\expE[X_t\lvert \tau>t]dt+dW_t.
\]
This satisfies the conditions of Proposition \ref{prop:properties of QSD}, so we may check using Part \ref{enum:PDE for QSD} of Proposition \ref{prop:properties of QSD} that the QSDs are given by the following:
\[
\pi_b=Ae^{bx}\cos(\frac{\pi}{2}x)\quad \text{where $b$ is a solution of}\quad b=\tanh(\gamma b)-\frac{8\gamma b}{4\gamma^2b^2+\pi^2}.
\]

For all values of $\gamma$ $\pi_0=A\cos(\frac{\pi}{2}x)$ is a QSD, which for small $\gamma$ is the only QSD. Moreover by calculating the derivative of
\[
F(b)=\tanh(\gamma b)-\frac{8\gamma b}{4\gamma^2b^2+\pi^2}
\]
at 0 we see that $b\mapsto F(b)$ exhibits a pitchfork bifurcation at $\gamma=\frac{\pi^2}{\pi^2-8}$ so that for $\gamma>\frac{\pi^2}{\pi^2-8}$ there are multiple QSDs $\pi_0$ and $\pi_{\pm}$. 

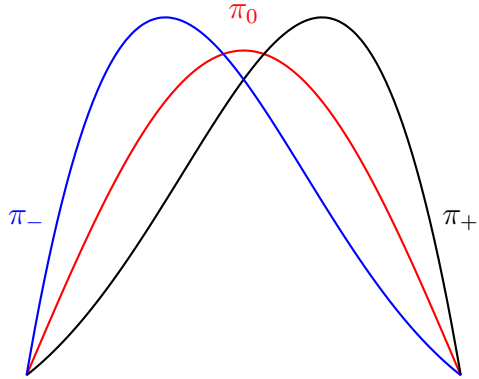
\[
\begin{tikzpicture}
  \begin{axis}[
      axis x line=none,
      axis y line=none,
      domain=-1:1,
    ]
    \addplot[domain=-1:1,samples=50,smooth,thick,red] {pi*cos(deg(pi*x/2))/4};
    \addplot[domain=-1:1,samples=50,smooth,thick,black] {(4+pi^2)*exp(x)*cos(deg(pi*x/2))/(2*pi*(exp(1)+exp(-1)))};
    \addplot[domain=-1:1,samples=50,smooth,thick,blue] {(4+pi^2)*exp(-x)*cos(deg(pi*x/2))/(2*pi*(exp(1)+exp(-1)))};
\node[blue] at (0,375) {$\pi_-$};
\node at (200,375) {$\pi_+$};
\node[red] at (100,875) {$\pi_0$};
  \end{axis}
\end{tikzpicture}
\]

\label{exam:non-uniqueness in bounded MKV case}
\end{example}

We now recall that in Theorem \ref{theo:Ergodicity of the N-Particle System} we establish that $\vec{X}^N_t$ is ergodic with stationary distribution we call $\psi^N$. We may therefore associate to this an empirical measure-valued stationary distribution:
\begin{equation}
    \Psi^N:=\vartheta^N_{\#}\psi^N
\label{eq:defin PsiN}
\end{equation}
which is the stationary distribution for the empirical measure-valued process $m^N_t$. We associate to each $\Psi^N$ a random variable:
\begin{equation}
\pi^N\sim \Psi^N.
\label{eq:defin piN}
\end{equation}

In Theorem \ref{theo:main t theo conv of stat dists} we establish that if we do have \eqref{eq:convergence to non-unique quasi-equil intro} then the $\pi^N$ converge in the $\Wah$ metric in probability to $\Pi$. In other words if we sample a random empirical measure from $\Psi^N$ for large N, then with large probability our random empirical measure is close in the $\Wah$ metric to some QSD $\pi\in\Pi$. This is an extension of \cite[Theorem 1.4 (ii)]{Burdzy2000} which dealt with the $b=0$ case. 

Whilst we show $\pi^N$ is close to the set $\Pi$ with large probability, we do not show that it is close to all of $\Pi$ with large probability. When the QSDs are non-unique - when $\Pi$ contains more than one element - one may ask which QSDs are "selected" by the Fleming-Viot particle system? We conjecture that this should correspond to the stability of the semigroup $G_t$, so that in particular the stability of the QSDs could be determined by sampling $\pi^N$ sufficiently many times and observing which QSDs are "selected".

If we drop the assumption \eqref{eq:convergence to non-unique quasi-equil intro}, we shall see that the distribution of $\pi^N$ converges to the set of invariant measures for the semigroup $G_t$. Thus at least one of the invariant measures can be obtained from the Fleming-Viot particle system. More broadly, due to the McKean-Vlasov interaction, the semigroup $G_t$ could have more interesting dynamical systems properties than in the Markovian case. We therefore ask what about the dynamical system $G_t$ can be deduced from the corresponding Fleming-Viot particle system?

The following diagram summarises the relationship between our results:

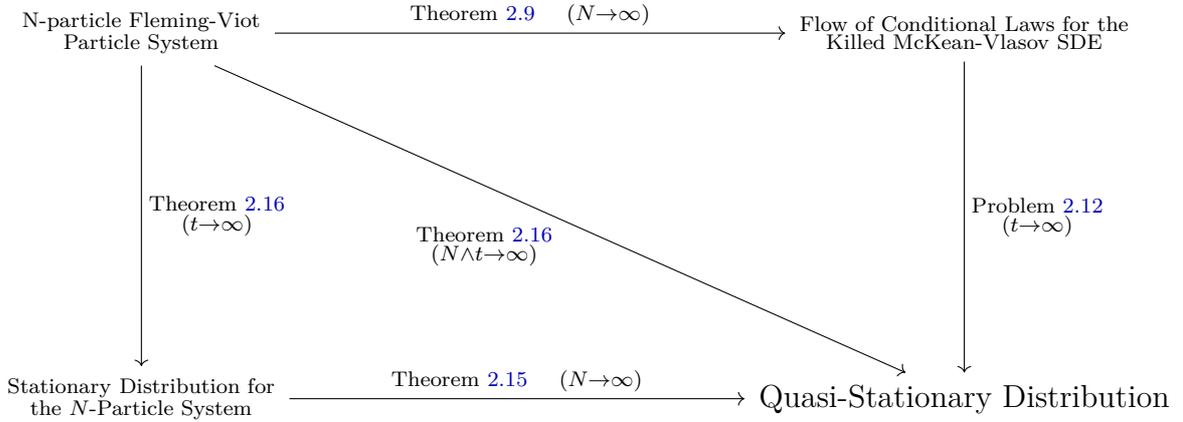
\begin{tikzcd}[column sep=6cm, row sep=4cm]
\substack{\text{N-particle Fleming-Viot}\\\text{Particle System}} \arrow[r, "\text{Theorem }\ref{theo:Hydrodynamic Limit Theorem}\quad (N\ra\infty)"]\arrow[rd, swap,"\substack{\text{Theorem }\ref{theo:main t theo jt conv}\\ (N\wedge t\ra\infty)}"]  \arrow[d, "\substack{\text{Theorem }\ref{theo:main t theo jt conv}\\(t\ra\infty)}"]
& \substack{\text{Flow of Conditional Laws for the}\\ \text{Killed McKean-Vlasov SDE}} \arrow[d, "\substack{\text{Problem }\ref{prob:conv to equil}\\ (t\ra \infty)}"] \\
\substack{\text{Stationary Distribution for}\\ \text{the $N$-Particle System}} \arrow[r, "\text{Theorem }\ref{theo:main t theo conv of stat dists}\quad  (N\ra\infty)"]
& \text{Quasi-Stationary Distribution}
\end{tikzcd}

\section{Statement of Results}
We begin with a more precise description of the particle system \eqref{eq:N-particle system sde} which is the subject of this paper:
\begin{defin} [Weak Solution to \eqref{eq:N-particle system sde}]
\label{defin:Fleming-Viot MKV dynamics weak}
Let $\vec{W}_t = (W^1_t, \dots,W^N_t)$ be a collection of independent Brownian motions on $\Rm^d$ with respect to a right-continuous filtration $\{ \mathcal{F}_t\}_{t \geq 0}$. Let $\upsilon^N$ be a probability measure on $U^N$. We say that $(\vec{X}_t,\vec{W}_t,\mathcal{F}_t)$ is a weak solution to \eqref{eq:N-particle system sde} with initial condition $\upsilon^N$ if $\vec{X}_0\sim \upsilon^N$, and there is an increasing sequence of $\mathcal{F}_t$-stopping times $\{ \tau_k \}_{k = 0}^\infty$ with $\tau_0 = 0$ such that the following hold: 
\begin{enumerate}
    \item $\vec{X}_t$ is a c\'adl\'ag process. For each $k$, $\vec{X}_t$ is continuous on $[\tau_k,\tau_{k+1})$ and satisfies
\begin{equation}
X^i_t= X^i_{\tau_k} + \int_{\tau_k}^t b(\vartheta^N(\vec{X_s}),X^i_s)ds+W^i_t - W^i_{\tau_k},\quad \quad \quad i = 1,\dots,N; \;\; t \in [\tau_k,\tau_{k+1}). \label{eq:weak1}
\end{equation}
For all $k \geq 1$, and with probability one, there is a unique particle index $\ell(k) \in \{1,\dots,N\}$ such that
\begin{equation}
\tau_{k} = \min_{i=1,\dots,N} \inf \{ t > \tau_{k-1} \;|\;\; \lim_{s \to t^-} X^i_s \in U^c \} = \inf \{ t > \tau_{k-1} \;|\;\; \lim_{s \to t^-} X^{\ell(k)}_s \in U^c \}. \label{eq:ellkhit}
\end{equation}
\item
For all $k \geq 1$,
\begin{equation}
\lim_{t \to \tau_k^-} X^j_t = X^j_{\tau_k} \in U, \quad \forall\;\; j \in \{1,\dots,N\} \setminus \{\ell(k)\}, \label{eq:jindexcont}
\end{equation}
and
\begin{equation}
\Pm( X^{\ell(k)}_{\tau_k} = X^{j}_{\tau_k} \;|\; \mathcal{F}_{\tau_k^-} ) = \frac{1}{N-1}, \quad  \forall\;\; j \in \{1,\dots,N\} \setminus \{\ell(k)\} \label{eq:newindexunif}
\end{equation}
hold with probability one.
\end{enumerate}
\end{defin}

We note that this is no longer well-defined once two particles hit the boundary at the same time:
\[
\tau_{\ST}=\inf\{t>0:\;\exists\; j\neq k\text{ such that }X^j_t,\;X^k_t\in \partial U\}.
\]

Moreover if there are an infinite number of stopping times $\tau_k$ in finite time, this is no longer well-defined after the time:
\begin{equation}
\tau_{\infty} = \lim_{k \to \infty} \tau_k.
\label{eq:tau infinity defn}
\end{equation}

Furthermore if the domain $U$ is unbounded, the particles may "escape to infinity" in finite time, after which time the particle system is not well-defined. We write:
\begin{equation}
\tau_{\max}=\inf\big\{t>0:\sup_{\substack{t'\leq t\\ 1\leq i\leq N}}\lvert X^{i}_{t'}\rvert=\infty\big\}.
\end{equation}
Therefore we have $(\vec{X}_t,\vec{W}_t)_{0\leq t<\tau_{\WD}}$ is defined up to the time:
\begin{equation}
    \tau_{\WD}:=\tau_{\ST}\wedge\tau_{\infty}\wedge\tau_{\max}.
    \label{eq:WD stopping time}
\end{equation}

The index $\ell(k)$ in \eqref{eq:ellkhit} is the index of the unique particle that hits the boundary $\partial U$ at time $\tau_{k}$; the statement \eqref{eq:jindexcont} means that the paths of the other particles are continuous at time $\tau_k$; the statement \eqref{eq:ellkhit} means that at time $\tau_k$, the particle with index $\ell(k)$ jumps to the location of another particle chosen uniformly at random from the other $N-1$ particles.

Before stating our results, we define the spaces of measures we employ throughout this paper:
\begin{defin}[Spaces of Measures]
Given a metric space $(\chi,d)$, we equip $(\chi,d)$ with the Borel sigma algebra $\mathcal{B}(\chi)$ and define $\mathcal{P}(\chi)$ to be the space of probability measures on $\chi$ equipped with the topology of convergence of probability measures. We write $\mathcal{M}(\chi)$ for the space of Borel measures on $(\chi,d)$ equipped with the topology of weak convergence of measures. We further define $\mathcal{M}_+(\chi)=\mathcal{M}(\chi)\setminus \{0\}$ to be those measures with positive total mass (equipped with the same topology). 

We equip $\mathcal{P}(\chi)$ with the Wasserstein-1 metric on $\mathcal{P}$ using the bounded metric $d^1(x,y) = 1 \wedge d(x,y)$ on the underlying space $\chi$. We denote this metric $\Wah$ (unless there is a possible confusion as to the underlying metric space $\chi$, in which case we write $d_{\mathcal{P}_{\Wah}(\chi)}$) and write $\mathcal{P}_{\Wah}(\chi)=(\mathcal{P}(\chi),\Wah)$.

\sloppy We shall establish hydrodynamic convergence in the sense of uniform convergence in $\mathcal{P}_{\Wah}(U)\times\Rm_{\geq 0}$ on compact subsets of time in probability. We metrize this as follows. We firstly define the uniform metric over finite time horizons:
\begin{equation}
\begin{split}
    d^{\infty}_{[0,T]}:D([0,T];\mathcal{P}_{\Wah}(U)\times\Rm_{\geq 0})\times D([0,T];\mathcal{P}_{\Wah}(U)\times\Rm_{\geq 0})\ra \Rm_{\geq 0}\\
   d^{\infty}_{[0,T]}((\mu^1,y^1),(\mu^2,y^2))=\sup_{0\leq t\leq T}(\Wah(\mu^1_2,\mu^2_t)+\lvert y^1_t-y^2_t\rvert).
\end{split}
\label{eq:dinfty0T uniform metric}
\end{equation}

We then define the following metric:
\begin{equation}
\begin{split}
    d^{\infty}:D([0,\infty);\mathcal{P}_{\Wah}(U)\times\Rm_{\geq 0})\times D([0,\infty);\mathcal{P}_{\Wah}(U)\times\Rm_{\geq 0})\ra \Rm_{\geq 0}\\
   d^{\infty}(f,g)=\sum_{T=1}^{\infty}2^{-T}(d^{\infty}_{[0,T]}((f_t)_{0\leq t\leq T},(g_t)_{0\leq t\leq T})\wedge 1).
\end{split}
\label{eq:dinfty0T locally uniform metric}
\end{equation}
This metrises uniform convergence on compacts, which means that
\[
d^{\infty}((\mu^n_t,y^n_t)_{0\leq t<\infty},(\mu,y)_{0\leq t<\infty})\ra 0\quad\text{as}\quad n\ra\infty
\]
if and only if 
\[
\sup_{t\leq T}\Wah(\mu^n_t,\mu_t)\ra 0\quad\text{and}\quad \sup_{t\leq T}\lvert y^n_t - y_t\rvert \ra 0\quad\text{as}\quad n\ra\infty\quad\text{for every }T<\infty.
\]

Thus the random $\mathcal{P}_{\Wah}(U)\times \Rm_{\geq 0}$-valued Cadlag processes $(\mu^N,y^N)$ converge to $(\mu,y)$ uniformly in $\mathcal{P}_{\Wah}(U)\times \Rm_{\geq 0}$ on compacts in probability if and only if $d^{\infty}((\mu^N,y^N),(\mu,y))\ra 0$ in probability.

We shall also make use of the Total Variation norm, which we label $\lvert\lvert . \rvert\rvert_{\TV}$.
\label{defin:spaces of measures}
\end{defin}

We will always assume $U$ is an open subdomain of $\Rm^d$ whose boundary $\partial U$ satisfies the following interior ball condition:

\begin{cond}\label{cond:ball}
The boundary $\partial U$ satisfies the uniform interior ball condition: there is a fixed radius $r>0$ such that for every $x\in U$ there exists a point $y\in U$ such that $x\in B(y,r)\subseteq U$.
\end{cond}

Regarding the drift $b$, we will always assume that
$(\mu,x) \mapsto b(\mu,x)$ is measurable with respect to the Borel sigma algebra on $\mathcal{P}(U)\times U$ and uniformly bounded with $\lvert b\rvert\leq B<\infty$. For some results, we will also assume the following condition:

\begin{cond} \label{cond:blip}
The boundary $\partial U$ is $C^\infty$. Moreover, in addition to being measurable and uniformly bounded, the drift $b:\mathcal{P}_{\Wah}(U)\times U\ra \Rm^d$ is jointly continuous, and is Lipschitz in the first variable: there is $C > 0$ such that:
\begin{equation}
|b(\mu_1,x) - b(\mu_2,x)| \leq C \Wah(\mu_1,\mu_2), \quad \forall x \in U, \quad \mu_1, \mu_2 \in \mathcal{P}_{\Wah}(U).
\label{eq:b Lipschitz with respect to Wasserstein}
\end{equation}

\end{cond}

\begin{rmk}
The Lipschitz assumption \eqref{eq:b Lipschitz with respect to Wasserstein} may be replaced with the strictly weaker assumption that $b$ is uniformly Lipschitz with respect to the total variation metric. This does not require changes to the proof, however for simplicity we assume $b$ is uniformly Lipschitz with respect to the $\Wah$ metric.

Moreover the Lipschitz condition \eqref{eq:b Lipschitz with respect to Wasserstein} is used only to establish uniqueness in law of global weak solutions to \eqref{eq:MKV sde} for given initial conditions; for all our results this Lipschitz condition may be replaced with any other condition providing for uniqueness in law of global weak solutions to \eqref{eq:MKV sde}.

Furthermore in Proposition \ref{prop:Properties of the McKean-Vlasov Process} and theorems \ref{theo:Hydrodynamic Limit Theorem} and \ref{theo:N ra infty theorem} we could without changes to the proofs assume $b$ to be time-inhomogeneous; so that $b:[0,\infty)\times \mathcal{P}_{\Wah}(U)\times U\ra \Rm^d$ is measurable, and for Lebesgue-almost every $t$ and uniform $C<\infty$, $(m,x)\mapsto b(t,m,x)$ is jointly continuous and satisfies \eqref{eq:b Lipschitz with respect to Wasserstein}.
\label{rmk:uniqueness remark}
\end{rmk}

We firstly establish the particle system is defined over an infinite time horizon:
\begin{theo}[Global Well-Posedness of the $N$-Particle System \eqref{eq:N-particle system sde}]\label{theo:Global Well-Posedness of the N-Particle System} There exists a weak solution of \eqref{eq:N-particle system sde} for which $\Pm(\tau_{\WD}= \infty) = 1$; any weak solution of \eqref{eq:N-particle system sde} satisfies $\Pm(\tau_{\WD}=\infty)$ and weak solutions of \eqref{eq:N-particle system sde} are unique in law.
\end{theo}

We now address the large time properties of the system for fixed finite $N$. We must impose the additional assumption that the domain $U$ is bounded and path-connected. The boundedness assumption is needed as we do not currently have a good way of preventing the mass "escaping to infinity" over an infinite time horizon when the domain is unbounded. We establish ergodicity of the system for fixed $N$:

\begin{theo}[Ergodicity of the $N$-Particle System \eqref{eq:N-particle system sde}] In addition to Condition \ref{cond:ball}, assume that $U$ is path-connected and bounded. Then we have that for every $N$ fixed, there exists a unique stationary distribution $\psi^N\in \mathcal{P}(U^N)$ of the $N$-particle system (Definition \ref{defin:Fleming-Viot MKV dynamics}). Moreover there exist constants $C_N,\lambda_N>0$ such that for every initial distribution $X_0\sim\upsilon^N$ we have $\lvert\lvert \mathcal{L}(\vec{X}^N_t)- \psi^N\rvert\rvert_{TV}\leq C_Ne^{-\lambda_Nt}$.
\label{theo:Ergodicity of the N-Particle System}
\end{theo}

We now turn to the question of extracting a hydrodynamic limit. We no longer need to impose the assumption that the domain $U$ is bounded or path-connected. Our hydrodynamic limit will be given by the flow of conditional laws \eqref{eq:flow of conditional laws} corresponding to solutions of \eqref{eq:MKV sde}, and so before stating our hydrodynamic limit theorem we firstly give the properties of \eqref{eq:MKV sde}. We recall that where a weak solution $(X,\tau,W)$ to \eqref{eq:MKV sde} satisfies $\Pm(\tau>t)>0$ for all $t\in [0,\infty)$, we say it is a global weak solution.

\begin{prop}[Properties of the McKean-Vlasov Process \eqref{eq:MKV sde}] Assume Condition \ref{cond:blip}. For every $\nu\in\mathcal{P}(U)$ there exists a unique in law weak solution $(X,\tau,W)$ to \eqref{eq:MKV sde} with initial condition $X_0\sim\nu$. Moreover this weak solution to \eqref{eq:MKV sde} is a global weak solution and satisfies:
\begin{itemize}
\item[(i)] $\Pm(\tau>t)$ is continuous and positive on $[0,\infty)$.
\item[(ii)] $\mathcal{L}(X_t\lvert \tau>t) \in C([0,\infty);\mathcal{P}_{\Wah}(U)) \cap C((0,\infty);L^1(U))$.
\end{itemize}

\label{prop:Properties of the McKean-Vlasov Process}
\end{prop}

We let $(\vec{X}^N_t:0\leq t<\infty)=((X^{N,1}_t,\ldots,X^{N,N}_t):0\leq t<\infty)$ be a sequence of weak solutions to \eqref{eq:N-particle system sde} with initial conditions $\vec{X}^N_0\sim\upsilon^N$. We define $m^N_t$, $\xi^N$ and $J^N_t$ as in \eqref{eq:defin mN}, \eqref{eq:initial dist mN0} and \eqref{eq:defin JN}. We have the following hydrodynamic limit theorem:
\begin{theo}[Hydrodynamic Limit Theorem]
Assume Condition \ref{cond:blip}. Let $\nu \in \mathcal{P}(U)$ and assume that $\Wah(m^N_0,\nu)\ra 0$ in probability as $N\ra\infty$. Let $(X,\tau,W)$ be a (unique in law) global weak solution to \eqref{eq:MKV sde} with initial distribution $X_0\sim \nu$, and define as in \eqref{eq:mdef} and \eqref{eq:Jdef}:
\[
m_t = \mathcal{L}(X_t\lvert \tau>t),\quad
J_t = -\ln \Pm (\tau > t).
\]
Then, as $N \to \infty$, we have uniform convergence on compacts in probability:
\[
\begin{split}
(m_t^N,J_t^N)_{0\leq t<\infty}\ra (m_t,J_t)_{0\leq t<\infty}\quad\text{in }d^{\infty}\text{ in probability.}
\end{split}
\]
\label{theo:Hydrodynamic Limit Theorem}
\end{theo}

The existence part of Proposition \ref{prop:Properties of the McKean-Vlasov Process} and theorems \ref{theo:Hydrodynamic Limit Theorem}, \ref{theo:main t theo conv of stat dists} and \ref{theo:main t theo jt conv} are essentially all corollaries of the following generalised hydrodynamic limit theorem - Theorem \ref{theo:N ra infty theorem}. Relying on the machinery of sections \ref{section:N-particle Estimates}, \ref{section:density Estimates} and \ref{section:coupling}, this theorem will be proven in Section \ref{Section:Hydrodynamic Limit Theorem}. 

This hydrodynamic limit theorem is valid when the initial conditions are only known to constitute a tight family of random measures (as opposed to convergent weakly in probability to a deterministic initial profile as in Theorem \ref{theo:Hydrodynamic Limit Theorem}). We define:
\begin{equation}
\begin{split}
\Xi=\{(\mathcal{L}(X_t\lvert \tau>t), -\ln\Pm(\tau>t))_{0\leq t<\infty}\in C([0,\infty);\mathcal{P}_{\Wah}(U)\times \Rm_{\geq 0}):\\
(X,\tau,W)\;\text{is a }\text{global weak solution of }\eqref{eq:MKV sde}\}.
\end{split}
\label{eq:xiinfty}
\end{equation}

For $T<\infty$ we define $d^D_{[0,T]}$ to be the Skorokhod metric on $D([0,T];\mathcal{P}_{\Wah}(U)\times \Rm_{\geq 0})$. We then define the following metric:
\begin{equation}
\begin{split}
d^D:D([0,\infty);\mathcal{P}_{\Wah}(U)\times \Rm_{\geq 0})\times D([0,\infty);\mathcal{P}_{\Wah}(U)\times \Rm_{\geq 0})\ra \Rm_{\geq 0}\\
d^D(f,g)=\sum_{T=1}^{\infty}2^{-T}(d^D_{[0,T]}((f_t)_{0\leq t\leq T},(g_t)_{0\leq t\leq T})\wedge 1).
\end{split}
\label{eq:Skorokhod metric on integer times}
\end{equation}

Note that convergence with respect to $d^D$ to a continuous function implies convergence with respect to $d^{\infty}$ to the same continuous function. 

\begin{theo}
\sloppy Assume Condition \ref{cond:blip} and that $\{\xi^N\}$ is a tight family of measures in $\mathcal{P}(\mathcal{P}_{\Wah}(U))$. Then the laws of the processes $\{(m^N_t,J^N_t)_{0\leq t\leq T}\}$ are a tight family of measures on $(D([0,\infty);\mathcal{P}_{\Wah}(U)\times \Rm_{\geq 0}),d^{D})$. Moreover if along some subsequence $(m^N_t,J^N_t)_{0\leq t<\infty}\overset{d}{\ra}(m_t,J_t)_{0\leq t<\infty}$, then 
\[
(m_t,J_t)_{0\leq t<\infty}\in \Xi\cap C((0,\infty);L^1(U)\times \Rm_{\geq 0})\subseteq C([0,\infty);\mathcal{P}_{\Wah}(U)\times \Rm_{\geq 0})
\]
holds almost surely.
\label{theo:N ra infty theorem}
\end{theo}

Proposition \ref{prop:Properties of the McKean-Vlasov Process} guarantees for us that the semigroup $G_t$ on $\mathcal{P}_{\Wah}(U)$ given in \eqref{eq:MKV SDE semigroup} is well-defined. We will establish in Section \ref{section:Further Estimates} the following properties of the semigroup $G_t$:
\begin{prop}[Properties of the Semigroup $G_t$]
Assume Condition \ref{cond:blip}. Then the semigroup $G_t$ is jointly continuous in $[0,\infty)\times \mathcal{P}_{\Wah}(U)$:
\begin{equation}
[0,\infty)\times \mathcal{P}_{\Wah}(U)\ni (t,\nu)\mapsto G_t(\nu)\in \mathcal{P}_{\Wah}(U)\quad\text{is continuous.}
\label{eq:joint continuity of semigroup}
\end{equation}
Furthermore if the domain $U$ is bounded, then for all $t_0>0$, $G_{t_0}$ has pre-compact image $\text{Image}(G_{t_0})\subset\subset \mathcal{P}_{\Wah}(U)$.
\label{prop:basic properties of Gt}
\end{prop}

Having established ergodicity for fixed $N$ and hydrodynamic convergence to the flow of conditional laws \eqref{eq:flow of conditional laws} for the system \eqref{eq:MKV sde}, we ask when the limit $\lim_{t\ra\infty}G_t(\nu)$ exists. We henceforth assume the domain is bounded. The following represents the most significant issue left to resolve from this paper:

\begin{prob}[Convergence to Quasi-Equilibrium]
Under what conditions does 
\begin{equation}
    \text{the limit }\lim_{t\ra\infty}G_t(\nu)\text{ exist with respect to }\Wah\text{ for every $\nu\in\mathcal{P}_{\Wah}(U)$}
    \label{eq:convergence to non-unique quasi-equil}
\end{equation}
(with the limit possibly depending on $\nu\in\mathcal{P}_{\Wah}(U)$)? Can we find conditions under which there exists $\pi\in \mathcal{P}_{\Wah}(U)$ such that
\begin{equation}
G_t(\nu)\ra \pi\text{ uniformly in }\Wah\text{ as } t\ra\infty ?
\label{eq:uniform convergence to unique equilibria}
\end{equation}
\label{prob:conv to equil}
\end{prob}
Although we are unable to resolve Problem \ref{prob:conv to equil}, we shall establish the following:
\begin{prop}[Existence and Properties of QSDs]
In addition to conditions \ref{cond:ball} and \ref{cond:blip}, we assume that $U$ is bounded. Then we have the following:
\begin{enumerate}
\item 
The following are equivalent:
\begin{enumerate}
\item \label{enum:QSD}
$\pi$ is a QSD for \eqref{eq:MKV sde}.

\item \label{enum:QLD}
$\pi$ is a QLD for \eqref{eq:MKV sde}.

\item \label{enum:PDE for QSD}
There is some $\lambda$ such that $(\pi,\lambda)\in L^1(U)\times (0,\infty)$ is a solution of:
\begin{equation}
\langle \pi(.),\lambda \varphi(.)+b(\pi,.)\cdot \nabla \varphi(.) +\frac{1}{2}\Delta \varphi(.)\rangle=0, \quad\varphi\in C_0^{\infty}(\bar U)
\label{eq:equation for det stat solns PDE}
\end{equation}
whereby we define the test functions:
\begin{equation}
C_0^{\infty}(\bar U)=\{\varphi\in C^{\infty}_c(\bar U):\varphi=0\text{ on }\partial U\}.
\label{eq:test fns no time variable}
\end{equation}
\end{enumerate}
\label{enum:QSD QLD equivalent}
\item \label{enum:tau exp distributed}
For any weak solution $(X,\tau,W)$ to \eqref{eq:MKV sde} with quasi-stationary initial condition $\pi$ we have $\tau\sim \text{Exp}(\lambda)$ where $\lambda$ is the constant such that $(\pi,\lambda)$ is a solution to \eqref{eq:equation for det stat solns PDE}. 
\item
$\Pi$ is a non-empty compact subset of $\mathcal{P}_{\Wah}(U)$.
\label{enum:Pi compact}
\end{enumerate}
\label{prop:properties of QSD}
\end{prop}

\begin{rmk}
The equation \eqref{eq:equation for det stat solns PDE} is the weak formulation of the following nonlinear PDE:
\begin{equation}
\lambda \pi-\nabla\cdot(b(\pi,.)\pi)+\frac{1}{2}\Delta\pi=0,\quad \pi=0 \text{ on }\partial U.
\label{eq:strong fomulation of QSD PDE}
\end{equation}

\end{rmk}

Whereas $\Pi$ must be a compact set (in $\mathcal{P}_{\Wah}(U)$), we recall that Example \ref{exam:non-uniqueness in bounded MKV case} demonstrates it need not be a singleton. We now show that, if we don't have \eqref{eq:convergence to non-unique quasi-equil}, then the stationary distributions for our $N$-particle system (given by Theorem \ref{theo:Ergodicity of the N-Particle System}) converge to the set of QSDs $\Pi$:

\begin{theo}[Convergence of the $N$-Particle Stationary Distributions to QSDs] 
In addition to conditions \ref{cond:ball} and \ref{cond:blip}, we assume that $U$ is bounded. Moreover we assume that
\begin{equation}\tag{\ref{eq:convergence to non-unique quasi-equil}}
    \text{the limit }\lim_{t\ra\infty}G_t(\nu)\text{ exists with respect to }\Wah\text{ for every }\nu\in\mathcal{P}_{\Wah}(U).
\end{equation}
We take the stationary empirical measures $\Psi^N=\vartheta^N_{\#}\psi^N$ and take a sequence of $\mathcal{P}_{\Wah}(U)$-valued random variables $\pi^N$ with distribution $\pi^N\sim \Psi^N$ as in \eqref{eq:defin piN}. Then we have:
\begin{equation}
\Wah(\pi^N,\Pi)\ra 0\quad \text{in probability as}\quad N\ra\infty.
\end{equation}
\label{theo:main t theo conv of stat dists}
\end{theo}

Since we do not necessarily have \eqref{eq:convergence to non-unique quasi-equil}, it is worthwhile asking what happens when we don't have \eqref{eq:convergence to non-unique quasi-equil}. In general we shall see that we obtain the invariant measures for $G_t$:
\begin{equation}
    M_G=\{\Pm\in \mathcal{P}(\mathcal{P}_{\Wah}(U)):(G_t)_{\#}\Pm= \Pm\quad \text{for all}\quad t\geq 0\}.
\end{equation}

Then by propositions \ref{prop:properties of QSD} and \ref{prop:basic properties of Gt} $M_G$ is a non-empty compact subset of $\mathcal{P}(\mathcal{P}_{\Wah}(U))$. Furthermore it is clear from the proof of Theorem \ref{theo:main t theo conv of stat dists} that under the same assumptions as Theorem \ref{theo:main t theo conv of stat dists}, except for \eqref{eq:convergence to non-unique quasi-equil intro}, we have:
\begin{equation}
W(\mathcal{L}(\pi^N),M_G)\ra 0\quad\text{as}\quad N\ra\infty.
\end{equation}
Therefore the Fleming-Viot particle system allows us to obtain at least one of the invariant measures of $G_t$.

Finally, under an additional assumption on the semigroup $G_t$, we establish convergence as the number of particles and the time horizon converge to infinity together. We prove the following theorem:
\begin{theo}
In addition to conditions \ref{cond:ball} and \ref{cond:blip}, we assume that $U$ is bounded. Moreover we assume that there exists a QSD $\pi$ such that:
\begin{equation}\tag{\ref{eq:uniform convergence to unique equilibria}}
    \Wah(G_t(\nu),\pi)\ra 0\quad\text{as}\quad t\ra\infty\quad\text{uniformly in}\quad \nu\in\mathcal{P}_{\Wah}(U).
\end{equation} 
Then by Proposition \ref{prop:properties of QSD} there exists $\lambda>0$ such that $(\pi,\lambda)$ is a solution of \eqref{eq:equation for det stat solns PDE}. We take a sequence of weak solutions $(\vec{X}^N_t:0\leq t<\infty)=((X^{N,1}_t,\ldots,X^{N,N}_t):0\leq t<\infty)$ to \eqref{eq:N-particle system sde} with arbitrary initial conditions $\vec{X}^N_0 \sim \upsilon^N$. We define $m^N_t$ and $J^N_t$ as in \eqref{eq:defin mN} and \eqref{eq:defin JN}:
\[
m^N_t=\vartheta^N(\vec{X}^N_t),\quad J^N_t = \frac{1}{N} \sup \{ k \in \mathbb{N}\;|\; \tau_k \leq t \}.
\]

Then we have:
\begin{equation}
\begin{split}
(m_{t_0+t}^N,J_{t_0+t}^N-J_{t_0}^N)_{0\leq t<\infty}\ra (\pi,\lambda t)_{0\leq t<\infty}\text{ in }d^{\infty}
\text{ in probability}\text{ as }t_0\wedge N\ra\infty.
\end{split}
\label{eq:convergence in Skorokhod as N,T goes to infty when locally unif}
\end{equation}
\label{theo:main t theo jt conv}
\end{theo}

\section{Proof Strategy for Sections \ref{section:N-particle Estimates} and \ref{section:density Estimates}-\ref{Section:Hydrodynamic Limit Theorem}}\label{section:Proof Strategy}

The results of sections \ref{section:N-particle Estimates}, \ref{section:density Estimates} and \ref{section:coupling} shall be used in Section \ref{Section:Hydrodynamic Limit Theorem} to establish our hydrodynamic limit theorem, as we shall explain here. 

The proof of our hydrodynamic limit theorem shall require defining the following Fleming-Viot particle systems with generalised dynamics - therefore we establish the results of Section \ref{section:N-particle Estimates} and \ref{section:density Estimates} for such generalised systems. Here the drift of particle $X^i$ is assumed only to be some $\mathcal{F}$-adapted process $b^i_t$:
\begin{equation}
dX^i_t=b^i_tdt+dW^i_t.
\end{equation}
Whenever we consider Fleming-Viot particle systems with generalised dynamics, we shall assume the domain $U$ is an open subdomain of $\Rm^d$ satisfying Condition \ref{cond:ball} and the drifts satisfy:
\begin{cond}
The drifts $b^i_t$ are $(\mathcal{F}_t)_{t\geq 0}$-adapted and uniformly bounded $\lvert b^i_t\rvert\leq B$ ($i=1,\ldots,N$).
\label{cond:assum generalised drifts}
\end{cond}

Otherwise, the Fleming-Viot particle system with generalised dynamics has the same definition as the particle system with McKean-Vlasov dynamics.

We establish in Section \ref{section:N-particle Estimates} estimates on the $N$-particle system which shall be used throughout this paper along with global well-posedness of the $N$-particle system with generalised dynamics (and hence for the system with McKean-Vlasov dynamics - Theorem \ref{theo:Global Well-Posedness of the N-Particle System}). These estimates, in particular, will allow us control the mass close to the boundary, uniformly in N. This will be an essential ingredient in our proof of hydrodynamic convergence in Section \ref{Section:Hydrodynamic Limit Theorem}.

The estimates of this section hinge on constructing - in a completely different manner - a family of Bessel processes similar to those constructed by Burdzy, Holyst and March \cite[Proof of Theorem 1.4]{Burdzy2000} to deal with the $b= 0$ case. These are $N$ i.i.d. Bessel processes coupled to the $N$-particle system in such a way so as to provide controls on the mass close to the boundary. The major difference between the Bessel processes in \cite{Burdzy2000} and those in this paper is the method of construction. In \cite{Burdzy2000} their construction begins by taking the Bessel processes and then using a classical skew-product decomposition \cite{Ito1974} to construct the particle system with $b\equiv 0$. This has no hope of working however in the $b\neq 0$ case as such a skew-product decomposition is not available. In contrast, we instead start with the particle system and from there construct the Bessel processes. We instead use a Doob-Meyer decomposition piecewise between a family of stopping times to construct an associated Brownian motion for each particle, and then use these Brownian motions to drive our Bessel processes.

In \cite{Burdzy2000,Grigorescu2007a, Bieniek2009}, conditions for the global well-posedness ($\Pm(\tau_{\WD}= +\infty) = 1$) of this system were established for the case $b \equiv 0$ when $U$ is bounded (and satisfies various additional conditions). These are complemented by \cite{Grigorescu2012,Villemonais2011}, providing well-posedness for general diffusions on possibly unbounded domains (satisfying various additional conditions). For such domains, one could then obtain the global well-posedness for the system with generalised dynamics from the $b\equiv 0$ case using Girsanov's theorem - they can be related via a Girsanov transform, which preserves $\{\tau_{\WD}<\infty\}$ as a null event. None of these, however, apply to general unbounded domains with $C^{\infty}$ boundary satisfying the uniform interior ball condition. Nevertheless, the Bessel processes we construct allow us to establish well-posedness for the system with generalised dynamics and possibly unbounded domains satisfying only the uniform interior ball condition.

We shall prove lemmas \ref{lem:char-almost sure abs cty} and \ref{lem:tight after bounded times} in Section \ref{section:density Estimates}. Lemma \ref{lem:char-almost sure abs cty} will be crucial in our proof of hydrodynamic convergence as it will make available to us a uniqueness theorem for the linear Fokker-Planck equation \cite[Theorem 1.1]{Porretta2015a}. It guarantees that subsequential limits of the empirical measure valued process almost surely has a density. The proof of Lemma \ref{lem:char-almost sure abs cty} hinges on an analysis of the dynamical historical processes introduced by Bieniek and Burdzy \cite{Bieniek2018}. The machinery we construct to prove Lemma \ref{lem:char-almost sure abs cty} then enables us to prove Lemma \ref{lem:tight after bounded times}, which constrains the number of particles far away from the boundary over fixed time horizons.

We then prove Proposition \ref{prop:coupling to system on bounded domain} in Section \ref{section:coupling}, establishing that we may couple the $N$-particle system on an infinite domain with an appropriately constructed Fleming-Viot $N$-particle system with generalised dynamics on a large but finite subdomain. Moreover we obtain uniform controls on the difference between the two $N$-particle systems. This coupled particle system having generalised dynamics is the reason we established the previous estimates of sections \ref{section:N-particle Estimates} and \ref{section:density Estimates} for such generalised systems. As we will explain in the proof of Theorem \ref{theo:N ra infty theorem}, this will allow us to circumvent the problem that the uniqueness theorem we use \cite[Theorem 1.1]{Porretta2015a} for the linear Fokker-Planck equation only applies on bounded domains.

Having established these estimates, we are in a position to prove Proposition \ref{prop:Properties of the McKean-Vlasov Process} and Theorem \ref{theo:Hydrodynamic Limit Theorem} by way of Theorem \ref{theo:N ra infty theorem}. Theorem \ref{theo:N ra infty theorem} characterises subsequential limits of the $N$-particle system as corresponding to solutions of the McKean-Vlasov SDE \eqref{eq:MKV sde} - but this doesn't assume the existence of such solutions. Therefore by choosing a sequence of $N$-particle systems with the appropriate initial conditions we are able to construct a weak solution to \eqref{eq:MKV sde} in the $N\ra\infty$ limit. We establish uniqueness of weak solutions to \eqref{eq:MKV sde} by a contraction argument using Girsanov's theorem similar to the proof of \cite[Proposition C.1]{Campi2017}, completing the proof of Proposition \ref{prop:Properties of the McKean-Vlasov Process}. Theorem \ref{theo:Hydrodynamic Limit Theorem} then follows by a compactness-uniqueness argument.

The estimates of Section \ref{section:N-particle Estimates} and Lemma \ref{lem:tight after bounded times} are used to establish tightness in the proof of Theorem \ref{theo:N ra infty theorem}; the former preventing mass from accumulating on the boundary and the latter preventing mass "escaping to infinity" over a finite time horizon.

We then employ martingale methods to characterise subsequential limits $(m_t,J_t)_{0\leq t<\infty}$ as being supported on the solution set of a nonlinear Fokker-Planck equation. We note that martingale methods have also been used to establish hydrodynamic convergence in the Markovian case (\cite{Grigorescu2004} and \cite{Villemonais2011}). We then show that these nonlinear Fokker-Planck solutions correspond to global weak solutions of our McKean-Vlasov SDE \eqref{eq:MKV sde} by verifying they satisfy the same linear Fokker-Planck equation and using a uniqueness theorem \cite[Theorem 1.1]{Porretta2015a}. Availing ourselves of this uniqueness theorem requires Lemma \ref{lem:char-almost sure abs cty} and - in the case of unbounded domains - combining Proposition \ref{prop:coupling to system on bounded domain} with a change to our notion of solution to the nonlinear Fokker-Planck equation. 

We note this is where the assumption that $U$ has $C^{\infty}$ boundary becomes necessary, as \cite[Theorem 1.1]{Porretta2015a} assumes the domain has $C^{\infty}$ boundary. Were a more general uniqueness theorem available, this would enable a corresponding generalisation of our results: to more general boundaries, the particles having non-constant diffusivities or the incorporation of "soft killing" (killing according to a Poisson clock).

\section{Well-Posedness of and Estimates for the $N$-Particle System}\label{section:N-particle Estimates}
The goal of this section is to establish Theorem \ref{theo:Global Well-Posedness of the N-Particle System} along with some estimates for the $N$-Particle System. We shall prove estimates on the jump times $\{\tau_k\}_{k=0}^\infty$ and on the empirical measure $m^N_t$ of the $N$-particle process. The estimates in particular will prevent mass accumulating on the boundary when we take various limits in later sections. Theorem \ref{theo:Global Well-Posedness of the N-Particle System} will be seen to be a consequence of these estimates.

As discussed in Section \ref{section:Proof Strategy}, we establish well-posedness and our estimates for Fleming-Viot particle systems with generalised dynamics. Throughout this section,
\[
((\vec{X}_t,\vec{W}_t,\vec{b}_t)_{0\leq t<\tau_{\WD}}=((X^1,\ldots,X^N_t),(W^1,\ldots,W^N_t),(b^1,\ldots,b^N_t))_{0\leq t<\tau_{\WD}}
\]
will refer to a weak solution to the Fleming-Viot particle system ($N \geq 2$) with generalised dynamics and drift processes bounded by $|b^i_t| \leq B$. We further define $m^N_t$ and $\xi^N$ as in \eqref{eq:defin mN} and \eqref{eq:initial dist mN0}:
\[
m^N_t=\vartheta^N(\vec{X}^N_t),\quad m^N_0\sim \xi^N.
\]

We will couple the particles $X^1,\ldots,X^N$ to appropriately constructed independent strong solutions $(\eta^1,\tilde W^1),\ldots,(\eta^N,\tilde{W}^N)$ of the following SDE:
\begin{equation}
\begin{split}
d\eta_t=
\begin{cases}
d\tilde W_t+Bdt+\frac{d-1}{2\eta_t}dt-dL^{r-\eta}_t,\quad d>1\\
d\tilde W_t+Bdt+dL^{\eta}_t-dL^{r-\eta}_t,\quad d=1
\end{cases}
,\quad
\eta_0=r
\end{split}
\label{eq:SDE for Bessel processes no i}
\end{equation}
where $r > 0$ is the constant from the global interior ball condition, Condition \ref{cond:ball}. Here $L^{\eta}$ and $L^{r-\eta}$ are the local times:
\begin{equation}
    L^{\eta}_t:=\lim_{\epsilon\downarrow 0}\frac{1}{2\epsilon}\int_0^t\Ind(\lvert \eta_s\rvert<\epsilon)d[\eta]_s,\quad     L^{r-\eta}_t:=\lim_{\epsilon\downarrow 0}\frac{1}{2\epsilon}\int_0^t\Ind(\lvert r-\eta_s\rvert<\epsilon)d[\eta]_s.
    \label{eq:defin of local times}
\end{equation}

We will then use this coupling to obtain estimates on the $N$-particle system.

\begin{prop}
\sloppy Assume the Brownian motions $W^i$ are jointly independent and defined up to time $\infty$. There exists on the same probability space a family $(\eta^1_t,\tilde{W}^1_t)_{0\leq t<\infty},\ldots,(\eta^N_t,\tilde{W}^N_t)_{0\leq t<\infty}$ of strong solutions to \eqref{eq:SDE for Bessel processes no i} which are jointly independent, but coupled to $X^1,\ldots,X^N$ up to time $\tau_{WD} = \tau_\infty\wedge\tau_\ST\wedge\tau_{\max}$ so that:
\begin{equation}
d(X_t^i,\partial U) \geq r-\eta^i_t\in [0,r] , \quad 0\leq t < \tau_{WD}.
\label{eq:lower bd dist to du time tauinf taustop}
\end{equation}
\label{prop:dominated by Bessel fns up to time tauinfty wedge taustop}
\end{prop}

\begin{rmk}
The coupling \eqref{eq:lower bd dist to du time tauinf taustop} only holds up to time $\tau_{WD}$, although $(\eta^{i}_t,\tilde{W}^i_t)$ are defined for all $t \geq 0$.
\end{rmk}

We then establish the following:
\begin{lem}
If $(\eta^1,\tilde{W}^1)$, $(\eta^2,\tilde{W}^2)$ are two independent solutions to \eqref{eq:SDE for Bessel processes no i} on the same probability space, then
\begin{equation}
\Pm(\exists\; t>0\text{ such that }\eta_t^1=\eta_t^2=r)=0.
\end{equation}
\label{lem:no simultaneous hitting for Bessel processes}
\end{lem}

For the case of Brownian dynamics ($b \equiv 0$) with bounded domain $U$, the authors of \cite{Burdzy2000} established controls analogous to Proposition \ref{prop:dominated by Bessel fns up to time tauinfty wedge taustop} with $b=0$ and $U$ bounded. The method of construction they used, however, was quite different. As outlined in Section \ref{section:Proof Strategy}, their approach does not work in our case. 

\begin{prop}
For any weak solution to the Fleming-Viot particle system with generalised dynamics, $\tau_{WD} = \tau_\infty=\tau_{\ST}=\tau_{\max}=\infty$ almost surely. In particular, the coupling defined in Proposition \ref{prop:dominated by Bessel fns up to time tauinfty wedge taustop} holds for all $t \geq 0$.
\label{prop:general weak solns are global}
\end{prop}

Having established $\tau_{\WD}=\infty$ almost surely in the case of generalised dynamics, we have $\tau_{\WD}=\infty$ in the case of McKean-Vlasov dynamics, giving the proof of Theorem \ref{theo:Global Well-Posedness of the N-Particle System}:
\begin{proof}[Proof of Theorem \ref{theo:Global Well-Posedness of the N-Particle System}]
It is clearly possible to construct a weak solution of the driftless system up to time $\tau_{\WD}$, so that between jump times and for $t<\tau_{\WD}$ particle $X^i$ satsifies $dX^i_t=dW^i_t$. Therefore by Girsanov's theorem we obtain the existence of a weak solution $(\vec{X}_t,\vec{W}_t)_{0\leq t<\tau_{\WD}}$ to the $N$-particle system with McKean-Vlasov dynamics \eqref{eq:N-particle system sde} up to time $\tau_{\WD}$. This and every other weak solution to \eqref{eq:N-particle system sde} defined up to time $\tau_{\WD}$ is defined for all time with $\tau_{\WD}=\infty$ almost surely by Proposition \ref{prop:general weak solns are global}.

Uniqueness of the law of $(\vec{X}_t)_{0\leq t<\infty}$ follows from uniqueness for the driftless system, by change of measure (by the same argument that weak solutions to SDEs with bounded measurable coefficients are unique in law; see \cite[Proposition 3.10]{Karatzas1991}).
\end{proof}

We shall then establish tightness for the laws of the empirical measure valued process at times bounded away from 0, when the domain $U$ is bounded:
\begin{prop}\label{prop:laws at positive times constrained to compact set}
We assume $U$ is bounded. For any $T_0>0$ there exists a compact set $\mathcal{K}_{T_0}\subseteq \mathcal{P}(\mathcal{P}_{\Wah}(U))$ dependent only upon the upper bound on the drift B and the domain $U$ such that the empirical measure $m^N_t:=\vartheta^N(\vec{X}^N_t)$ must satisfy $\mathcal{L}(m_t^N)\in \mathcal{K}_{T_0}$ for all $t\geq T_0$.
\end{prop}

We henceforth fix a finite time horizon $T<\infty$, but no longer assume $U$ is bounded. We establish the following proposition:
\begin{prop} Define for $c>0$ the closed set $V_{c}=\{x\in U:d(x,\partial U)\geq c\}$. Then we have:
\begin{enumerate}
\item
For every $\epsilon>0$, $T_0>0$ there exists $c=c(\epsilon,T_0)$ dependent only upon the upper bound for the drift, the constant of the interior ball condition $r>0$, $\epsilon>0$ and $T\geq T_0>0$ such that $K_{\epsilon,T_0}=V_{c(\epsilon,T_0)}\subseteq U$ must satisfy:
\begin{equation}\label{eq:bound on mass near bdy after time T0}
\lim_{N \to \infty} \Pm\left(\sup_{t \in [T_0,T]} m_t^N(K_{\epsilon, T_0}^c )\geq \epsilon \right) = 0.
\end{equation}
\label{enum:bound on mass near bdy after small time}
\item
We now assume $\xi^N:=\vartheta^N_{\#}\upsilon^N$ is tight in $\mathcal{P}(\mathcal{P}_{\Wah}(U))$ (i.e. as a tight family of random measures on the open set $U$) - so that mass does not concentrate on the boundary. Fix $\epsilon,\delta>0$. Then there exists $\tilde{c}(\epsilon,\delta)>0$ such that $\hat K_{\epsilon,\delta}=V_{\tilde{c}(\epsilon,\delta)} \subseteq U$ satisfies:
\begin{equation}\label{eq:bound on mass near bdy}
\limsup_{N\ra\infty}\Pm \left(\sup_{t \in [0,T]} m^N_t(\hat K_{\epsilon,\delta}^c)\geq \epsilon  \right)<\delta.
\end{equation}
\label{enum:bound on mass near bdy}
\end{enumerate}
\label{prop:bound on mass near bdy}
\end{prop}
\begin{rmk}
In Part \ref{enum:bound on mass near bdy after small time} of Proposition \ref{prop:bound on mass near bdy}, we do not assume that the initial random measures $\xi^N:=\vartheta^N_{\#}\upsilon^N$ are tight as random measures on $U$. In particular we may have $\xi^N$ converging weakly in probability to an atom on $\partial U$ or the mass could escape to infinity.
\end{rmk}

\begin{rmk}
There are two conventions as to the definition of a Geometric Random Variable. Throughout we use the definition in which the distribution is supported on $\{1,2,\ldots\}$, with distribution given by:
\[
\Pm(G \geq k) = (1-p)^{k-1}.
\]
\label{rmk:convention for Geometric distribution}
\end{rmk}

Our final estimate controls the number of jumps by any particle over a finite time horizon:
\begin{prop}
Assume that $\{\mathcal{L}(m^N_0)\}$ is tight in $\mathcal{P}(\mathcal{P}_{\Wah}(U))$. Let $J^{N,i}_t$ be the number of jumps of the i-th particle in the $N$-particle system up to time t. Then for every $\epsilon>0$, there exists a stopping time $\tau_{\epsilon}^N$ and constants $M_{\epsilon}<\infty$, $p_{\epsilon}>0$ (all dependent upon T) such that for all $N$ large enough:
\begin{enumerate}
\item
The number of jumps $J^{N,i}_{\tau_{\epsilon}^N\wedge T}$ by particle i up to time $T\wedge \tau_{\epsilon}^N$ is stochastically bounded by the sum of $M_{\epsilon}$ i.i.d. $\text{Geom}(p_{\epsilon})$ distributions.
\item
\[
\limsup_{N\ra\infty}\Pm(\tau_{\epsilon}^N\leq T)\leq \epsilon.
\]
\end{enumerate}
\label{prop:bound on number of jumps by particle}
\end{prop}

\subsection{Proof of Proposition \ref{prop:dominated by Bessel fns up to time tauinfty wedge taustop}}
The proof proceeds in the follow steps: 
\begin{enumerate}
\item
We fix for the time being $X^i_t$ (with driving Brownian motion $W^i$) and seek to construct $(\eta^i_t,\tilde{W}^i)_{0\leq t<\infty}$ satisfying:
\begin{equation}\tag{\ref{eq:SDE for Bessel processes no i}}
\begin{split}
d\eta_t=
\begin{cases}
d\tilde W_t+\frac{d-1}{2\eta_t}dt+Bdt-dL^{r-\eta}_t,\quad d>1\\
d\tilde W_t+Bdt+dL^{\eta}_t-dL^{r-\eta}_t,\quad d=1
\end{cases}
,\quad
\eta_0=r
\end{split}
\end{equation}
and some Cadlag process $D^i_t$ such that:
\begin{equation}
\eta^i_t\geq D^i_t\geq r-d(X^i_t,\partial U),\quad 0\leq t<\tau_{\WD}.
\label{eq:inequality to induct Bessel proof}
\end{equation}

For clarity, we will usually drop the superscript $i$ in what follows: $\eta_t$, $\tilde W_t$, $D_t$, $\tau_\omega$ will refer to quantities that depend on the particle index $i$. Our construction of $(\eta_t,\tilde W_t)_{0\leq t<\infty}$ proceeds as follows:
\begin{enumerate}
\item \label{enum:construction of stopping times}
We define stopping times $\tau_{(j,k,\ell)}$ for every triple $(j,k,\ell)\in \mathbb{N}_0^3$, thereby obtaining a collection of random subintervals $[\tau_{(j,k,\ell)},\tau_{(j,k,\ell+1)})$ of $[0,\tau_\infty\wedge\tau_\ST)$. We write $\omega_0$ for the order-type of the natural numbers, associate to the ordinal $\omega=j\omega_0^2+k\omega_0+\ell<\omega_0^3$ the triple $(j,k,\ell)$ and write $\tau_{\omega}$ for the stopping time $\tau_{(j,k,\ell)}$. The use of ordinals will enable us to use ordinal induction. Moreover we write $\tau_{\omega_0^3}:=\tau_{\WD}$ and write $I_{\omega}$ for the interval $[\tau_{\omega},\tau_{\omega+1})=[\tau_{(j,k,\ell)},\tau_{(j,k,\ell+1)})$ (whereby $[t,t):=\emptyset$). The following property shall be immediate from the construction:
\begin{equation}
\tau_{\omega_1}\leq \tau_{\omega_2}\text{ for }\omega_1\leq\omega_2\leq\omega_0^3.
\end{equation}
Moreover we shall establish the following lemma:
\begin{lem}
For limit ordinals $\omega\leq \omega_0^3$ we have:
\begin{equation}
\tau_{\omega'}\uparrow \tau_{\omega}\quad\text{as}\quad\omega'\uparrow\omega\quad \text{for every }\omega\leq \omega_0^3\text{ a limit ordinal.}
\label{eq:stopping times converge limit ordinals}
\end{equation}
\label{lem:ordinal stopping times prop lem}
\end{lem}
By ordinal induction the random subintervals $I_{\omega}$ form a disjoint cover of $[0,\tau_{\omega_0^3})=[0,\tau_{\WD})$. Moreover on each interval $I_{\omega}$, $X_t$ will be contained in the ball $B(v_{\omega},r)$ (where $r>0$ is the constant we assume to exist in the interior ball condition and $v_{\omega}=X_{\tau_{\omega}}$).

\item \label{enum:Brownian construction step}
We use our construction in part \ref{enum:construction of stopping times} to define:
\begin{equation}
\begin{split}
D_t=\sum_{\omega}\Ind_{I_{\omega}}(t)d(X_t,v_{\omega}),\quad 0\leq t<\tau_{\WD},\\
D^{\omega}_t=(d(X_{t\wedge \tau_{\omega+1}-},v_{\omega})-d(X_{t\wedge \tau_{\omega}},v_{\omega})\big)\Ind(\tau_\omega<\tau_{\omega+1}).
\end{split}
\label{eq:defin of Dt and Domegat}
\end{equation}

We observe that $D^{\omega}$ is a continuous semimartingale, with $dD^{\omega}_t=dD_t$ for $t\in I_{\omega}$. We employ a Doob-Meyer decomposition of $D_t^{\omega}$ on each interval $I_\omega$ to construct a Brownian motion $(\tilde W_t)_{0\leq t<\infty}$ such that $(D_t)_{0\leq t<\tau_{\WD}}$ is a $[0,r]$-valued process which satisfies:
\begin{equation}
dD_t =\begin{cases}
d\tilde W_t+Bdt+ \frac{d-1}{2D_s} dt   - dH_t,\quad d>1\\
d\tilde W_t + B dt+ dL^D_s ds - dH_t,\quad d=1
\end{cases}
\label{eq:sde for Dt}
\end{equation}
where $H_t$ is a non-decreasing, adapted process. Moreover there exists a Cadlag adapted process $n_t$ such that:
\begin{equation}
\tilde{W}_t=\int_0^tn_s\cdot dW_s,\quad 0\leq t<\infty.
\label{eq:driving BM for bessel in terms of for particle}
\end{equation}

\item

We establish that \eqref{eq:SDE for Bessel processes no i} has strong solutions for this driving motion $\tilde W_t$, and that $\eta_t = \eta_t^i$ satisfies \eqref{eq:inequality to induct Bessel proof}. 
\label{enum:exist of strong solutions}
\item \label{enum:eta dominates D}
We then compare \eqref{eq:sde for Dt} and \eqref{eq:SDE for Bessel processes no i} establish:
\begin{equation}
D_t\leq \eta_t,\quad 0\leq t<\tau_{\WD}
\label{eq:Dt at most eta t}
\end{equation}
and therefore we have \eqref{eq:inequality to induct Bessel proof}.
\end{enumerate}

\item \label{enum:Bessel construction independence}

We repeat the above construction for each $X^i$, writing $(\eta^i,\tilde{W}^i)$ for the strong solutions we construct. By examining the quadratic covariation of the Brownian motions $\tilde{W}^i$ (using \eqref{eq:driving BM for bessel in terms of for particle}) we establish the $(\eta^i,\tilde{W}^i)$ are jointly independent.

\end{enumerate}

\subsubsection*{Step \ref{enum:construction of stopping times}}

We now define functions $\rho$ and v as in \cite{Burdzy2000}. With $r>0$ being the constant assumed to exist by the interior ball condition (Condition~\ref{cond:ball}), define
\[
\rho(x)=\sup_{\substack{B(y,r)\text{ such that}\\ U\supseteq B(y,r)\ni x}}d(x,\partial B(y,r)).
\]
We claim there exists $v:U\ra U$ measurable such that:
\begin{enumerate}
\item \label{enum:Ball around v(x) in U}
$B(v(x),r)\subseteq U$ for every $x\in U$.
\item \label{enum:distance to ball around v close to max}
$d(x,\partial B(v(x),r))\geq \frac{\rho(x)}{2}$.
\end{enumerate}

The construction of v is fairly elementary. We firstly take an ascending sequence of compact sets $K_1,K_2,\ldots$ with union $U$. We fix $K_i$ and seek to define on $K_i$ a suitable function $v^i$ satisfying \ref{enum:Ball around v(x) in U} and \ref{enum:distance to ball around v close to max}. It is easy to see that for every $x\in K_i$ we can choose $y(x)$	 such that $d(x,\partial B(y,r))> \frac{\rho(x)}{2}$. Then on an open neighbourhood $V_x\ni x$ we have $d(x',\partial B(y,r))>\frac{\rho(x')}{2}$ as both $x'\mapsto d(x',\partial B(y(x),r))$ and $x'\mapsto \rho(x')$ are continuous functions. We may cover $K_i$ with open sets $V_x$, $x\in K_i$, and take a finite subcover $V_{x_1},\ldots,V_{x_n}$ (for some $n$). We now define:
\[
v^i(x'):=
\begin{cases}
x_1,\quad x'\in K_i\cap V_{x_1}\\
x_j,\quad x'\in K_i\cap (V_{x_j}\setminus (V_{x_1}\cup\ldots\cup V_{x_{j-1}})).
\end{cases}
\]
Then $v^i$ is piecewise constant (and hence measurable) and satisfies \ref{enum:Ball around v(x) in U} and \ref{enum:distance to ball around v close to max} on $K^i$. Therefore defining v as follows we are done:
\[
v(x'):=
\begin{cases}
v^1(x'),\quad x'\in K_1\\
v^i(x'),\quad x'\in K_i\setminus (K_{1}\cup\ldots\cup K_{{j-1}})).
\end{cases}
\]

We now turn to the construction of the stopping times $\tau_{(j,k,\ell)}$, for triples $(j,k,\ell)\in \mathbb{N}_0^3$. 
\begin{enumerate}
\item
$\tau_{(0,0,0)}:=0$.
\item
$\tau_{(j+1,0,0)}:=\inf\{t>\tau_{(j,0,0)}:X_{t-}\in\partial U\}\wedge\tau_{\WD}$, for all $j \in \mathbb{N}_0$. 
\item\label{enum:defin of l stopping times k=0 step}
With $j \in \mathbb{N}_0$ fixed, we now define $\tau_{(j,0,\ell)} \in [\tau_{(j,0,0)}\,, \,\tau_{(j+1,0,0)}]$ for every $\ell \in \mathbb{N}$. We proceed inductively, having already defined $\tau_{(j,0,\ell)}$ for $\ell = 0$ in the previous step. We suppose that $\tau_{(j,0,\ell)} \in [\tau_{(j,0,0)}\,, \,\tau_{(j+1,0,0)}]$ has been defined for some $\ell \in \mathbb{N}_0$. If $\tau_{(j,0,\ell)} = \tau_{(j+1,0,0)}$, we set $\tau_{(j,0,\ell+1)}:=\tau_{(j,0,\ell)}$. Otherwise, $\tau_{(j,0,\ell)} < \tau_{(j+1,0,0)}$ holds and $X_{\tau_{(j,0,\ell)}}\in U$. Therefore, we may define $X_{(j,0,\ell)}:=X_{\tau_{(j,0,\ell)}}$ and $v_{(j,0,\ell)}:=v(X_{(j,0,\ell)})$ which satisfies:
\begin{equation}
    B(v_{(j,0,\ell)},r)\subseteq U.
    \label{eq:Ball around vj0l in U}
\end{equation} 
We then define:
\[
\tau_{(j,0,\ell+1)}=
\tau_{\WD}\wedge\begin{cases}
\inf\{t>\tau_{(j,0,\ell)}:d(X_{t-},v_{(j,0,\ell)})\geq r\},\quad \text{if} \;\;\rho(X_{\tau_{(j,0,\ell))}})>2^{-0}\\
\tau_{(j,0,\ell)},\quad \text{if}\;\;\rho(X_{\tau_{(j,0,\ell)}})\leq 2^{-0}\quad\text{or}\quad\tau_{(j,0,\ell)}=\tau_{(j+1,0,0)}
\end{cases}
\]
which satisfies $\tau_{(j,0,\ell+1)}\leq \tau_{(j+1,0,0)}$ by \eqref{eq:Ball around vj0l in U}. By induction on $\ell$, this defines $\tau_{(j,0,\ell)}$ for all $\ell \in \mathbb{N}_0$ and we have $\tau_{(j,0,0)} \leq \tau_{(j,0,\ell)}\leq \tau_{(j,0,\ell+1)}\leq \ldots\leq \tau_{(j+1,0,0)}$.

\item\label{enum:defin of k=1 stopping time step}
We then establish (Lemma \ref{lem:fin number of jumps fixed j,k}) that either $\tau_{(j+1,0,0)}=\infty$ and $\tau_{(j,0,\ell)}\uparrow \tau_{(j+1,0,0)}$ as $\ell\uparrow \infty$, or else $\tau_{(j+1,0,0)}<\infty$ and there exists some random $\ell_{(j,0)}<\infty$ such that either $\rho(X_{(j,0,\ell_{(j,0)})})\leq 2^{-0}$ or $\tau_{(j,0,\ell(j,0))}=\tau_{(j+1,0,0)}$. In the former case ($\tau_{(j+1,0,0)}=\infty$) we define:
\[
\tau_{(j,1,0)}:=\tau_{(j+1,0,0)}.
\]
Otherwise we have $\tau_{(j,0,\ell)}:=\tau_{(j,0,\ell_{(j,0)})}$ for all $\ell \geq \ell_{(j,0)}$ so that we may define:
\[
    \tau_{(j,1,0)}:=\tau_{(j,0,\ell_{(j,0)})}\leq \tau_{(j+1,0,0)}.
\]

\item
We repeat the above inductively. We fix k and assume we have defined $\tau_{(j,0,0)}\leq \tau_{(j,k,0)}\leq \tau_{(j+1,0,0)}$. We seek to define:
\[
\tau_{(j,0,0)}\leq \tau_{(j,k,0)}\leq \tau_{(j,k,1)}\leq\ldots\leq \tau_{(j,k+1,0)}\leq \tau_{(j+1,0,0)}.
\]

Proceeding as in Step \ref{enum:defin of l stopping times k=0 step}, if $\tau_{(j,k,\ell)}=\tau_{(j+1,0,0)}$ we define:
\[
\tau_{(j,k,\ell+1)}=\tau_{(j+1,0,0)}.
\]
Otherwise $\tau_{(j,k,\ell)}< \tau_{(j+1,0,0)}$ so we may define as before $v_{(j,k,\ell)}=v(X_{\tau_{(j,k,\ell)}})$ and $X_{(j,k,\ell)}:=X_{\tau_{(j,k,\ell)}}$. We may then define:
\[
\tau_{(j,k,\ell+1)}=
\tau_{\WD}\wedge\begin{cases}
\inf\{t>\tau_{(j,k,\ell)}:d(X_t,v_{(j,k,\ell)})\geq r\},\quad \text{if}\;\; \rho(X_{\tau_{(j,k,\ell))}})>2^{-k}\\
\tau_{(j,k,\ell)},\quad \text{if}\;\; \rho(X_{\tau_{(j,k,\ell)}})\leq 2^{-k}\quad\text{or}\quad\tau_{(j,k,\ell)}=\tau_{(j+1,0,0)}
\end{cases}.
\]

Having defined $\tau_{(j,k,\ell)}$ for $\ell=0,1,\ldots$ we now turn to defining $\tau_{(j,k+1,0)}$. We establish the following lemma:
\begin{lem}
Either $\tau_{(j+1,0,0)}=\infty$ and $\tau_{(j,k,\ell)}\uparrow \tau_{(j+1,0,0)}$ as $\ell\uparrow \infty$, or else $\tau_{(j+1,0,0)}<\infty$ and there exists some random $\ell_{(j,k)}<\infty$ such that either $\rho(X_{(j,k,\ell_{(j,k)})})\leq 2^{-k}$ or $\tau_{(j,k,\ell(j,k))}=\tau_{(j+1,0,0)}$. 
\label{lem:fin number of jumps fixed j,k}
\end{lem}
In the former case ($\tau_{(j+1,0,0)}=\infty$) we define:
\[
\tau_{(j,k+1,0)}:=\tau_{(j+1,0,0)}.
\]
Otherwise we have $\tau_{(j,k,\ell)}:=\tau_{(j,k,\ell_{(j,k)})}$ for all $\ell \geq \ell_{(j,k)}$ so that we may define:
\[
    \tau_{(j,k+1,0)}:=\tau_{(j,k,\ell_{(j,k)})}\leq \tau_{(j+1,0,0)}.
\]

Repeating inductively in k we have defined $\tau_{(j,k,\ell)}$ for $(j,k,\ell)\in\mathbb{N}_0^3$, subject to proving Lemma \ref{lem:fin number of jumps fixed j,k}.
\end{enumerate}

\begin{proof}[Proof of Lemma \ref{lem:fin number of jumps fixed j,k}]

We consider sub-intervals $[rh,(r+1)h]$ ($r=0,1,\ldots$) of length $h>0$ to be determined, over each of which the diffusion term dominates the drift term. We write $N_r:=\lvert \{\ell':\tau_{(j,k,\ell')}\in [rh,(r+1)h]\cap [0,\tau_{\WD})\text{ and }\tau_{(j,k,\ell')}<\tau_{(j+1,0,0)}\}\rvert$. Then it is sufficient to show that $h>0$ may be chosen so that:
\[
\Pm(N_{r}=\infty)=0\quad \text{for all}\quad 0\leq r<\infty.
\]

We recall that we have fixed i, and moreover $X_t=X^i_t$ has driving Brownian motion $W_t=W^i_t$ which satisfies:
\begin{equation}
\lvert (X_{t_2}-X_{t_1})-(W_{t_2}-W_{t_1})\rvert\leq B(t_2-t_1)
\label{eq:displacement of X approximately displacement of B}
\end{equation}
if $X_t$ does not hit the boundary during the time interval $[t_1,t_2]$. 

We observe therefore that if our distance to the boundary is bounded from below then in order for our particle to die within a sufficiently small time interval, the driving Brownian motion $W_t$ must travel a distance bounded from below in this small time interval. In particular we suppose that we have $\tau_{(j,k,\ell)}\in [rh,(r+1)h]\cap [0,\tau_{\WD})$ with $\tau_{(j,k,\ell)}< \tau_{(j+1,0,0)}$ and $\rho(v_{(j,k,\ell)})\geq 2^{-k}$. Then in order to also have $\tau_{(j,k,\ell+1)}\leq (r+1)h$ it must be the case that $X_t$ hits $\partial B(v(X_{\tau_{(j,k,\ell)}}),r))$ before time $(r+1)h$. We now recall:
\[
d(X_{\tau_{(j,k,\ell)}},\partial B(v(X_{\tau_{(j,k,\ell)}}),r))\geq \frac{\rho(X_{\tau_{(j,k,\ell)}})}{2}\geq 2^{-k-1}.
\]
Therefore if $\tau_{(j,k,\ell)}\in [rh,(r+1)h]\cap [0,\tau_{\WD})$ with $\tau_{(j,k,\ell)}< \tau_{(j+1,0,0)}$ and $\rho(v_{(j,k,\ell)})\geq 2^{-k}$, then in order to also have $\tau_{(j,k,\ell+1)}< \tau_{(j+1,0,0)}$ and $\rho(v_{(j,k,\ell+1)})\geq 2^{-k}$ we must have $\lvert X_{(j,k,\ell+1)}-X_{(j,k,\ell)}\rvert\geq 2^{-k-1}$, which requires the driving Brownian motion satisfy 
\[
\lvert W_{\tau_{(j,k,\ell+1)}}-W_{\tau_{(j,k,\ell)}}\rvert\geq 2^{-k-1}-Bh.
\]
We note that this latter event happening is independent of $\mathcal{F}_{\tau_{(j,k,\ell)}}$, and for $h<\frac{2^{-(k+2)}}{B}$ has probability at most some $p<1$. Therefore at time $\tau_{(j,k,\ell)}\in [rh,(r+1)h]\cap [0,\tau_\infty\wedge\tau_\ST)$, the probability this is the final such stopping time in the interval $[rh,(r+1)h]$ is at least $1-p>0$. Recalling Remark \ref{rmk:convention for Geometric distribution}, we see that $N_r$ is stochastically dominated by a $\text{Geom}(1-p)$ distribution for $h<\frac{2^{-(k+2)}}{B}$.

We have now concluded the proof of Lemma \ref{lem:fin number of jumps fixed j,k}.
\end{proof}

We have left to prove Lemma \ref{lem:ordinal stopping times prop lem}:

\begin{proof}[Proof of Lemma \ref{lem:ordinal stopping times prop lem}] 
We begin with the $\omega=\omega_0^3$ case. This is true by definition.

Next, consider the case that $\omega=(j+1)\omega_0^2$. If $\tau_{(j+1,0,0)}=\infty$ then $\tau_{(j,k,0)}=\infty$ for all $k\geq 1$ and we are done. We may therefore assume $\tau_{(j,k+1,0)}<\tau_{(j+1,0,0)}<\infty$ for all $k\in\mathbb{N}_0$ otherwise we are done. Then Lemma \ref{lem:fin number of jumps fixed j,k} gives that $\rho(X_{\tau_{(j,k+1,0)}})=\rho(X_{\tau_{(j,k,\ell_{(j,k)})}})\leq 2^{-k}\ra 0$ as $k\ra \infty$ and so $d(X_{\tau_{j\omega_0^2+k\omega_0}},\partial U)\ra 0$ as $k\ra \infty$. Therefore by the almost-sure continuity of the path $X_t$ and the fact that $\tau_{(j+1) \omega_0^2} < \infty$ we have $\lim_{k\ra\infty}X_{\tau_{(j,k,0)}}\in \partial U$. Therefore we have $\tau_{j\omega_0^2+k\omega_0}\ra\tau_{(j+1)\omega_0^2}$ as $k\ra\infty$.

Finally, in the case that $\omega=j\omega_0^2+(k+1)\omega_0$, this is an immediate consequence of Lemma \ref{lem:fin number of jumps fixed j,k}.

\end{proof}

\subsubsection*{Step \ref{enum:Brownian construction step}}

We begin by constructing $\tilde{W}_t$ and showing that it can be written in the form \eqref{eq:driving BM for bessel in terms of for particle} for a Cadlag adapted process $n_t$ which we also construct. We recall that we define for $\omega<\omega_0^3$:
\begin{equation}\tag{\ref{eq:defin of Dt and Domegat}}
\begin{split}
D_t=\sum_{\omega}\Ind_{I_{\omega}}(t)d(X_t,v_{\omega}),\quad 0\leq t<\tau_{\WD},\\
D^{\omega}_t=(d(X_{t\wedge \tau_{\omega+1}-},v_{\omega})-d(X_{t\wedge \tau_{\omega}},v_{\omega})\big)\Ind(\tau_\omega<\tau_{\omega+1}).
\end{split}
\end{equation}
After adding a positive drift $B\Ind_{I_{\omega}}(t)$ $D^{\omega}_t$ becomes a submartingale, so we may take the Doob-Meyer decomposition, obtaining a mean zero Martingale term $W^{\omega}_t$ with quadratic variation $\int_0^t\Ind_{I_{\omega}}(s)ds$ (i.e. a Brownian motion started at time $\tau_{\omega}$ and stopped at time $\tau_{\omega+1}$). 

Indeed we can write:
\[
D_t^{\omega}=\sqrt{(X_{t\wedge \tau_{\omega+1}}-v_{\omega})\cdot (X_{t\wedge \tau_{\omega+1}}-v_{\omega})}-\sqrt{(X_{t\wedge \tau_{\omega}}-v_{\omega})\cdot (X_{t\wedge \tau_{\omega}}-v_{\omega})}
\]
so that we have:
\[
dD^{\omega}_t=\Ind_{I_{\omega}}(t)\frac{X_t-v_{\omega}}{\lvert X_t-v_{\omega}\rvert}\cdot dW_t+\text{finite variation terms}
\]
and therefore
\[
d\tilde{W}^{\omega}_t=\Ind_{I_{\omega}}(t)\frac{X_t-v_{\omega}}{\lvert X_t-v_{\omega}\rvert}\cdot dW_t.
\]

We fix $\hat{n}\in \Rm^d$ such that $\lvert \hat{n}\rvert=1$ so that $\hat{n}\cdot W_t$ is a Brownian motion. We now write:
\begin{equation}
\tilde W_t = \int_0^{t\wedge \tau_{\WD}} \sum_{\omega} \Ind_{I_\omega}(s) d\tilde W^\omega_s+\int_{t\wedge \tau_{\WD}}^t\hat{n}\cdot dW_s,\quad 0\leq t<\infty
\end{equation}
which is clearly a Brownian motion, since the $I_\omega$ form a countable partition of $[0,\tau_{\WD})$. We recall that we want to define $\tilde{W}_t$ beyond time $\tau_{\WD}$ if $\tau_{\WD}<\infty$. In particular we can write:
\[
\tilde{W}_t=\int_0^t\underbrace{\big(\sum_{\omega}\Ind_{I_{\omega}}(s)\frac{X_s-v_{\omega}}{\lvert X_t-v_{\omega}\rvert}+\Ind_{[\tau_{\WD},\infty)}(s)\big)\hat{n}}_{=:n_s}\cdot dW_s
\]
and hence we have \eqref{eq:driving BM for bessel in terms of for particle}.

We now claim the following is non-decreasing:
\begin{equation}
H_t:=D_0-D_t+
\int_0^t\begin{cases}
d\tilde W_s+Bds+ \frac{d-1}{2D_s} ds,\quad d>1\\
d\tilde W_s + B ds+ dL^D_s,\quad d=1
\end{cases},\quad 0\leq t<\tau_{\WD}.
\label{eq:formula for Ht}
\end{equation}
\begin{proof}[Proof \eqref{eq:formula for Ht} is non-decreasing]

It will be convenient here to extend the definition of $D_t$ by defining $D_{\tau_{\WD}}=r$ if $\tau_{\WD}<\infty$.

We proceed by ordinal induction. We inductively claim:
\begin{equation}
H_t\text{ is non-decreasing on }[0,\tau_{\omega}]\cap [0,\tau_{\WD})\text{ for }\omega\leq \omega_0^3
\label{eq:Ht nonincreasing induction step}
\end{equation}
which immediately implies \eqref{eq:formula for Ht} by Lemma \ref{lem:ordinal stopping times prop lem}. 

The $\omega=0$ case is immediate.

If $\omega=\omega'+1$ is a successor ordinal, then it is sufficient to show that $H_t$ is non-decreasing on $[\tau_{\omega'},\tau_\omega]$. We may assume $\tau_{\omega'}<\tau_\omega$, otherwise we are done.

We note that $D^{\omega'}_t$ must satisfy the following SDE:
\[
dD^{\omega'}_t=
\Ind(t\in I_{\omega'})\begin{cases}
d\tilde W^{\omega'}_t+\bar b^{\omega'}_tdt+\frac{d-1}{2D_t}dt,\quad d>1\\
d\tilde W^{\omega'}_t+\bar b^{\omega'}_tdt+dL^{D},\quad d=1
\end{cases}
\]
for some process $\bar b^{\omega'}_t\leq B$. Therefore for $\tau_{\omega'}\leq t<\tau_\omega$:
\[
\begin{split}
H_{t}-H_{\tau_{\omega'}}=D^{\omega'}_{\tau_{\omega'}}-D^{\omega'}_{t}+\int_{\tau_{\omega'}}^t\begin{cases}
d\tilde W_s+Bds+ \frac{d-1}{2D_s} ds,\quad d>1\\
d\tilde W_s + B ds+ dL^D_s,\quad d=1
\end{cases}
=\int_{\tau_{\omega'}}^t(B-\bar b^{\omega'}_s)ds
\end{split}
\]
which is non-increasing.

Moreover we note by construction that $\limsup_{t\uparrow \tau_{\omega}}D_t=r$ so that if $\tau_{\omega}<\tau_{\WD}$, $\limsup_{t\uparrow \tau_{\omega}}(H_{\tau_{\omega}}-H_t)\geq 0$. Thus we have dealt with the case where $\omega$ is a successor ordinal.

We finally consider the case whereby $\omega\leq \omega_0^3$ is a limit ordinal. If $\tau_{\omega'}=\tau_{\omega}$ for some $\omega'<\omega$ we are done by our induction hypothesis. Moreover $(H_t)_{0\leq t<\tau_{\omega}}$ is non-decreasing by our induction hypothesis. Therefore if $\tau_{\omega}=\tau_{\WD}$ we are done.

We now assume otherwise, so that for $\omega'<\omega$ we have $\tau_{\omega'}<\tau_{\omega}<\tau_{\WD}$. It is sufficient to show that $\limsup_{t\uparrow\tau_{\omega}}H_t\leq H_{\tau_\omega}$. We take a sequence of successor ordinals $\omega_n\uparrow\omega$ with $\omega_n<\omega$. For each n we have some $\omega_n\leq \omega_n'<\omega$ such that $\tau_{\omega'_n}<\tau_{\omega'_n+1}$. However we know by construction that $D_{\tau_{\omega'_n+1}-}=r$ so by the same calculation as in the case of successor ordinals, $\limsup_{t\uparrow \tau_{\omega}}(H_{\tau_{\omega}}-H_t)\geq 0$ hence we are done.

\end{proof}

Thus we have established \eqref{eq:driving BM for bessel in terms of for particle} whereby $H_t$ defined in \eqref{eq:defin of Dt and Domegat} is a non-decreasing adapted process.

\subsubsection*{Step \ref{enum:exist of strong solutions}}

Theorem 1.3 of \cite{Andres} gives the existence and uniqueness of strong solutions to reflected SDEs in convex domains where the drift is $C^1$ and Lipschitz. That theorem applies directly to \eqref{eq:SDE for Bessel processes no i} in the $d=1$ case. In the $d > 1$ case, the only issue is that the drift is locally Lipschitz but not globally Lipschitz. Here we must stop the process $\eta_t$ when it hits $\epsilon>0$, then take $\epsilon$ to zero and note that on any fixed finite time horizon the probability of hitting this barrier goes to zero as $\epsilon\ra 0$.

\subsubsection*{Step \ref{enum:eta dominates D}}

We have constructed a solution $(\eta,\tilde W)$ to \eqref{eq:SDE for Bessel processes no i} and claim that:
\begin{equation}\tag{\ref{eq:inequality to induct Bessel proof}}
\eta_t\geq D_t\geq r-d(X_t,\partial U),\quad 0\leq t<\tau_{\WD}.
\end{equation}
The second inequality is obvious, we now establish the first.

We recall that $(\eta,\tilde W)$ satisfies:
\begin{equation}\tag{\ref{eq:SDE for Bessel processes no i}}
\begin{split}
d\eta_t=
\begin{cases}
d\tilde W_t+\frac{d-1}{2\eta_t}dt+Bdt-dL^{r-\eta}_t,\quad d>1\\
d\tilde W_t+Bdt+dL^{\eta}-dL^{r-\eta}_t,\quad d=1
\end{cases},\quad 0\leq t<\infty
,\quad
\eta_0=r,
\end{split}
\end{equation}
whereas $D_t$ is a $[0,r]$-valued process which satisfies:
\begin{equation}\tag{\ref{eq:sde for Dt}}
dD_t =\begin{cases}
d\tilde W_t+ \frac{d-1}{2D_s} dt+Bdt  + dH_t,\quad d>1\\
d\tilde W_t + B dt+ dL^D_s ds + dH_t,\quad d=1
\end{cases},\quad 0\leq t<\tau_{\WD},
\end{equation}
for some non-decreasing adapted process $H_t$. Therefore we have:
\begin{equation}
\begin{split}
d(\eta_t-D_t)=dH_t+\Ind_{d>1}\big(\frac{d-1}{2\eta_t}-\frac{d-1}{2D_t}\big)dt\\+\Ind_{d=1}\big(dL^{\eta}_t-dL^{D}_t\big)-dL^{r-\eta}_t,\quad 0\leq t<\tau_{\WD}.
\end{split}
\label{eq:equation for eta-D on interval}
\end{equation}

We fix $\delta>0$ and assume for contradiction there exists $t_1<\tau_{\WD}$ such that $\eta_{t_1}-D_{t_1}\leq -2\delta$. We define $t_0=\sup\{t'<t_1:\eta_t'-D_{t'}\geq -\delta\}$. Then since $H_t$ is non-decreasing we have $(\eta_{t_0}-D_{t_0})\geq \limsup_{t'\uparrow t_1}(\eta_t'-D_{t'})\geq -\delta$. Therefore $t_0<t_1$ and we must have $\eta_{t}<D_t-\delta$ for $t\in (t_0,t_1]$. Thus as $H_t$ is non-decreasing $(\eta_{t_0}-D_{t_0})\leq \liminf_{t'\downarrow t_1}(\eta_t'-D_{t'})\leq -\delta$ and therefore $(\eta_{t_0}-D_{t_0})=-\delta$. Therefore we must have:
\[
\begin{split}
L_{t_1}^{r-\eta}-L_{t_0}^{r-\eta}=0,\quad L_{t_1}^{D}-L_{t_0}^{D}=0,\quad \Ind_{d>1}\big(\frac{d-1}{2\eta_t}-\frac{d-1}{2D_t}\big)dt\geq 0\quad\text{for}\quad t\in [t_0,t_1].
\end{split}
\]
Therefore we have:
\[
\begin{split}
-\delta\geq (\eta_{t_1}-D_{t_1})-(\eta_{t_0}-D_{t_0})=\underbrace{\int_{t_0}^{t_1}\Ind_{d>1}\big(\frac{d-1}{2\eta_s}-\frac{d-1}{2D_s}\big)ds}_{\geq 0}\\
+\underbrace{H_{t_1}-H_{t_0}}_{\geq 0}+\underbrace{\Ind_{d=1}(L^{\eta}_{t_1}-L^{\eta}_{t_0})}_{\geq 0}-\underbrace{(L^{D}_{t_1}-L^{D}_{t_0})}_{= 0}-\underbrace{(L_{t_1}^{r-\eta}-L_{t_0}^{r-\eta})}_{= 0}\geq 0.
\end{split}
\]
This is a contradiction, hence we must have $\eta_t\geq D_t$ for $t<\tau_{\WD}$.

We have now completed Step \ref{enum:eta dominates D}.
\subsubsection*{Step \ref{enum:Bessel construction independence}}

From \eqref{eq:driving BM for bessel in terms of for particle} we can write:
\[
d\tilde{W}^i_t=n^i_t\cdot dW^i_t
\]
for some processes $n^i_t$. We write $n^i_t(k)$ and $W^i_t(k)$ for the component of $n^i$ and $W^i$ in the $k^{\thh}$ dimension respectively. Therefore we have for $i\neq j$:
\[
d[\tilde{W}^i, \tilde{W}^j]_t=\sum_{k,l}n^i_t(k)n^j_t(l)d[W^i(k),W^j(l)]_t=0.
\]

Thus the Brownian motions $\tilde{W}^i$ and $\tilde{W}^j$ have zero covariance, so $\tilde{W}^1,\ldots,\tilde{W}^N$ are jointly independent. Since each $(\eta^i,\tilde{W}^i)$ is a measurable function of $\tilde{W}^i$, they must also be independent.

Thus we have constructed independent identically distributed strong solutions \newline $(\eta^1,\tilde{W}_1),\ldots,(\eta^N,\tilde{W}^N)$ of \eqref{eq:SDE for Bessel processes no i} satisfying \eqref{eq:lower bd dist to du time tauinf taustop} so have established Proposition \ref{prop:dominated by Bessel fns up to time tauinfty wedge taustop}.

\qed

\subsection{Proof of Lemma \ref{lem:no simultaneous hitting for Bessel processes}}
We consider on the same probability space $(\Omega,\mathcal{F},\Pm)$ two independent strong solutions $(\eta^k,\tilde W^k)$ ($k=1,2$) to \eqref{eq:SDE for Bessel processes no i}:
\begin{equation}\tag{\ref{eq:SDE for Bessel processes no i}}
\begin{split}
d\eta_t=
\begin{cases}
d\tilde W_t+\frac{d-1}{2\eta_t}dt+Bdt-dL^{r-\eta}_t,\quad d>1\\
d\tilde W_t+Bdt+dL^{\eta}_t-dL^{r-\eta}_t,\quad d=1
\end{cases}
,\quad
\eta_0=r
\end{split}
\end{equation}
such that $(\eta^1,\tilde{W}^1)$ and $(\eta^2,\tilde{W}^2)$ are independent of each other. Given $t_0\in \Qm_{\geq 0}$ we define:
\[
\tau_{t_0}=\inf\{t> t_0:\text{min}(\eta_t^1,\eta_t^2)\leq \frac{r}{2}\}.
\]
We write for $k=1,2$:
\[
\overline{W}^k_t=\tilde W^k_t+
\begin{cases}
0,\quad t\leq t_0\\
B(t-t_0)+\int_{t_0}^t\frac{d-1}{2\tilde \eta_s}ds,\quad t_0\leq t\leq \tau_{t_0}\\
B(\tau_{t_0}-t_0)+\int_{t_0}^{\tau_{t_0}}\frac{d-1}{2\tilde \eta_s}ds,\quad t\geq \tau_{t_0}
\end{cases}.
\]

By Girsanov's theorem there is an equivalent probability measure $\overline{\Pm}$ under which $\overline{W}^1$ and $\overline{W}^2$ are Brownian motions, which by examining the covariation we see must be independent. Now we observe that $(\eta^k_t)_{t_0\leq t\leq \tau_{t_0}}$ must satisfy:
\[
\begin{split}
d\eta^k_t=d\overline{W}^k_t-dL^{r-\eta^k}_t,\quad t_0\leq t\leq \tau_{t_0},\quad
\eta^k_0=r.
\end{split}
\]
We have the existence of a strong solution $\hat\eta^k_t=\eta^k_{t_0}+\overline W^k_t-\sup_{t_0\leq t'\leq t}\overline{W}^k_{t'}$ which by computing $d(\hat{\eta}^k-\eta^k)^2\leq 0$ we see must be equal to $\eta^k$ (i.e. we have pathwise uniqueness). Therefore $\eta^k$ is a measurable function of $\overline{W}^k$, hence $r-\eta^1$ and $r-\eta^2$ are independent and distributed under $\overline\Pm$ like the absolute value of a 1-dimensional Brownian motion. Therefore by Pythagoras $\sqrt{(r-\eta^1_t)^2+(r-\eta^2_t)^2}$ must be distributed under $\overline{\Pm}$ like the absolute value of a 2-dimensional Brownian motion. Therefore $\overline{\Pm}(\exists t_0< t<\tau_{t_0}\text{ such that }\eta^1_t=\eta^2_t=r)=0$ hence $\Pm(\exists t_0< t<\tau_{t_0}\text{ such that }\eta^1_t=\eta^2_t=r)=0$. Taking the union over $t_0\in \Qm_{\geq 0}$ we are done.

\qed

\subsection{Proof of Proposition \ref{prop:general weak solns are global}}

We now use Proposition \ref{prop:dominated by Bessel fns up to time tauinfty wedge taustop} and Lemma \ref{lem:no simultaneous hitting for Bessel processes} to establish that $\tau_{\WD}=\infty$. The main idea is that by Proposition \ref{prop:dominated by Bessel fns up to time tauinfty wedge taustop} the event $\tau_{\WD}<\infty$ corresponds to the event that two of the $\eta^i$ hit r at the same time, which almost surely doesn't happen by Lemma \ref{lem:no simultaneous hitting for Bessel processes}.

In \cite{Burdzy2000} they justified that $\tau_{\ST}\geq \tau_{\infty}$ on the basis of the hitting time of a Brownian motion in an arbitrary domain having a continuous density. However, \eqref{eq:lower bd dist to du time tauinf taustop} and Lemma \ref{lem:no simultaneous hitting for Bessel processes} give us that $\tau_{\ST}\geq\tau_{\infty}\wedge \tau_{\max}$ for free. Indeed, if $\tau_{\ST}<\tau_{\infty}\wedge \tau_{\max}$ then two particles (say $X^i$ and $X^j$) hit the boundary at time $\tau_{\ST}$, so that by \eqref{eq:lower bd dist to du time tauinf taustop} $\eta_\ST^i=\eta^j_\ST=r$. Therefore by Lemma \ref{lem:no simultaneous hitting for Bessel processes}, $\Pm(\tau_{\ST}<\tau_{\infty}\wedge\tau_{\max})=0$.

We now have $\tau_\ST\geq \tau_\infty\wedge\tau_{\max}$ almost surely. Since between killing times $\tau_k$, the particles can't travel an infinite distance over a finite time horizon $T<\infty$, we may inductively in $k$ see that $\tau_{\max}\geq \tau_k\wedge T$. Since $T<\infty$ is arbitrary, $\tau_{\max}\geq \tau_{\infty}$.

Thus $\tau_{\infty}\leq \tau_{\max}\wedge \tau_{\ST}$, so we now seek to show $\tau_{\infty}=\infty$ almost surely. We assume for the sake of contradiction $\tau_{\infty}<\infty$ with positive probability. We write $\tau_k^i$ for the $k^{\thh}$ jump time of particle $i$. Then there exists $1\leq i\leq N$ such that $\tau_k^i\uparrow \tau_\infty<\infty$ as $k\ra\infty$ with positive probability. If this is the case, then $i$ must jump an infinite number of times up to time $\tau_\infty$. Therefore by the pigeonhole principle, for some $j\neq i$, $i$ jumps infinitely many times onto $j$ before time $\tau_\infty<\infty$ with positive probability 

We therefore assume $i$ jumps onto $j$ infinitely many times up to time $\tau_\infty<\infty$. Since the drift is bounded and $\tau_{\infty}=\tau_{\WD}<\infty$ we almost surely have:
\begin{equation}
\sum_{\substack{k\text{ such that}\\
\text{i jumps onto j}}}(X^i_{{\tau^i_{k+1}}-}-X^i_{\tau^i_k})^2<\infty.
\label{eq:sum of squares of i to j jump times to boundary hitting position}
\end{equation}
We write $\tau^{i,j}_k$ for the $k^{\thh}$ time particle $i$ hits the boundary and jumps to particle $j$. Then by \eqref{eq:sum of squares of i to j jump times to boundary hitting position} we have:
\[
d(X^j_{\tau^{i,j}_k},\partial U)=d(X^i_{\tau^{i,j}_k},\partial U)\ra 0\quad\text{as}\quad k\ra\infty.
\]
Thus $\limsup_{t\uparrow \tau_\infty}\eta^i_t=\limsup_{t\uparrow \tau_\infty}\eta^j_t=r$ by \eqref{eq:lower bd dist to du time tauinf taustop}. Thus if $\tau_\infty=\tau_{\WD}<\infty$ with positive probability then $\eta^i_{\tau_{\infty}}=\eta^j_{\tau_{\infty}}=r$ with positive probability, which is not the case by Lemma \ref{lem:no simultaneous hitting for Bessel processes}.

Therefore $\tau_{\WD}=\infty$ almost surely.\qed

\subsection{Proof of Proposition \ref{prop:laws at positive times constrained to compact set}}
It is sufficient by \cite[Theorem 4.10]{Kallenberg2017} to show that the expected mean measures,
\[
\begin{split}
\{\chi_t:\chi_t(A):=\expE[m^N_t(A)]\text{ whereby }m^N_t:=\vartheta^N(\vec{X}^N_t)\text{ for some weak solution }\vec{X}^N_t\text{ to }\eqref{eq:N-particle system sde}\\
\text{for any initial condition }\vec{X}^N_0\sim \upsilon^N \in\mathcal{P}(U^N)\text{, any $N$ and any }t\geq T_0\},
\end{split}
\]
are tight. We define $V_{\delta}=\{x\in U:d(x,\partial U)\geq \delta\}$. Then we have $\chi_t(V_{\delta}^c)=\Pm(d(X^{N,1}_t,\partial U)< \delta)\leq \Pm(r-\eta^{N,1}_t< \delta)=\Pm(\eta^{N,1}_t>r-\delta)$ by Tonelli's theorem and Proposition \ref{prop:dominated by Bessel fns up to time tauinfty wedge taustop}. This bound is uniform over all weak solutions for all N, all initial conditions, and all $T\geq T_0$, hence we are done.

\qed

\subsection{Proof of Part \ref{enum:bound on mass near bdy after small time} of Proposition \ref{prop:bound on mass near bdy}}

We henceforth fix $T_0>0$ and $\epsilon>0$. We shall take $K_{\epsilon,T_0}=V_c=\{x:d(x,\partial U)\geq c\}$ for $c>0$ to be determined. We may by Proposition \ref{prop:dominated by Bessel fns up to time tauinfty wedge taustop} construct i.i.d. solutions of \eqref{eq:SDE for Bessel processes no i} $(\eta^{1},\tilde{W}^1),\ldots,(\eta^N,\tilde{W}^N)$ such that:
\[
\begin{split}
\{m_t^N(K^c)\geq \epsilon\text{ for some }T_0\leq t\leq T\}\subseteq\{\sup_{T_0\leq t\leq T}\frac{1}{N}\sum_{j=1}^N\Ind(d(X^j,\partial U)\leq c)\geq \epsilon\}\\
\subseteq \{\sup_{T_0\leq t\leq T}\frac{1}{N}\sum_{j=1}^N\Ind(\eta^j_t\geq r-c)\geq \epsilon\}.
\end{split}
\]
Therefore it is sufficient to show that we may take $c>0$ small enough such that:
\begin{equation}
\limsup_{N\ra\infty}\Pm(\sup_{T_0\leq t\leq T}\frac{1}{N}\sum_{j=1}^N\Ind(\eta^j_t\geq r-c)\geq \epsilon)=0.
\label{eq:Bound on bessel processes for bound on mass in probability}
\end{equation}

Our strategy will be to implement Kingman's Subadditive Ergodic Theorem. We will establish \eqref{eq:Bound on bessel processes for bound on mass in probability} with $\eta^1,\ldots,\eta^N$ constructed on a different probability space (which is sufficient). We consider a strong solution $(\eta,W)$ of \eqref{eq:SDE for Bessel processes no i} on the probability space $(\Omega,\mathcal{F},\Pm)$. We thereby, by taking an infinite product, construct i.i.d. solutions $\eta^i$ of \eqref{eq:SDE for Bessel processes no i} on the probability space $(\Omega,\mathcal{F},\Pm)^{\otimes \infty}$. It is classical that the following map is ergodic:
\[
\mathcal{T}:(\Omega,\mathcal{F},\Pm)^{\otimes \infty}\ni (\omega_1,\omega_2,\ldots)\mapsto (\omega_2,\omega_3,\ldots)\in (\Omega,\mathcal{F},\Pm)^{\otimes\infty}.
\]

For $c>0$ to be determined and every $n\in\mathbb{N}$ we let $g_n(\omega)=\sup_{T_0\leq t\leq T}\sum_{1\leq i\leq n}\Ind(\eta_t^i\geq r-c)$. Then it is easy to see $g_n$ satisfies:
\begin{equation}
g_{n+m}(\omega)\leq g_n(\omega)+g_m(\mathcal{T}^n(\omega)).
\label{eq:ineq for kingman sub for close count}
\end{equation}
Therefore by Kingman's Subadditive Ergodic Theorem we have:
\[
\lim_{n \to \infty} \frac{1}{n}g_n(\omega) = \inf_{m \geq 1} \expE[\frac{1}{m}g_m] \quad \quad \Pm^{\otimes \infty} \text{-almost surely.}
\]
Thus it is sufficient to establish that there exists $c>0$ and $n<\infty$ such that $\expE[\frac{1}{n}g_n]< \epsilon$. We fix $n>\frac{2}{\epsilon}$ and note that:

\[
\begin{split}
\expE[\frac{1}{n}g_n]\leq \frac{1}{n}\Pm(g_n\leq 1)+\Pm(g_n\geq 2)\leq \frac{\epsilon}{2}+\Pm(g_n\geq 2).
\end{split}
\]
Therefore it is sufficient to show $\Pm(g_n\geq 2)<\frac{\epsilon}{2}$ for some $c>0$ small enough. We may consider the ranked particles $\eta^{(1)}_t\geq \eta^{(2)}_t\geq \ldots\geq \eta^{(n)}_t$, in particular we consider the second ranked particle:
\[
\eta^{(2)}_t=\sup\{\eta^i_t:\exists j\neq i\text{ such that }\eta^j_t\geq \eta^i_t\}
\]
which we note has continuous sample paths. Then we have:
\[
\{g_n\geq 2\}=\{\sup_{T_0\leq t\leq T}\eta^{(2)}\geq r-c\}.
\]
Our goal is to show the probability of this event is less than $\frac{\epsilon}{2}$ for $c>0$ small enough. Since $\eta^{(2)}$ has continuous sample paths and $[T_0,T]$ is compact we have:
\[
\begin{split}
\{\eta^i_t=\eta^j_t=r\text{ for some }i\neq j\text{ and }T_0\leq t\leq T\}\\
=\{\eta^{(2)}_t=r\text{ for some }T_0\leq t\leq T\}=\cap_{c>0}\{\eta^{(2)}_t\geq r-c\text{ for some }T_0\leq t\leq T\}.
\end{split}
\]
The probability of this event is zero by Lemma \ref{lem:no simultaneous hitting for Bessel processes} hence we have
\[
\lim_{c\ra 0}\Pm(\sup_{T_0\leq t\leq T}\eta^{(2)}\geq r-c)=0.
\]
Therefore $\Pm(g_n\geq 2)=\Pm(\sup_{T_0\leq t\leq T}\eta^{(2)}\geq r-c)<\frac{\epsilon}{2}$ for $c>0$ small enough. Therefore we have:
\[
\lim_{N\ra\infty}\sup_{T_0\leq t\leq T}\frac{1}{N}\sum_{j=1}^N\Ind(\eta^j_t\geq r-c)<\epsilon\quad \Pm^{\otimes\infty}\text{-almost surely.}
\]

Thus we have \eqref{eq:Bound on bessel processes for bound on mass in probability} on our original probability space. We finally note that the choice of $c>0$ is dependent only upon the parameters of the Bessel processes, hence dependent only upon $T_0,\; T,\;\epsilon,\;B$ and $r$.

\qed

\subsection{Proof of Part \ref{enum:bound on mass near bdy} of Proposition \ref{prop:bound on mass near bdy}}

We recall $\xi^N=\vartheta^N_{\#}\upsilon^N$. Since $\{\xi^N\}$ are tight as a family of random measures, for every $\epsilon$, $\delta>0$ there exists $c'>0$ such that:
\[
\Pm(\xi^N(V_{c'}^c)\geq \frac{\epsilon}{10})<\delta.
\] 

So, by bounding the distance travelled by a particle in time $T_0$ for small enough $T_0>0$, we have that for some smaller $c''>0$ and all $N$ large enough:
\[
\Pm(m^N_t(V_{c''}^c)\geq \epsilon\text{ for some }t\leq T_0)<\delta.
\]

We now take $\hat{c}(\epsilon,\delta)=c(\epsilon,T_0)\wedge c''$ so that $\hat K_{\epsilon,\delta}=V_{\hat{c}}$ satisfies \eqref{eq:bound on mass near bdy}.

\qed

\subsection{Proof of Proposition \ref{prop:bound on number of jumps by particle}}

Here we adopt a strategy similar to Part 1 of the proof of \cite[Theorem 1.3]{Burdzy2000} (where they considered the Brownian case). There they argued that a positive proportion of specially selected particles stay within a given set with probability converging to 1. Then they argued that each time some particle dies there is a probability bounded away from 0 of this particle jumping onto one of these specially selected particles. If that is the case, then the probability of not dying off is bounded away from 0 as the distance between the given set and the boundary is bounded away from 0. Thus each time a particle hits the boundary, there is a probability bounded away from 0 of this being the last death time of the particle so long as the specially selected particles are within the given set.

Their proof that a positive proportion of specially selected particles stay within a given set with probability converging to 1 relies on the independence of the particles in the Brownian case. This does not apply in our case, so instead we must use the closed set we constructed in Part \ref{enum:bound on mass near bdy} of Proposition \ref{prop:bound on mass near bdy}. Moreover we break $[0,T]$ into a large number of sub-intervals, over each of which the diffusive term dominates the drift term (this is not necessary in the $b=0$ case as there is no drift).

We recall that $\text{Geom}($p$)$ refers to the geometric distribution on $\{1,2,\ldots\}$ with distribution given by $\Pm(G \geq k) = (1-p)^{k-1}$ (Remark \ref{rmk:convention for Geometric distribution}). We now set 
\begin{equation}\label{eq:formula for stopping time}
\tau^N_{\epsilon}=\inf\{t\geq 0:m^N_t(\hat K_{\frac{1}{2},\epsilon}^c)\geq \frac{1}{2}\},
\end{equation}
so that we have $\limsup_{N\ra\infty}\Pm(\tau^N_{\epsilon}\leq T)\leq \epsilon$. We break $[0,T]$ into M to be determined sub-intervals $[rh,(r+1)h]$ ($r=0,\ldots,M-1$) of length $h=\frac{T}{M}$ and define 
\[
J_r:=\lvert \{k:\tau_k^i\in [rh,(r+1)h]\text{ and }\tau_i^i\leq \tau_\epsilon^N\}\rvert.
\]

We recall that $X^i_t$ has driving Brownian motion $W^i_t$ and satisfies:
\begin{equation}\tag{\ref{eq:displacement of X approximately displacement of B}}
\lvert (X^i_{t_2}-X^i_{t_1})-(W^i_{t_2}-W^i_{t_1})\rvert\leq B(t_2-t_1)
\end{equation}
if $X^i_t$ does not hit the boundary during the time interval $[t_1,t_2]$. We recall the observation that if our distance to the boundary is bounded from below then in order for our particle to die within a sufficiently small time interval, the driving Brownian motion $W_t$ must travel a distance bounded from below in this small time interval. Using Part \ref{enum:bound on mass near bdy} of Proposition \ref{prop:bound on mass near bdy} take $\delta=\frac{\hat{c}(\epsilon,\delta)}{3}>0$ so that $d(\hat K_{\frac{1}{2},\epsilon},\partial U)= 3\delta$, and further take $M>\frac{TB}{\delta}$. Thus if $rh\leq \tau^i_k\leq \tau^i_{k+1}\leq (r+1)h$ and $X^i_{\tau^i_k}\in \hat K_{\frac{1}{2},\epsilon}$ we must have:
\[
\lvert W_{(r+1)h\wedge \tau^i_{k+1}-}- W_{\tau^i_{k}}\rvert \geq 3\delta - Bh\geq 2\delta.
\]

Moreover $\Pm(\lvert W_{(r+1)h\wedge \tau^i_{k+1}-}- W_{\tau^i_{k}}\rvert \geq 2\delta \lvert \mathcal{F}_{\tau_k^i})<p$ for some $p<1$. Therefore at each death time $\tau^i_k\in [rh,(r+1)h]$ with $\tau^i_k\leq \tau_\epsilon^N$ there is a probability at least $\frac{1}{2}$ of jumping to a particle in $\hat{K}_{\frac{1}{2},\epsilon}$ and if this is the case there is then a probability of at least $1-p>0$ of this being the final time particle i jumps during the interval $[rh,(r+1)h]$. Therefore $J_r$ can be coupled to a Geometric random variable of success probability $(1-p)\times \frac{1}{2}$ which is independent of $\mathcal{F}_{hr}$ and dominates $J_r$.

\qed

\section{Ergodicity of the $N$-particle System \eqref{eq:N-particle system sde} - Theorem \ref{theo:Ergodicity of the N-Particle System}}\label{section:ergodicity}

The goal of this section is to establish Theorem \ref{theo:Ergodicity of the N-Particle System}, giving ergodicity of the particle system for fixed N. We recall that in Theorem \ref{theo:Ergodicity of the N-Particle System} we assume $U$ is bounded and path-connected, which we therefore assume in this section. Since $N$ is fixed, we neglect to write it for convenience. We write $G=U^N$ and $P_t$ for the transition semigroup for $\vec{X}$. We recall the Doeblin condition in continuous time:
\begin{defin}[Doeblin Condition, \cite{Locherbach2015}]
There exists $t_*>0$, $\alpha>0$ and a probability measure $\nu$ such that for any $x\in G$, $P_{t_*}(x,dy)\geq \alpha \nu(dy)$.
\end{defin}

We now recall \cite[Corollary 2.7]{Locherbach2015}:
\begin{theo}[\cite{Locherbach2015}]
Assume that the Doeblin condition in continuous time holds. Then there exists a unique invariant distribution $\psi$, and moreover we have:
\[
\lvert\lvert P_t(x,\cdot)-\psi(\cdot)\rvert\rvert_{\TV}\leq (1-\alpha)^{\lfloor \frac{t}{t_*}\rfloor},\quad \forall x\in G.
\]
\end{theo}

Thus it is sufficient to establish the Doeblin condition holds. 

\underline{Step 1}

We define $V_{\epsilon}=\{x\in U: d(x,\partial U)\geq \epsilon\}$ and $G_{\epsilon}=V_{\epsilon}^N$. We fix $0<\epsilon_1<\frac{r}{2}$. We shall construct K compact, smooth and path connected such that:
\[
G_{{\epsilon_1}}\subseteq K\subset\subset G.
\]
In particular if $\vec{x}=(x_1,\ldots,x_N)\in U^N$ then $d(x_i,\partial U)\geq {\epsilon_1}$ for all i implies $\vec{x}\in K$.

We fix $\vec{x}^{\ast}\in G$ and define the following function:
\[
p(\vec{x})=\sup_{\substack{\gamma:\vec{x}^{\ast}\ra \vec{x}\\ \text{a path in }U}}d(\gamma,\partial U).
\]
Then since p is continuous and positive, there exists $\epsilon'>0$ such that $p\geq \epsilon'>0$ on $G_{\epsilon_1}$. We then define the compact, path-connected set:
\[
K'=\{x\in U:p(x)\geq \epsilon'\}^N\supseteq G_{\epsilon_1}.
\]
We now expand $K'$ a bit to obtain a smooth domain $K$. There exists $\epsilon''>0$ such that $d(K',\partial U)>\epsilon''$. We take $\varphi\in C_c^\infty(\Rm^{Nd})$ a mollifier supported on $B(0,\frac{\epsilon''}{4})$ so that by Sard's theorem there exists $0<c<1$ such that
\[
K''=\{x:\varphi\ast\Ind_{K'+B(0,\frac{\epsilon''}{2})}> c\}\supseteq K'\supseteq G_{\epsilon_1}
\]
is a compact domain with smooth boundary. Thus taking $K$ to be the path-connected component of $K''$ containing $K'$, we obtain our desired domain.

\underline{Step 2}

We recall that $U$ satisfies the interior ball condition with radius r. We may by Proposition \ref{prop:dominated by Bessel fns up to time tauinfty wedge taustop} define $N$ i.i.d. Bessel processes, with positive drift B, $\eta^1,\ldots,\eta^N$, such that $r-\eta^i\leq d(X^i,\partial U)$ for each i. Then with probability at least $p_1$ for some $p_1>0$, $\eta^1_1,\ldots,\eta^N_1\leq r-2\epsilon_1$. This gives us that there exists $p_1>0$ such that $P_1(\vec{x}, G_{2{\epsilon_1}})\geq p_1$ for all $\vec{x}\in G= U^N$.

\underline{Step 3}

For $\vec{u}=(u_1,\ldots,u_{Nd})\in G$ and $\epsilon>0$ we define $D(\vec{u},\epsilon)=\{(u'_1,\ldots,u'_{Nd}):\lvert u'_i-u_i\rvert<\epsilon\}$. We take $\vec{u}\in G$ and $\epsilon_2>0$ such that $D(\vec{u},5\epsilon_2)\subseteq G_{2\epsilon_1}$ and fix $C=D(\vec{u},\epsilon_2)$. We claim that there exists $\delta_2>0$ and $p_2>0$ such that $P_{\delta_2}(\vec{x},C)\geq p_2$ for all $\vec{x}\in G_{2\epsilon_1}$.

We let $Q_t(x,.)$ be the transition kernel for Brownian motion started at $\vec{x}\in K$ and killed when it hits $\partial K$. We can write the SDE for $\vec{X}_t$ between jump times as:
\[
d\vec{X}_t=\vec{b}(\vec{X}_t)dt+d\vec{W}_t.
\]
Since the drift is bounded, and both $d(C^c,D(\vec{u},\frac{\epsilon_2}{2}))$ and $d(K,\partial G)$ are bounded away from 0, there exists $\delta_2>0$ small enough such that for all $x\in K$:
\[
\begin{split}
\vec{x}+\vec{W}_{\delta}\in D(\vec{u},\frac{\epsilon_2}{2})\text{ and } \vec{x}+\vec{W}_{t'}\text{ does not leave }K\text{ for }t'\leq \delta\\
\Rightarrow \vec{X}_{\delta}\in C\text{ and }\vec{X}_{t'}\text{ does not hit }\partial G\text{ for }t'\leq \delta.
\end{split}
\]
Therefore we have:
\[
P_{\delta_2}(\vec{x},C)\geq Q_{\delta_2}(\vec{x},D(\vec{u},\frac{\epsilon_2}{2}))\quad\text{for all}\quad x\in G_{2\epsilon_1}.
\]
Taking a smooth function $\Ind_{D(\vec{u},\frac{\epsilon_2}{4})}\leq \Phi\leq \Ind_{D(\vec{u},\frac{\epsilon_2}{2})}$ we see $(t,\vec{x})\mapsto Q_t(\vec{x},\Phi)$ is a smooth solution of the heat equation on K with Dirichlet boundary conditions, so by the Maximum principle $\vec{x}\mapsto Q_1(\vec{x},\Ind_{D(\vec{u},\frac{\epsilon_2}{2})})$ is bounded away from 0 on $G_{2\epsilon_1}$. 
Thus $P_{\delta_2}(\vec{x},C)$ is bounded away from 0 on $G_{2\epsilon_1}$.

\underline{Step 4}

Lemma \ref{lem:upper and lower bounds on density of diffusions} then implies there exists $p_3>0$ such that $P_{1}(\vec{x},.)\geq p_3\text{Leb}_{\lvert_{C}}(.)$ for all $\vec{x}\in C$. Setting $t_*=1+\delta_2+1$, $\alpha=p_1p_2p_3\text{Leb}(C)$ and $\nu=\frac{1}{\text{Leb}(C)}\text{Leb}_{\lvert_{C}}$ we have established Doeblin's condition.

This completes our proof of Theorem \ref{theo:Ergodicity of the N-Particle System}.

\qed

\section{Density Estimate for the Proof of Theorem \ref{theo:Hydrodynamic Limit Theorem}}\label{section:density Estimates}
Unlike the previous section, we no longer assume $U$ is path-connected or bounded; here we assume only that $U$ is an open subdomain of $\Rm^d$ satisfying the uniform interior ball condition - Condition \ref{cond:ball}. We take a sequence of weak solutions to the Fleming-Viot particle system with generalised dynamics
\[
(\vec{X}^N_t,\vec{W}^N_t,\vec{b}^N_t)_{0\leq t<\infty}=((X_t^{N,1},\ldots,X^{N,N}_t),(W_t^{N,1},\ldots,W^{N,N}_t),(b_t^{N,1},\ldots,b^{N,N}_t))_{0\leq t<\infty}
\]
and with initial conditions $\vec{X}^N_0\sim\upsilon^N$. Moreover the drifts $b^{N,i}_t$ are uniformly bounded with $\lvert b^{N,i}_t\rvert\leq B<\infty$. We define $m^N_t$ and $J^N_t$ as in \eqref{eq:defin mN} and \eqref{eq:defin JN}:
\[
m^N_t=\vartheta^N(\vec{X}^N_t),\quad J^N_t = \frac{1}{N} \sup \{ k \in \mathbb{N}\;|\; \tau^N_k \leq t \}.
\]

The goal of this section is to establish the following lemma, which provides for controls on possible sub-sequential limits:
\begin{lem}
For fixed $T<\infty$ we assume that laws of $\{ (m^N_t,J^N_t)_{0\leq t\leq T})\}_{N \geq 2}$ are a tight family of measures on $D([0,T];\mathcal{P}_{\Wah}(U)\times \Rm_{\geq 0})$ with limit distributions supported on $C([0,T];\mathcal{P}_{\Wah}(U)\times \Rm_{\geq 0})$. Then for every subsequential limit in distribution $(m_t,J_t)_{0\leq t\leq T}\in C([0,T];\mathcal{P}_{\Wah}(U)\times \Rm_{\geq 0})$ we have:
\begin{enumerate}
\item
The random measure m defined by $dm=dm_tdt$ is almost surely absolutely continuous with respect to $\text{Leb}_{U\times [0,T]}$.
\label{enum:char-almost sure abs cty of m}
\item
For every $0<t\leq T$ we almost surely have $m_t$ is absolutely continuous with respect to $\text{Leb}_U$.
\label{enum:char-almost sure abs cty of marginals}
\end{enumerate}
\label{lem:char-almost sure abs cty}
\end{lem}

Note that we are not claiming here that almost surely $m_t$ is absolutely continuous with respect to $\text{Leb}_U$ for all $0<t\leq T$.

We focus on the proof of Part \ref{enum:char-almost sure abs cty of m} of Lemma~\ref{lem:char-almost sure abs cty} - the proof of Part \ref{enum:char-almost sure abs cty of marginals} is the same. We then use the machinery we construct to prove Lemma \ref{lem:char-almost sure abs cty} to prove the following lemma:
\begin{lem}
We assume that $\{\mathcal{L}(m^N_0)\}$ is tight in $\mathcal{P}(\mathcal{P}_{\Wah}(U))$. Then for any $T<\infty$ we have:
\begin{equation}
\limsup_{N\ra\infty}
\expE[\sup_{t\leq T}m^N_t(B(0,R)^c)]\ra 0\quad\text{as}\quad R\ra\infty.
\label{eq:tight ics then mass stays within ball}
\end{equation}
\label{lem:tight after bounded times}
\end{lem}

The proofs of this section shall rely on an analysis of the "Dynamical Historical Processes" defined in \cite{Bieniek2018}.

\subsection{Dynamical Historical Processes}

We provide a definition of "Dynamical Historical Process" (DHP) which is equivalent to that found in \cite{Bieniek2018} but will be more useful for our purposes. The Dynamical Historical Process $(\mathcal{H}^{N,i,t}_s)_{0\leq s\leq t}$ is the unique continuous path from time $0$ to time $t$ which is equal to one of the particles at all times and equal to $X^i_t$ at time $t$.

\begin{figure}[ht]
  \begin{center}
    \begin{tikzpicture}
      \foreach \y [count=\n]in {0,2}{	
			\draw[->,snake=snake,ultra thick,red](0,\y) -- ++(1,1);
			\draw[->,snake=snake,ultra thick,blue](1,1+\y) -- ++(-1,1);
			\draw[->,snake=snake,red](1,1+\y) -- ++(1,1);
      }
       \foreach \y [count=\n]in {0,2}{	
			\draw[-,dashed,blue](4,1+\y) -- ++(-3,0);
			\draw[->,snake=snake,blue](3,\y) -- ++(1,1);
      }
 
       \foreach \y [count=\n]in {2}{	
			\draw[-,dashed,red](-4,\y) -- ++(4,0);
			\draw[->,snake=snake,red](-3,-1+\y) -- ++(-1,1);
			\draw[->,snake=snake,blue](0,\y) -- ++(-1,1);
      }
			\draw[-,thick,black](4,0) -- ++(0,4); 
			\draw[-,thick,black](-4,0) -- ++(0,4); 	
	\draw node[draw,circle,inner sep=2pt,fill]  at (0,4){};

	\draw node at (4.4,3) {$\tau^{i}_{k_0}$};
	\draw node at (-0.25,4) {$i$};
	\draw node at (2.6,2.5) {$i = j_0$};
	
	\draw node at (3,0.5) {$j_2$};
	\draw node at (4.4,1) {$\tau^{j_2}_{k_2}$};
	
	\draw node at (-3,1.5) {$j_1$};
	\draw node at (0.9,2.5) {$j_1$};
	\draw node at (1.9,3.5) {$j_1$};
	\draw node at (-4.4,2) {$\tau^{j_1}_{k_1}$};
	
	\draw node at (0.9,0.5) {$j_3$};
	\draw node at (1.9,1.5) {$j_3$};
	\draw node at (5.75,3.5) {$t$};
    \end{tikzpicture}
    \caption{The continuous thick path denotes the path of the DHP up to time $t$.}  \label{fig:DHP example}
  \end{center}
\end{figure}

We shall define the set of "Chains" $\mathcal{C}^N$ and associate to each $\alpha\in \mathcal{C}^N$ a solution $(X^{\alpha},W^{\alpha})$ of:
\[
dX^{\alpha}_t=b(X^{\alpha}_t,m^N_t)dt+dW^{\alpha}_t,\quad 0\leq t<\tau=\inf\{t:X^{\alpha}_{t-}\in\partial U\}.
\]
Each $\alpha\in\mathcal{C}^N$ provides a recipe for a continuous path made from the trajectories of the particle system, killed at the first time it hits $\partial U$. 

We shall then define for each $1\leq i\leq N$ a Cadlag $\mathcal{C}^N$-valued process $\alpha^{N,i}_t$ which provides a recipe for the unique continuous path made from the trajectories of the particles finishing with $X^{N,i}_t$ at time t.

\begin{defin}[Set of Chains $\mathcal{C}^N$]
We define $\mathcal{C}^N$ to be the collection of all "Chains", which we define as follows:
\[
\begin{split}
\mathcal{C}^N=\{((j_{\ell},0),(j_{\ell-1},k_{\ell-1}),\ldots,(j_1,k_1),(j_0,k_0)):j_{\ell'}\in \{1,\ldots,N\}\text{ for }\ell'\leq \ell,\\ k_{\ell'}\in \mathbb{N},j_{\ell'}\neq j_{\ell'+1}\text{ for }\ell'<\ell\text{ and }0\leq \ell<\infty\}.
\end{split}
\]
\label{defin:chains}
Given $\alpha=((j_{\ell},0),(j_{\ell-1},k_{\ell-1}),\ldots,(j_0,k_0))\in \mathcal{C}^N$ we write $\lvert \alpha\rvert=\ell$ for the "length" of the chain. Thus $\alpha=((i,0))$ is defined to have length $\lvert \alpha\rvert=0$.
\end{defin}

We now construct $(X^{\alpha},W^{\alpha})$ for $\alpha\in\mathcal{C}^N$ as follows. We firstly define the Cadlag processes $(\mathcal{I}^{\alpha}_t,\Lambda^{\alpha}_t)_{0\leq t<\infty}$ for $\alpha=((j_{\ell},0),(j_{\ell-1},k_{\ell-1}),\ldots,(j_0,k_0))$:
\[
\begin{cases}
\text{Initial Condition:}\quad (\mathcal{I}^{\alpha}_0,\Lambda^{\alpha}_0)=(j_{\ell},\ell)\\
(\mathcal{I}^{\alpha}_t,\Lambda^{\alpha}_t):(j_{r},r)\mapsto (j_{r-1},r-1)\quad \text{if }t=\tau^{j_{r-1}}_{k_{r-1}}\text{ and }j_r=U^{j_{r-1}}_{k_{r-1}}\\
(\mathcal{I}^{\alpha}_t,\Lambda^{\alpha}_t)\quad\text{is constant otherwise.}
\end{cases}
\]
We then define:
\[
\begin{split}
X^{\alpha}_t=X^{\mathcal{I}^{\alpha}_t}_t,\quad 0\leq t<\tau^{\alpha}=\inf\{t>0: X^{\alpha}_{t-}\in \partial U\}\\
dW^{\alpha}_t:=dW^{\mathcal{I}^{\alpha}_t}_t,\quad 0\leq t<\infty,\quad W^{\alpha}_0=0.
\end{split}
\]

We see that $X^{\alpha}$ must satisfy the following SDE:
\begin{equation}
\begin{split}
dX^{\alpha}_t=b(m^N_t,X^{\alpha}_t)dt+dW^{\alpha}_t,\quad 0\leq t<\tau^{\alpha}=\inf\{t>0:X^{\alpha}_{t-}\in\partial U\}.
\end{split}
\label{eq:SDE for Xalpha}
\end{equation}
We now define the Dynamical Historical Process:

\begin{defin}[Dynamical Historical Processes]
Given $\alpha=((j_{\ell},0),(j_{\ell-1},k_{\ell-1}),\ldots,(j_1,k_1)) \in\mathcal{C}^N$ and $(j_0,k_0)$ with $j_0\neq j_1$ we set:
\[
\alpha\oplus (j_0,k_0)=((j_{\ell},0),(j_{\ell-1},k_{\ell-1}),\ldots,(j_1,k_1),(j_0,k_0)).
\]

We then define the $\mathcal{C}^N$-valued processes $\alpha^{N,i}_t$ ($i=1,\ldots,N$):
\begin{enumerate}
\item
At time 0 we define:
\[
\alpha^{N,i}_0=(i,0).
\]
\item
Between death times of $X^i$, $\alpha^{N,i}_t$ is constant:
\[
\alpha^{N,i}_t=\alpha^{N,i}_{\tau^i_k},\quad \tau^i_k\leq t<\tau^i_{k+1}.
\]
\item
At time $\tau^{i}_k$ if $U^i_k=j$ then we set:
\begin{equation}
\alpha^{N,i}_{\tau^i_k}=\alpha^{N,j}_{\tau^i_k-}\oplus (i,k).
\label{eq:formula for alpha N,i after jump}
\end{equation}
\end{enumerate}
Then we note that by construction $\tau^{\alpha^{N,i}_t}>t$. We may now define the Dynamical Historical Processes of $X^{N,1},\ldots,X^{N,N}$:
\begin{equation}
\mathcal{H}^{N,i,t}_s:=X^{\alpha^{N,i}_t}_s,\quad 0\leq s\leq t.
\end{equation}
\label{defin:Dynamical Historical Processes Equivalent Definition}
\end{defin}
We say that the DHP $\mathcal{H}^{N,i,t}$ follows particle $j$ at time $s$ if $\mathcal{I}^{\alpha^{N,i}_t}_s = j$. Thus in Figure \ref{fig:DHP example} the DHP $\mathcal{H}^{N,i,t}$ follows particle $j_3$ at time 0 and particle i at time t. We let $R^i_t \geq 0$ be the index of the most recent jump time of particle $i$:
\[
R^i_t = \max \{ k \geq 0\;|\; \tau^i_k \leq t \},
\]
with the convention that $\tau^i_0 = 0$ so that $\mathcal{H}^{N,i,t}_s$ follows particle i at time s for $s\in [\tau^i_{R^i_t}, t]$.

\subsection{Proof of Part \ref{enum:char-almost sure abs cty of m} of Lemma \ref{lem:char-almost sure abs cty}}

\sloppy Without loss of generality, suppose that $(m^N_t,J^N_t)_{0\leq t\leq T}$ converges in distribution on $D([0,T];\mathcal{P}_{\Wah}(U)\times \Rm_{\geq 0})$ to $(m_t,J_t)_{0\leq t\leq T} \in C([0,T];\mathcal{P}_{\Wah}(U)\times \Rm_{\geq 0})$, as $N \to \infty$ (or along a subsequence). We will write:
\begin{equation}
m=m_t\otimes dt,\quad dm=dm_tdt,\quad m^N=m^N_t\otimes dt,\quad dm^N=dm^N_tdt.
\end{equation}
Our goal is to show that, $\Pm$-almost surely, the random measure $m = m_t\otimes dt$ is absolutely continuous with respect to $\text{Leb}_{U \times [0,T]}$. 

For $\vec{h}=(h_1,\ldots,h_d)\in \Rm_{>0}^d$ and $\vec{x}=(x_1,\ldots,x_d)\in \Rm^d$ we define the rectangle:
\begin{equation}
R_{\vec{x}}(\vec{h})=(x_1-h_1,x_1+h_1)\times\ldots\times(x_d-h_d,x_d+h_d).
\label{eq:h rectangle}
\end{equation}
Define $\mathcal{A}=\{R_{\vec{x}}(\vec{h})\times [t_0,t_1]:t_0,t_1\in \Qm,\;0<t_0\leq t_1,\;\vec{x}\in \Qm^d,\;\vec{h}\in \Qm_{>0}^d\}$ and take $\mathcal{R}$ to be the set of finite unions of sets in $\mathcal{A}$ (note that $\mathcal{R}$ is a countable collection of sets). For $E\in\mathcal{B}(U\times (0,T])\setminus\{\emptyset\}$, define $T_{\min}(E)=\inf\{t:(x,t)\in E\}$. For $\rho\in\mathcal{R}$ we define 
\[
\rho_t:=\{x:(x,t)\in\rho\}.
\]

Our proof of the almost-sure absolute continuity of the random measure $m$ begins with the following two lemmas:

\begin{lem}
Fix $T<\infty$ and suppose that we have a random measure $m\in \mathcal{P}(U\times [0,T])$ defined on a probability space $(\Omega,\mathcal{F},\Pm)$ such that $m(U\times \{0\})=0$ holds $\Pm$-almost surely. We further assume that for every $\epsilon>0$ there exists a non-increasing function $C_{\epsilon}:(0,T]\ra \Rm_{\geq 0}$ such that
\begin{equation} \label{eq:supmC}
\expE\left[ 0\vee \sup_{\rho \in \mathcal{R}} \big( m(\rho) - C_{\epsilon}(T_{\min}(\rho))\text{Leb}(\rho) \big) \right] \leq \epsilon.
\end{equation}
Then $m<\!<\text{Leb}_{U\times [0,T]}$ holds $\Pm$-almost surely.
\label{lem:condition to prove for sets in R for abs cty lemma}
\end{lem}

The proof of Lemma~\ref{lem:condition to prove for sets in R for abs cty lemma} is given later in the appendix. We note that \eqref{eq:supmC} is a property of the law of the random measure $m$. Therefore, by Skorokhod's representation theorem, we could assume the convergence of $(m^N_t,J^N_t)_{0\leq t\leq T}$ to $(m_t,J_t)_{0\leq t\leq T}$ holds almost surely on a possibly different probability space $(\Omega^{a.s.},\mathcal{F}^{a.s.},\Pm^{a.s.})$. 

\begin{lem} \label{lem:mmN2}
Suppose that, on the probability space $(\Omega^{a.s.},\mathcal{F}^{a.s.},\Pm^{a.s.})$, some random variables $\{(m_t^N)_{0 \leq t \leq T}\}$ converge in $D([0,T]; \mathcal{P}_{\Wah}(U))$ as $N \to \infty$, $\Pm^{a.s.}$-almost surely, to $(m_t)_{0\leq t\leq T}\in C([0,T];\mathcal{P}_{\Wah}(U))$. Then, for all $\rho \in \mathcal{R}$, we $\Pm^{a.s.}$-almost surely have:
\[
\int_0^T m_t(\rho_t) dt \leq \liminf_{N \to \infty} \int_0^T m_t^N(\rho_t) \,dt.
\]

\end{lem}
{\bf Proof:} Since $(m_t)_{0\leq t\leq T}\in C([0,T];\mathcal{P}_{\Wah}(U))$, by assumption, we know that $(m^{N}_t)_{0\leq t\leq T}$ converges to $(m_t)_{0\leq t\leq T}$ with respect to the uniform (in $\Wah$) metric. So, $\Pm^{a.s.}$-almost surely we have 
\[
m_t(\rho_t) \leq \liminf_{N\ra\infty}m^N_t(\rho_t)
\]
for every $t > 0$ by the Portmenteau Theorem and the fact $\rho_t$ is an open set. From this fact and Fatou's lemma, we infer that, $\Pm^{a.s}$-almost surely,

\begin{align}
\int_0^Tm_t(\rho_t)dt  \leq \int_0^T \liminf_{N\ra\infty}m^N_t(\rho_t)dt \leq \liminf_{N\ra\infty}\int_{0}^Tm^N_t(\rho_t)dt. 
\end{align}
\hfill \qed

So, to verify the condition \eqref{eq:supmC} for the limit measure $m$, we turn our attention to estimating $m^N(\rho)$. Whereas Lemma \ref{lem:mmN2} requires almost-sure convergence, the construction we will use to obtain controls on $m^N(\rho)$ doesn't necessarily make sense on such a new probability space obtained with Skorokhod's representation theorem. We will therefore obtain controls on $m^N(\rho)$ working on our original filtered probability space $(\Omega,\mathcal{F},(\mathcal{F}_t)_{t\geq 0},\Pm)$. We will then transfer these controls to controls on the limit by way of Skorokhod's representation theorem and Lemma \ref{lem:mmN2}.

Working for the time being on $(\Omega,\mathcal{F},(\mathcal{F}_t)_{t\geq 0},\Pm)$, we now turn our attention to estimating:
\[
m^N(\rho)=\int_0^T m^N_t(\rho_t) \,dt = \frac{1}{N} \sum_{i=1}^N \int_0^T \Ind(X^{N,i}_t \in \rho_t )\,dt.
\]  

Estimating this quantity involves bounding the number of particles in a given set $\rho_t$ at time $t$. It is straightforward to do this with pure diffusions. In our system, however, the jumps make this estimate more difficult.

Recalling the definition of the Dynamical Historical Process $\mathcal{H}_s^{N,i} = X^{\alpha_t^{N,i}}_s$, for $s \in [0,t]$, we let $G^{\ell,n,i}_t$ be the event that 
\begin{equation}
\lvert \alpha^{N,i}_t\rvert \leq \ell
\label{eq:no more than l disc}
\end{equation}
and
\begin{equation}
J^{N,\mathcal{I}^{\alpha^{N,i}_t}_s}_s \leq n, \quad \forall \;\; s \in [0,t].
\label{eq:jumps by particles in the path bounded}
\end{equation}
The first condition says that the DHP makes no more than $\ell$ "transfers", and the second says that if the DHP $\mathcal{H}^{N,i,t}_s$ is following particle $j$ at time $s$, then particle $j$ has made no more than $n$ jumps up to time $s$. We recall that:
\begin{equation}
X^{N,i}_t=\mathcal{H}^{N,i,t}_t,\quad 0\leq t\leq T,\quad 1\leq i\leq N.
\label{eq:X equal to its historical process}
\end{equation}

Now we bound $m^N(\rho)$ by:
\begin{align}
m^N(\rho) & = \frac{1}{N} \sum_{i=1}^N \int_0^T \Ind(X^{N,i}_t \in \rho_t ) \Ind(G^{\ell,n,i}_t) \,dt + \frac{1}{N} \sum_{i=1}^N \int_0^T \Ind(X^{N,i}_t \in \rho_t ) \Ind((G^{\ell,n,i}_t)^C) dt \\
&  \leq \frac{1}{N} \sum_{i=1}^N \int_0^T \Ind(\mathcal{H}^{N,i,t}_t \in \rho_t ) \Ind(G^{\ell,n,i}_t) \,dt + \sup_{t \in [0,T]}  \frac{T}{N} \sum_{i=1}^N    \Ind((G^{\ell,n,i}_t)^C). \label{mN2sums}
\end{align}
Let us write $S_1^{N,\ell,n}(\rho)$ and $S_2^{N,\ell,n}$ for the two terms in \eqref{mN2sums}:
\[
S_1^{N,\ell,n}(\rho) = \frac{1}{N} \sum_{i=1}^N \int_0^T \Ind(\mathcal{H}^{N,i,t}_t \in \rho_t ) \Ind(G^{\ell,n,i}_t) \,dt, \quad \quad S_2^{N,\ell,n} = \sup_{t \in [0,T]}  \frac{T}{N} \sum_{i=1}^N    \Ind((G^{\ell,n,i}_t)^C).
\]
In particular, notice that $S_2^{N,\ell,n}$ does not depend on the set $\rho$. 

For $\ell,n\in \mathbb{N}$ fixed, we will show that there exists $C_{\ell,n}:(0,T]\ra \Rm_{\geq 0}$ a non-increasing function such that for all $\rho\in\mathcal{R}$ we have:
\begin{equation}\label{eq:mass of rho dominated by eta}
C_{\ell,n}(T_{\min}(\rho))\text{Leb}(\rho)\vee S_1^{N,\ell,n}(\rho) \overset{p}{\ra}C_{\ell,n}(T_{\min}(\rho))\text{Leb}(\rho)
\end{equation}
as $N \to \infty$. In addition to this, we will show that for any $\epsilon > 0$, we may choose $\ell = \ell(\epsilon)$ and $n = n(\epsilon)$ such that:
\begin{equation}
\limsup_{N\ra\infty}\expE[ S_2^{N,\ell,n} ] \leq \epsilon. \label{S2Neps}
\end{equation}
Clearly, the random variables $S_2^{N,\ell,n}$ are uniformly bounded: $|S_2^{N,\ell,n}| \leq T$. In particular, for fixed $\ell$ and $n$, the laws of $\{ S_2^{N,\ell,n}\}_{N \geq 2}$ are a tight family. Therefore, there is a random variable $G^\epsilon$ so that along a subsequence, $S_2^{N,\ell,n} \to G^\epsilon$ in distribution as $N \to \infty$. By \eqref{S2Neps}, $\expE[G^\epsilon] \leq \epsilon$ must hold. 

By the Skorokhod representation theorem, we may for fixed $\epsilon>0$ assume that both
\begin{equation}
(m^N_t)_{0\leq t\leq T}\ra (m_t)_{0\leq t\leq T}\quad\text{and}\quad S_2^{N,\ell,n}\ra G^{\epsilon} \label{PasmNG}
\end{equation}
hold almost surely (perhaps on a new probability space $(\Omega^{a.s.},\mathcal{F}^{a.s.},\Pm^{a.s.})$, which does not depend on $\rho$). From \eqref{mN2sums} and \eqref{eq:mass of rho dominated by eta} we have for any $\rho\in\mathcal{R}$ and $\delta>0$:
\begin{equation}
\Pm^{a.s.}(m^N(\rho) - S^{N,\ell,n}_2 \geq C_{\ell,n}\text{Leb}(\rho)+\delta)\ra 0\quad \text{as}\quad N\ra\infty.\label{eq:mNwedge C-S2}
\end{equation}
(The quantities in \eqref{eq:mNwedge C-S2} are all defined on the probability space $(\Omega^{a.s},\mathcal{F}^{a.s.},\Pm^{a.s.})$). Using Lemma \ref{lem:mmN2} and \eqref{PasmNG} we have:
\[
\liminf_{N\ra\infty}\big(m^N(\rho)-S^{N,\ell,n}_2\big)\geq m(\rho)-G^{\epsilon},\quad \Pm^{a.s.}\text{-almost surely.}
\]
Therefore for every $\rho\in\mathcal{R}$ and $\delta>0$, using \eqref{eq:mNwedge C-S2} and Fatou's lemma we have:
\begin{equation}
\begin{split}
    \Pm^{a.s.}(m(\rho) - G^\epsilon >C_{\ell,n}\text{Leb}(\rho)+\delta) \leq \liminf_{N \to \infty} \Pm^{a.s.}(m^N(\rho) - S^{N,\ell,n}_2 >C_{\ell,n}\text{Leb}(\rho) + \delta)=0.
\end{split}
\end{equation}
Therefore, since $\delta>0$ is arbitrary and $\mathcal{R}$ is countable, this implies:
\[
\sup_{\rho\in\mathcal{R}}\big(m(\rho)-C_{\ell,n}\text{Leb}(\rho)\big)\leq G^{\epsilon}\quad\Pm^{a.s.}\text{-almost surely.}
\]
We finally note that
\[
\expE^{\Pm^{a.s.}}\big[\sup_{\rho\in\mathcal{R}}\big(m(\rho) - C_{\ell,n}(T_{\min}(\rho))\text{Leb}(\rho)\big)\big]\leq \epsilon
\]
is a statement about the distribution of $m$, so must also hold true under $\Pm$. Except for the proof of \eqref{eq:mass of rho dominated by eta} and \eqref{S2Neps}, this establishes condition \eqref{eq:supmC} in Lemma~\ref{lem:condition to prove for sets in R for abs cty lemma} and completes the proof of 
of Part 1 of Lemma~\ref{lem:char-almost sure abs cty}. The rest of this section is devoted to the proofs of \eqref{eq:mass of rho dominated by eta} and \eqref{S2Neps}.

\subsubsection*{Proof of \eqref{eq:mass of rho dominated by eta}}

The following lemma will be a key tool for controlling the density of diffusions with bounded drift. We write $\vec{n}(\vec{x})$ ($\vec{x}\in \partial \Rm_{>0}^d$) for the inward normal of the positive orthant $\Rm_{>0}^d$ and consider strong solutions of the following SDE:
\begin{equation}
    dY_t=(-B,\ldots,-B)dt+\vec{n}(Y_t)dL^{Y}_t,\quad Y_0=0
\label{eq:SDE for Y}
\end{equation}
where $L^{Y}_t$ is the local time of $Y_t$ with the boundary $\partial \Rm^d_{>0}$. This is a normally reflected diffusion in $\Rm_{>0}^d$ with constant drift. 
\begin{lem}
Consider on some filtered probability space $(\Omega',\mathcal{F}',(\mathcal{F}'_t)_{t\geq 0},\Pm')$ the family of $\Rm^d$-valued weak solutions $(X^{\gamma},W^{\gamma})$ $(\gamma\in\Gamma)$ of the SDE:
\begin{equation}
    dX^{\gamma}_t=b^{\gamma}_tdt+dW^{\gamma}_t,\quad 0\leq t<\infty
    \label{eq:SDE for Xgamma}
\end{equation}
whereby $\lvert b^{\gamma}\rvert\leq B$ is $(\mathcal{F}'_t)_{t\geq 0}$-adapted. Then there exists on $(\Omega',\mathcal{F}',(\mathcal{F}'_t)_{t\geq 0},\Pm')$ a family of identically distributed strong solutions $(Y^{\gamma},\tilde{W}^{\gamma})$ to \eqref{eq:SDE for Y} which satisfy the following:
\begin{enumerate}
    \item $X^{\gamma}$ dominates $Y^{\gamma}$ so that:
    \begin{equation}
        X^{\gamma}_t\in R_{\vec{h}}(\vec{0})\Rightarrow Y^{\gamma}_t\in R_{\vec{h}}(\vec{0}),\quad t\geq 0,\; \vec{h}\in \Rm_{>0}^d
\label{eq:Ygamma dominated by Xgamma}
    \end{equation}
whereby $\vec{0}=(0,\ldots,0)$.
\item
We have explicit controls on the density of $Y^{\gamma}$ so that there exists $C:(0,\infty)\ra \Rm_{>0}$ non-increasing such that:
\begin{equation}
    \Pm(Y^{\gamma}_t\in R_{\vec{h}}(\vec{0}))\leq C_t\text{Leb}(R_{\vec{h}}(\vec{0})).
\label{eq:density of Ygamma controls}
\end{equation}
\item
For any event $A\in \mathcal{F}'_0$ and $\gamma_1,\gamma_2\in\Gamma$: if conditional upon the event A, $W^{\gamma_1}$ and $W^{\gamma_2}$ are conditionally independent, then so too are $Y^{\gamma_1}$ and $Y^{\gamma_2}$.
\end{enumerate}
\label{lem:dominated family of Ys}
\end{lem}

We will use Lemma \ref{lem:dominated family of Ys} in the Appendix to prove Lemma \ref{lem:upper and lower bounds on density of diffusions}, providing controls on the density of a diffusion for generic bounded drift, which shall be used throughout this paper.

We consider the possibilities for $\alpha^{N,i}_t$ given $G^{\ell,n,i}_t$. The condition that $J^{N,\mathcal{I}^{\alpha^{N,i}_t}}_s\leq n$ for all $0\leq s\leq t$ then allows us to see that for each transfer of the DHP from particle $j$ to particle $k$, $j$ is within the first $n$ particles $k$ jumps onto. Therefore to obtain all possibilities for $\alpha^{N,i}_t$ given $G^{\ell,n,i}_t$, it is sufficient to consider the first $n$ particles $i$ jumps onto, the first $n$ particles each of these children jumps onto, and repeating this $\ell$ times to obtain all possibilities for $\alpha^{N,i}_t$; these possibilities form a tree structure. We take $\Pi^{\ell,n}$ to be a perfect n-ary tree of length $\ell$ and construct a random injective function
\[
\hat{\alpha}^{N,\ell,n}_i:\Pi^{\ell,n}\ra \mathcal{C}^N
\]
with image $\mathcal{C}_i^{N,\ell,n}\subseteq \mathcal{C}^{N}$. This random function shall be such that
\begin{equation}
G^{\ell,n,i}_t\subseteq \{\alpha^{N,i}_t=\hat{\alpha}^{N,\ell,n}_i(v)\quad\text{for some}\quad v\in\Pi^{\ell,n}\}
\label{eq:relationship between Gt and hat alpha}
\end{equation}
and such that
\begin{equation}
\hat{\alpha}^{N,\ell,n}_i(v)\quad\text{is}\quad \sigma(U^i_k:1\leq i\leq N,\;k\geq 0)\text{-measurable.}
\end{equation}

We then define $\mathcal{T}_i^{N,\ell,n}$ to be the following $\{1,\ldots,N\}^{\Pi^{\ell,n}}$-valued random variable:
\[
\mathcal{T}_i^{N,\ell,n}(v)=j_r\quad\text{whereby}\quad\hat{\alpha}^{N,\ell,n}_i(v)=((j_r,0),\ldots).
\]
$\mathcal{T}_i^{N,\ell,n}$ assigns the root of $\Pi^{\ell,n}$ to $i$, assigns the $k^{\thh}$ child of the root to the $k^{\thh}$ particle i jumps onto, and so forth. We then define 
\[
\mathcal{G}_i^{N,\ell,n}=\text{Image}(\mathcal{T}_i^{N,\ell,n})
\]
to be the collection of all particles given by $\mathcal{T}_i^{N,\ell,n}$ at some branch of $\Pi^{\ell,n}$. Thus $\mathcal{G}_i^{N,\ell,n}$ is the collection of all particles which may be followed by $X^{\hat{\alpha}^{N,\ell,n}_i(v)}$ for some $v\in \Pi^{\ell,n}$.

We define a new filtered probability space $(\Omega,\mathcal{F},(\bar{\mathcal{F}}_t)_{t\geq 0},\Pm)$ given by the initial enlargement:
\begin{equation}
\bar{\mathcal{F}}_t=\mathcal{F}_t\wedge \sigma(U^i_k:1\leq i\leq N,\;k\geq 0).
\label{eq:initially enlarged filtration by jumps}
\end{equation}
We note the following:
\begin{enumerate}
    \item This new filtered probability space has the same sigma-algebra as our previous probability space $(\Omega,\mathcal{F},\Pm)$. Thus any random variable we define on this new sigma-algebra is defined on our previous probability space and vice-versa - only the adaptedness properties with respect to the filtration may change.
    \item Since $(\mathcal{F}_t)_{t\geq 0}$ is a subfiltration of $(\bar{\mathcal{F}}_t)_{t\geq 0}$ any $(\mathcal{F}_t)_{t\geq 0}$-adapted process is $(\bar{\mathcal{F}}_t)_{t\geq 0}$-adapted.
    \item
    The Brownian motion $W^i$ is independent of $\sigma(U^i_k:1\leq i\leq N,\;k\geq 0)$, hence an $(\bar{\mathcal{F}}_t)_{t\geq 0}$-Brownian motion. Moreover since $\hat{\alpha}^{N,\ell,n}_i(v)$ is $\bar{\mathcal{F}}_0$-measurable we have:
    \[
    W^{\hat{\alpha}^{N,\ell,n}_i(v)}_t:=\int_0^tdW^{\mathcal{I}^{\hat{\alpha}^{N,\ell,n}_i(v)}_s}_s,\quad 0\leq t<\infty,\quad W^{\hat{\alpha}^{N,\ell,n}_i(v)}_0=0
    \]
    is an $(\bar{\mathcal{F}}_t)_{t\geq 0}$-Brownian motion.
    \item 
    If the set of particles $\hat{\alpha}^{N,\ell,n}_i(v)$ follows is disjoint from those followed by $\hat{\alpha}^{N,\ell,n}_j(v')$ - hence if $\mathcal{G}^{N,\ell,n}_i\cap \mathcal{G}^{N,\ell,n}_j=\emptyset$ - then $W^{\hat{\alpha}^{N,\ell,n}_i(v)}$ and $W^{\hat{\alpha}^{N,\ell,n}_j(v')}$ have zero covariation. Therefore conditional on the event \[
    A^{N,\ell,n}_{i,j}=\{\mathcal{G}^{N,\ell,n}_i\cap \mathcal{G}^{N,\ell,n}_j=\emptyset\}\in \bar{\mathcal{F}}_0,
    \]
    $W^{\hat{\alpha}^{N,\ell,n}_i(v)}$ and $W^{\hat{\alpha}^{N,\ell,n}_j(v')}$ must be independent. 
\end{enumerate}

For every $\rho\in\mathcal{R}$ we fix a finite index set $\iota_{\rho}$ such that $\rho$ is given by the union:
\begin{equation}
\rho=\cup_{\beta\in \iota_{\rho}}[t^{\beta}_0,t^{\beta}_1]\times R_{\vec{h}_{\beta}}(\vec{x}_{\beta})
\label{eq:rho in terms of iota}
\end{equation}
whereby:
\begin{equation}
\sum_{\beta\in\iota_\rho}\text{Leb}([t^{\beta}_0,t^{\beta_1}]\times R_{\vec{h}_{\beta}}(\vec{x}_{\beta}))\leq 2\text{Leb}(\rho).
\label{eq:sum of Lebesgue at most twice Lebesgue of union}
\end{equation}

For each $1\leq i\leq N$, $v\in \Pi^{\ell,n}$, $\beta\in \iota_{\rho}$ we apply Lemma \ref{lem:dominated family of Ys} to $\{X^{\hat{\alpha}^{N,\ell,n}_i(v)}-\vec{x}_{\beta}\}$ to construct $Y^{i,v,\beta}_t$ and define:
\begin{equation}
\eta_i^{N,\ell,n,\rho}:=\sum_{\substack{\beta\in\iota_{\rho}\\ v\in \Pi^{\ell,n}}}\int_{t^{\beta}_0}^{t^{\beta}_1}\Ind(Y^{i,v,\beta}_t\in R_{\vec{h}_{\beta}}(\vec{0}))dt,\quad \rho\in\mathcal{R}.\label{eq:defin eta}
\end{equation}

We therefore have:
\[
\begin{split}
S_1^{N,l,n}(\rho)=\frac{1}{N}\sum_{i=1}^N\int_0^T\Ind(X^{\alpha^{N,i}_t}_t\in\rho_t) \Ind_{G^{\ell,n,i}_t}dt
\leq \frac{1}{N}\sum_{i=1}^N\sum_{\beta\in \iota_{\rho}}\int_{t^{\beta}_0}^{t^{\beta}_1}\Ind(X^{\alpha^{N,i}_t}_t\in R_{\vec{h}_{\beta}}(\vec{x}_{\beta})) \Ind_{G^{\ell,n,i}_t}dt
\\\underbrace{\leq}_{\text{by}\;\eqref{eq:relationship between Gt and hat alpha}} \frac{1}{N}\sum_{i=1}^N\sum_{\beta\in \iota_{\rho}}\frac{1}{N}\sum_{i=1}^N\sum_{v\in\Pi^{\ell,n}}\int_{t^{\beta}_0}^{t^{\beta}_1}\Ind(X^{\hat{\alpha}^{N,\ell,n}_i(v)}_t\in R_{\vec{h}_{\beta}}(\vec{x}_{\beta})) \Ind_{G^{\ell,n,i}_t}dt\\
\underbrace{\leq}_{\text{by}\;\eqref{eq:Ygamma dominated by Xgamma}}\frac{1}{N}\sum_{i=1}^N\sum_{\substack{\beta\in\iota_{\rho}\\ v\in \Pi^{\ell,n}}}\int_{t^{\beta}_0}^{t^{\beta}_1}\Ind(Y^{i,v,\beta}_t\in R_{\vec{h}_{\beta}}(\vec{0}))dt=\frac{1}{N}\sum_{i=1}^N\eta^{N,\ell,n,\rho}_i.
\end{split}
\]

We conclude our proof of \eqref{eq:mass of rho dominated by eta} by establishing the following lemma and verifying $\{\eta^{N,l,n,\rho}_i:1\leq i\leq N\}$ satisfies the conditions of this lemma with $M=C_{\ell,n}(T_{\min}(\rho))\text{Leb}(\rho)$:
\begin{lem}
Let $\{\gamma^{N}_k:1\leq k\leq N\in \mathbb{N}\}$ be a triangular array of random variables, and let $S_N=\sum_{k\leq N}\gamma^{N}_k$. We suppose that the $\gamma^{N}_k$ are uniformly bounded, that $\sup_{j\neq k}\text{Cov}(\gamma^{N}_j,\gamma^{N}_k)\ra 0$ as $N\ra \infty$, and that $\limsup_{N\ra\infty}\sup_{1\leq j\leq N}\expE[\gamma^N_j]\leq M$. Then we have $\frac{S_N}{N}\vee M\ra M$ in probability.
\label{lem:weak convergence lemma}
\end{lem}

\subsubsection*{Proof of Lemma \ref{lem:dominated family of Ys}}\label{proof:dominated family of Ys}

We firstly construct $Y^{\gamma},\tilde{W}^{\gamma}$ for $\gamma\in\Gamma$. We write $X^{\gamma,d'}_t$ for the $d'^{th}$ coordinate of $X^{\gamma,d'}_t$ for $1\leq d'\leq d$. We take the Doob-Meyer decomposition of $\lvert X^{\gamma,d'}_t\rvert$, obtaining it as the sum of a Brownian motion $\tilde{W}^{\gamma,d'}_{t}$, a drift ($\leq B$) term and a local time term up to the time $\tau^{\gamma}$. We then write
\[
\tilde{W}^{\gamma}_t=(\tilde{W}^{\gamma,1}_t,\ldots,\tilde{W}^{\gamma,d}_t)
\]
and continue $\tilde{W}^{\gamma}_{t}$ after the time $\tau^{\gamma}$ by setting $d\tilde{W}^{\gamma}_t=dW^{\gamma}_t$. It is then immediate that there exists an $(\mathcal{F}'_t)_{t\geq 0}$-adapted $d\times d$ signature matrix-valued process $K^{\gamma}_t$ such that $\tilde{W}^{\gamma}$ satisfies:
\begin{equation}
d\tilde{W}^{\gamma}_t=K^{\gamma}_tdW^{\gamma}_t.
\label{eq:Walpha tilde in terms of Walpha}
\end{equation}

Having constructed $(\mathcal{F}'_t)_{t\geq 0}$-Brownian motions $\tilde{W}^{\gamma,d'}$, we have $(\mathcal{F}'_t)_{t\geq 0}$- adapted strong solutions $(Y^{\gamma,d'},\tilde{W}^{\gamma,d'})$ of the following (which exists by \cite[Theorem 1.3]{Andres}):
\begin{equation}
\begin{split}
dY_t=d\tilde{W}_t-Bdt+dL^{Y}_t,\quad Y_0=0.
\end{split}
\label{eq:equation for Y}
\end{equation}

Thus $(Y^{\gamma},W^{\gamma})=((Y^{\gamma,1},\ldots,Y^{\gamma,d}),W^{\gamma})$ is a strong solution to \eqref{eq:SDE for Y}. Now we observe that for some $\lvert b^{\gamma,d'}_t\rvert \leq B$ we have:
\[
d(\lvert X^{\gamma,d'}_{t}\rvert - Y^{\gamma,d'}_t)=(B-b^{\gamma,d'}_t)dt+dL^{\lvert X^{\gamma,d'}\rvert}-dL^{Y^{\gamma,d'}}_t,\quad t<\tau^{\gamma}.
\]

Hence by the same proof that $\eta_t\geq D_t$ in the proof of Step \ref{enum:eta dominates D} of Proposition \ref{prop:dominated by Bessel fns up to time tauinfty wedge taustop} we have $\lvert X^{\gamma,d'}_{t}\rvert \geq Y^{\gamma,d'}_t$ for all $t<\tau^{\gamma}$. This immediately implies \eqref{eq:Ygamma dominated by Xgamma}.

We now control the expectation, showing that there exists $C:(0,T]\ra \Rm_{\geq 0}$ non-increasing such that for all $\vec{h}\in\Rm_{>0}^d$ and $\gamma\in\Gamma$ we have \eqref{eq:density of Ygamma controls}. We have \cite[Equation (1.1)]{Abate1987} an explicit expression for the cumulative density function of reflected Brownian motion with constant negative drift reflected at 0. Differentiating \cite[Equation (1.1)]{Abate1987} in y we have that for some $c<\infty$ the transition density satisfies:
\[
p_t(x,y)\leq \frac{c}{\sqrt{t}}.
\]
Therefore $\Pm(Y^{\gamma,d'}_t\in [0,h])\leq \frac{c}{\sqrt{t} }h$ for $t>0$, $h\geq 0$ and $1\leq d'\leq d$.

We use \eqref{eq:Walpha tilde in terms of Walpha} to see that $\tilde{W}^{\gamma,d_1}$ and $\tilde{W}^{\gamma,d_2}$ are pairwise independent Brownian motions for $d_1\neq d_2$ and hence jointly independent. Therefore $\{Y^{\gamma,d'}:1\leq d'\leq d\}$ are independent as they are measurable functions of independent Brownian motions. Thus we have:
\[
\Pm(Y^{\gamma}_t \in R_{\vec{h}}(\vec{0}))=\prod_{1\leq d'\leq d}\Pm(Y^{\gamma,d'}_t\in [0,h_{d'}])\leq \frac{c^d}{t^{\frac{d}{2}}}\text{Leb}(R_{\vec{h}}(\vec{0})).
\]

Finally we observe that for any event $A\in \mathcal{F}'_0$ and $\gamma_1,\gamma_2\in\Gamma$; if conditional upon the event A, $W^{\gamma_1}$ and $W^{\gamma}$ are conditionally independent; then they must have zero covariation Using \eqref{eq:Walpha tilde in terms of Walpha} we see that $\tilde{W}^{\gamma_1}$ and $\tilde{W}^{\gamma_2}$ must also have zero covariation, hence be conditionally independent. Therefore upon the event A, $Y^{\gamma_1}$ and $Y^{\gamma_2}$ are independent as they are measurable functions of independent Brownian motions.

\qed

\subsubsection*{Construction of $\hat{\alpha}^{N,l,n}_i$}
We define the random function $\hat{\alpha}^{N,l,n}_i$ by firstly defining its image:
\begin{defin}
We define $\mathcal{C}_i^{N,\ell,n}\subseteq \mathcal{C}^{N}$ as follows:
\begin{equation}
\begin{split}
\mathcal{C}^{N,\ell,n}_i=\{((j_{\ell'},0),(j_{\ell'-1},k_{\ell'-1}),\ldots, (j_1,k_1),(j_0,k_0))\in\mathcal{C}^{N}:\\
U^{j_r}_{k_r}=j_{r+1}\text{ whereby }k_r\leq n\text{ for all }r<\ell',\; \ell'\leq \ell,\; j_0=i\}.
\end{split}
\end{equation}
\end{defin}

We now parametrise the elements of $\mathcal{C}^{N,\ell,n}_i$ as follows. We define $\Pi^{\ell,n}$ to be a perfect n-ary tree of length $\ell$:
\begin{defin}[$\Pi^{\ell,n}$]
We define $\Pi^{\ell,n}$ to be a perfect n-ary tree of length $\ell$ (so that each leaf is of depth $\ell$ with the root defined to be of depth $0$). We adopt standard Ulam-Harris notation, writing $\emptyset$ for the root of $\Pi^{\ell,n}$, $(k_0)$ for the $k_0^\thh$ child of $\emptyset$ ($k_0\leq n$) and recursively defining $(k_0,\ldots,k_{r},k_{r+1})$ to be the $k_{r+1}^{\thh}$ child of $(k_0,\ldots,k_{r})$ (for $r\leq \ell-2$ and $k_{r}\leq n$).
\end{defin}
Note that the leaves of this tree terminate with an $\ell-1$ subscript: $(k_0,k_1,\ldots,k_{\ell-1})$. Then we see that the following random map is bijective:
\[
\begin{split}
    \iota^{N,\ell,n}_i:\mathcal{C}^{N,\ell,n}_i\ra \Pi^{\ell,n},\quad
    (i,0)\mapsto \emptyset\\
    ((j_{r},0),(j_{r-1},k_{r-1}),\ldots,(j_1,k_1), (i,k_0))\mapsto (k_0,k_1,\ldots,k_{r-1}),\quad 1\leq r\leq \ell.
\end{split}
\]
To see that $\iota^{N,\ell,n}_i$ is surjective, fix some $(k_0,\ldots,k_{\ell'-1})\in \Pi^{\ell,n}$ and recursively define $j_{r+1}=U^{j_r}_{k_r}$ ($r<\ell'$), $j_0=i$. Then we see $\iota^{N,\ell,n}_i(((j_{\ell'},0),(j_{\ell'-1},k_{\ell'-1}),\ldots,(j_1,k_1),(i,k_0)))=(k_0,k_1,\ldots,k_{\ell'-1})$ whereby $((j_{\ell'},0),(j_{\ell'-1},k_{\ell'-1}),\ldots,(j_1,k_1),(i,0))\in \mathcal{C}^{N,\ell,n}_i$. 

To see that $\iota^{N,\ell,n}_i$ is injective, suppose that $\iota^{N,\ell,n}_i(((j_{\ell'},0),(j_{\ell'-1},k_{\ell'-1}),\ldots,(j_1,k_1),(i,k_0)))=(k_0,k_1,\ldots,k_{\ell'-1})$ $(\ell'\leq \ell)$. Then we must have $j_1=U^{i}_{k_0}$ and $j_{r+1}=U^{j_r}_{k_r}$ for $r<\ell'$. This uniquely defines $((j_{\ell'},0),(j_{\ell'-1},k_{\ell'-1}),\ldots,(j_1,k_1),(i,k_0))$.

Thus we can take the inverse of $\iota^{N,\ell,n}_i$, parametrising the elements of $\mathcal{C}_i^{N,\ell,n}$ with $\Pi^{\ell,n}$:
\[
\begin{split}
    \Pi^{\ell,n}\ra \mathcal{C}^{N,\ell,n}_i,\quad
    v\mapsto \hat{\alpha}^{N,\ell,n}_i(v),\quad
    \hat{\alpha}^{N,\ell,n}_i(\emptyset)=(i,0)\\
    \hat{\alpha}^{N,\ell,n}_i((k_0,k_1,\ldots,k_{\ell'-1}))=((j_{\ell'},0),(j_{\ell'-1},k_{\ell'-1}),\ldots, (j_1,k_1),(i,k_0)).
\end{split}
\]
Therefore we have
\[
\hat{\alpha}^{N,\ell,n}_i:\Pi^{\ell,n}\ra \mathcal{C}^N
\]
is a random injection with image $\mathcal{C}^{N,\ell,n}_i$. The interpretation of this map can be seen from the following example:

\begin{figure}[H]
    \begin{tikzpicture}[scale=0.6, every node/.style={transform shape}]
			\draw[-,snake=snake,red](-2,0) -- ++(-2,2);
	\draw node at (-1.6,0) {$i$};
			\draw[->,snake=snake,thin,black](-2,2) -- ++(-2,2);
			\draw[-,snake=snake,ultra thick,black](0,0) -- ++(-2,2);
			\draw[->,snake=snake,ultra thick,red](-2,2) -- ++(6,3);			
			\draw[-,snake=snake,thin,green](3,0) -- ++(-1,3); 			\draw[->,snake=snake,green](2,3) -- ++(2,1);
			\draw[->,snake=snake,thin,blue](2,3) -- ++(0,3);
			\draw[->,snake=snake,thin,red](2,5) -- ++(-2,1);
			\draw[->,snake=snake,ultra thick,blue](1,0) -- ++(3,3);
			\draw[->,dashed,thin,red](-4,2) -- ++(2,0);
			\draw[->,dashed,thin,red](4,5) -- ++(-2,0);			\draw[->,dashed,blue](4,3) -- ++(-2,0);
			\draw node at (0.4,0) {$j_2$};
			\draw node at (1.8,0) {$j_4$};
			\draw node at (3.4,0) {$j_3$};
			\draw node at (-0.4,6) {$i$};			\draw[-,thick,black](4,0) -- ++(0,6); 
			\draw[-,thick,black](-4,0) -- ++(0,6); 	
	\draw node[draw,circle,inner sep=2pt,fill]  at (0,6){};
	\draw[->] (5.5,0) -- ++(0,6);
	\draw node at (5.75,5.5) {$t$};
    \end{tikzpicture}
    \caption{We have $\hat{\alpha}^{N,\ell,n}_i((1))=((j_2,0),(i,1))$ and $\hat{\alpha}^{N,\ell,n}_i((2))=((j_4,0),(i,2))$, hence $X^{\hat{\alpha}^{N,\ell,n}_i((1))}$ corresponds to the thick path on the left whilst $X^{\hat{\alpha}^{N,\ell,n}_i((2))}$ corresponds to the thick path on the right. Note that since $X^{j_4}$ dies before the second death time of $X^i$, its path does not actually include $X^i$.}
\end{figure}
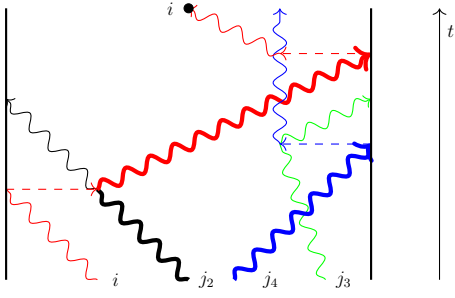

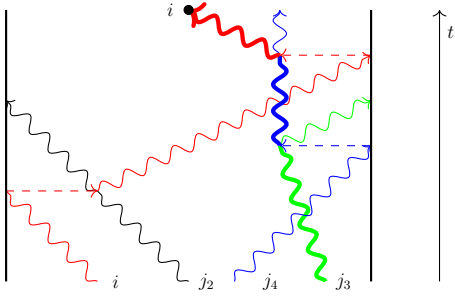
\begin{figure}[H]
    \begin{tikzpicture}[scale=0.6, every node/.style={transform shape}]
			\draw[-,snake=snake,thin,red](-2,0) -- ++(-2,2);
	\draw node at (-1.6,0) {$i$};
			\draw[->,snake=snake,thin,black](-2,2) -- ++(-2,2);
			\draw[-,snake=snake,thin,black](0,0) -- ++(-2,2);
			\draw[->,snake=snake,thin,red](-2,2) -- ++(6,3);	
			\draw[-,snake=snake,ultra thick,green](3,0) -- ++(-1,3); 			\draw[->,snake=snake,green](2,3) -- ++(2,1);
			\draw[-,snake=snake,ultra thick,blue](2,3) -- ++(0,2);
			\draw[->,snake=snake,thin,blue](2,5) -- ++(0,1);
			\draw[->,snake=snake,ultra thick,red](2,5) -- ++(-2,1);
			\draw[->,snake=snake,thin,blue](1,0) -- ++(3,3);
			\draw[->,dashed,thin,red](-4,2) -- ++(2,0);
			\draw[->,dashed,thin,red](4,5) -- ++(-2,0);			\draw[->,dashed,blue](4,3) -- ++(-2,0);
			\draw node at (0.4,0) {$j_2$};
			\draw node at (1.8,0) {$j_4$};
			\draw node at (3.4,0) {$j_3$};
			\draw node at (-0.4,6) {$i$};			\draw[-,thick,black](4,0) -- ++(0,6); 
			\draw[-,thick,black](-4,0) -- ++(0,6); 	
	\draw node[draw,circle,inner sep=2pt,fill]  at (0,6){};
	\draw[->] (5.5,0) -- ++(0,6);
	\draw node at (5.75,5.5) {$t$};
    \end{tikzpicture}
    \caption{We have $\hat{\alpha}^{N,\ell,n}_i((2,1))=((j_3,0),(j_4,1),(i,2))$, hence $X^{\hat{\alpha}^{N,\ell,n}_i((2,1))}$ corresponds to the thick path, and is the DHP $\mathcal{H}^{N,i,t}$.}
    \label{fig:figure for DHP}
\end{figure}

\subsubsection*{Proving and Verifying the Conditions of Lemma \ref{lem:weak convergence lemma}}

\begin{proof}[Proof of Lemma \ref{lem:weak convergence lemma}]
Clearly $\frac{S_N}{N}-\frac{1}{N}\sum_{j=1}^N\expE[\gamma^{N}_{j}]$ has zero expectation, so we now show it has variance converging to zero. Since the $\gamma^N_k$ are uniformly bounded, so are $\text{Var}(\gamma^N_k)$. We have:
\[
\begin{split}
\text{Var}\big(\frac{S_N}{N}-\frac{1}{N}\sum_{j=1}^N\expE[\gamma^{N}_{j}]\big)=\frac{1}{N^2}\sum_{k=1}^N\text{Var}(\gamma_k^N)+\sum_{j\neq k}\text{Cov}(\gamma^N_j,\gamma_k^N)\\
\leq \underbrace{\frac{N}{N^2}\sup_k\text{Var}(\gamma_k^N)}_{\ra 0}+\underbrace{\frac{N^2-N}{N^2}\sup_{j\neq  k}\text{Cov}(\gamma_j^N,\gamma_k^N)}_{\ra 0}\ra 0.
\end{split}
\]
Therefore $\frac{S_N}{N}-\frac{1}{N}\sum_{j=1}^N\expE[\gamma^{N}_{j}]\ra 0$ in probability. We fix $\epsilon>0$. Since $\limsup_{N\ra\infty}\sup_{1\leq j\leq N}\expE[\gamma^N_j]\leq M$ we have $\frac{S_N}{N}\vee M\ra M$ in probability as $N\ra\infty$.

\end{proof}
Clearly the $\eta_i^{N,\ell,n,\rho}$ are uniformly bounded in N, so it is sufficient to control the expectation and covariance as in Lemma \ref{lem:weak convergence lemma}. We do this using Lemma \ref{lem:dominated family of Ys}.

We start by controlling the expectation, using Tonelli's theorem and \eqref{eq:density of Ygamma controls} to see that we have $C_t$ non-increasing such that:
\[
\begin{split}
\expE[\eta^{N,\ell,n,\rho}_i]\leq \sum_{\substack{\beta\in\iota_{\rho}\\ v\in \Pi^{\ell,n}}}C_{t^{\beta}_0}(t^{\beta}_1-t^{\beta}_0)\text{Leb}(R_{\vec{h}_{\beta}}(\vec{0}))\\
\leq C_{T_{\min}(\rho)}\lvert \Pi^{\ell,n}\rvert \sum_{\beta\in \iota_{\rho}}\text{Leb}([t^{\beta}_0,t^{\beta}_1]\times R_{\vec{h}_{\beta}}(\vec{x}_{\beta})) \leq 2C_{T_{\min}(\rho)}\lvert \Pi^{\ell,n}\rvert \text{Leb}(\rho).
\end{split}
\]
We therefore define $C_{\ell,n}(t)=2\lvert \Pi^{\ell,n}\rvert C_t$ so that $\expE[\eta^{N,\ell,n,\rho}_i]\leq 2\lvert \Pi^{\ell,n}\rvert C_{T_{\min}(\rho)}$. We now seek to show that 
\[
\sup_{i\neq j}\text{Cov}(\eta_i^{N,\ell,n,\rho},\eta_j^{N,\ell,n,\rho})\ra 0\quad\text{as}\quad N\ra\infty.
\]
We recall that conditional on the event $A^{N,\ell,n}_{i,j}$ the Brownian motions $W^{\hat{\alpha}^{N,\ell,n}_i(v)}$ and $W^{\hat{\alpha}^{N,\ell,n}_j(v')}$ are independent. Thus using Lemma \ref{lem:dominated family of Ys}, conditional on the event $A^{N,\ell,n}_{i,j}$, $Y^{i,v,\beta_1}$ and $Y^{j,v',\beta_2}$ are independent for $\beta_1,\beta_2\in\iota_{\rho}$. Therefore it is sufficient to show that:
\[
\inf_{i\neq j}\Pm(A^{N,\ell,n}_{i,j})\ra 1\quad\text{as}\quad N\ra\infty.
\]

We calculate:
\[
\Pm(\mathcal{G}^{N,\ell,n}_i\cap\mathcal{G}^{N,\ell,n}_j\neq \emptyset)\leq \frac{\lvert\Pi^{\ell,n}\rvert^2}{N-1}\ra 0\quad \text{ as }\quad N\ra \infty.
\]
To see this, we see that the elements of $\mathcal{G}^{N,\ell,n}_i$ and $\mathcal{G}^{N,\ell,n}_j$ are chosen independently and uniformly at random, so that each element of $\mathcal{G}^{N,\ell,n}_j$ has a probability at most $\frac{\lvert \mathcal{G}^{N,\ell,n}_i\rvert}{N-1}\leq \frac{\lvert\Pi^{\ell,n}\rvert}{N-1}$ of being in $\mathcal{G}^{N,\ell,n}_i$. Therefore by a union bound we are done.

We have concluded our proof of \eqref{eq:mass of rho dominated by eta}.
\qed

\subsubsection*{Proof of \eqref{S2Neps}}

We recall $\tau^N_{\epsilon}$ is the stopping time defined in Proposition \ref{prop:bound on number of jumps by particle}, and $J^{N,i}_t$ is the number of jumps by particle $X^i$ in time t. We shall now bound the probability of $(G^{\ell,n,i}_t)^c$ by decomposing it into events $A_i^{N,\ell,n,\epsilon}$, $B_t^{N,i,\ell,\epsilon}$ and $\{\tau^N_{\epsilon}\leq T\}$:
\begin{align}
   A_i^{N,\ell,n,\epsilon} & =\cup_{v\in \Pi^{\ell,n}}\{J^{N,\mathcal{T}^{N,\ell,n}_i(v)}_{T\wedge\tau^N_{\epsilon}}\geq n+1\}\label{eq:event A},\\
B_t^{N,\ell,\epsilon,i} & =\{\lvert \alpha^{N,i}_{t\wedge \tau_{\epsilon}^N}\rvert \geq \ell+1\}\label{eq:event B}.
\end{align}

\subsubsection*{Step 1}

We begin by decomposing $(G^{\ell,n,i}_t)^c$ into the events
    \begin{equation}\label{eq:decomp of Gc event}
        (G^{\ell,n,i}_t)^c\subseteq A^{N,\ell,n,\epsilon}_i\cup B_t^{N,\ell,\epsilon,i}\cup \{\tau^N_{\epsilon}\leq T\}
    \end{equation}
none of which are dependent upon any choice of $\rho\in\mathcal{R}$ or $t\leq T$ and whereby $B_t^{N,\ell,\epsilon,i}$ is not dependent upon n.

We may decompose $(G^{l,n,i}_t)^c$:
\[
\begin{split}
\{(G^{\ell,n,i}_t)^c\}\subseteq \{\lvert \alpha^{N,i}_{t}\rvert\geq \ell+1\}
\cup\{ J^{N,\mathcal{I}^{N,\alpha^{N,i}_{t}}_s}_s> n\;\text{for some}\; s\in [0,t]\; \text{and}\; \lvert \alpha^{N,i}_{t}\rvert\leq \ell\}.
\end{split}
\]
Since Proposition \ref{prop:bound on number of jumps by particle} gives controls on the number of jumps only up to time $\tau^N_{\epsilon}$, it is necessary to localise up to time $\tau^N_{\epsilon}$:
\[
\begin{split}
\{(G^{\ell,n,i}_t)^c\}\subseteq \{\lvert \alpha^{N,i}_{t\wedge\tau^N_{\epsilon}}\rvert\geq \ell+1\}\cup \{\tau^N_{\epsilon}\leq T\}\\
\cup\{ J^{N,\mathcal{I}^{N,\alpha^{N,i}_{t}}_s}_{T\wedge \tau^N_{\epsilon}}> n\;\text{for some}\; s\in [0,t]\; \text{and}\; \lvert \alpha^{N,i}_{t}\rvert\leq \ell\}.
\end{split}
\]
Focusing on the third term on the right hand side, since $\lvert \alpha^{N,i}_{t}\rvert\leq l$ we can write $\alpha^{N,i}_{t}=((j_{\ell'},0),(j_{\ell'-1},k_{\ell'-1}),\ldots,(i,k_0))$ for some $\ell'\leq \ell$, so that we may take r minimal such that $J^{N,j_r}_{T\wedge\tau^N_{\epsilon}}> n$. Therefore $k_0,\ldots,k_{r-1}\leq n$ and $r\leq \ell$ so that we have:
\[
((j_r,0),(j_{r-1},k_{r-1}),\ldots,(i,k_0))\in \mathcal{C}^{\ell,n}_i.
\]
Thus $j_r=\mathcal{T}^{N,\ell,n}_i(v)$ for $v=(k_0,k_1,\ldots,k_{r-1})\in\Pi^{\ell,n}$. Therefore we have \eqref{eq:decomp of Gc event}:
\[
\begin{split}
\{(G^{\ell,n,i}_t)^c\}
\subseteq \{\lvert \alpha^{N,i}_{t\wedge\tau^N_{\epsilon}}\rvert\geq \ell+1\}\cup \cup_{v\in\Pi^{\ell,n}}\{J^{N,\mathcal{T}^{\ell,n}_i(v)}_{T\wedge\tau_{\epsilon}^N}\geq n+1\}\cup\{\tau^N_{\epsilon}\leq T\}.
\end{split}
\]

\subsubsection*{Step 2}

We now show that we may choose $\ell=\ell(\epsilon)$ large enough so that:
    \begin{equation}
    \limsup_{N\ra\infty}\expE^{\Pm}\big[\sup_{t\leq T}\frac{1}{N}\sum_{i=1}^N\Ind(B_t^{N,\ell,\epsilon,i})dt\big]\leq \epsilon. \label{eq:bounding B in Gc step}
    \end{equation}\label{enum:bounding B in Gc step}

There exists (by Proposition \ref{prop:bound on number of jumps by particle}) $\bar J<\infty$ such that 
\begin{equation}
\limsup_{N\ra\infty}\Pm(J_{T\wedge\tau^N_{\epsilon}}^N\geq \bar J)\leq \frac{\epsilon}{3}.
\label{eq:prob of exceeding bar J}
\end{equation}
We define $\mathcal{S}_{N}=\inf\{t:J_t^N\geq \bar J\}$ and $L^{N}_t:=\frac{1}{N}(1+\sum_i\lvert \alpha^{N,i}_{t}\rvert)$. We fix for the time being $1\leq i\leq N$. We see from \eqref{eq:formula for alpha N,i after jump} that if i jumps at time t, the expected value of $\lvert\alpha^{N,i}_{t}\rvert$ is at most 
\[
\frac{1}{N-1}\sum_{j\neq i}\lvert \alpha^{N,j}_{t-}\rvert+1=\frac{1}{N-1}\sum_{j\neq i}(\lvert \alpha^{N,j}_{t-}\rvert+1)\leq \frac{N}{N-1}L^N_{t-}.
\]
Moreover the length $\lvert\alpha^{N,i}_{t-}\rvert$ immediately prior to the jump must be non-negative, hence the expected increase in $\lvert\alpha^{N,i}_t\rvert$ at time t is at most $\frac{N}{N-1}L^{N}_{t-}$. Therefore the expected value of $L^N_t$ immediatly after the jump at time t is at most $\frac{N}{N-1}L^{N}_{t-}$. Further, the length of $\lvert \alpha^{N,j}_t\rvert$ does not change for $j\neq i$ and the $\lvert\alpha^{N,i}_t\rvert$ are bounded by $N(\bar J+1)+1$ up to time $\mathcal{S}_{N}$. Thus we see that 
\begin{equation}
\big(1+\frac{1}{N-1}\big)^{-NJ^N_{t\wedge\mathcal{S}_{N}}}L^{N}_{t\wedge\mathcal{S}_{N}}=\big(\frac{N}{N-1}\big)^{-NJ^N_{t\wedge\mathcal{S}_{N}}}L^{N}_{t\wedge\mathcal{S}_{N}}
\label{eq:supermartingale expression for L}
\end{equation}
is a supermartingale, which takes the value 1 at time 0. We now observe that
\[
\frac{\ell}{N}\sum_{i=1}^N\Ind(B^{N,\ell,n,\epsilon,i}_t)=\frac{\ell}{N}\sum_{i=1}^N\Ind(\lvert \alpha^{N,i}_{t\wedge\tau^N_{\epsilon}}\rvert\geq \ell+1)\leq L^{N}_{t\wedge \tau^N_{\epsilon}}.
\]
Thus, since \eqref{eq:supermartingale expression for L} is a supermartingale, we have for all $N$ and $t\leq T$:
\[
\begin{split}
\Pm\big(\sup_{t\leq T\wedge \mathcal{S}_N}\frac{1}{N}\sum_{i=1}^N\Ind(B^{N,\ell,n,\epsilon,i}_t)\geq \frac{\epsilon}{3}\big)\leq \Pm(\sup_{t\leq T}L^{N}_{t\wedge \tau^N_{\epsilon}\wedge\mathcal{S}_{N}}\geq \frac{\ell\epsilon}{3})\\
\leq\Pm\big(\sup_{t\leq T}\big(\frac{N}{N-1}\big)^{N(\bar J-J^N_{t\wedge\mathcal{S}_{N}\wedge\tau^N_{\epsilon}})}L^{N}_{t\wedge \tau^N_{\epsilon}\wedge\mathcal{S}_{N}}\geq \frac{\ell\epsilon}{3}\big)\\
=\Pm\big(\sup_{t\leq T}\big(\frac{N}{N-1}\big)^{-NJ^N_{t\wedge\mathcal{S}_{N}\wedge\tau^N_{\epsilon}}}L^{N}_{t\wedge \tau^N_{\epsilon}\wedge\mathcal{S}_{N}}\geq \frac{\ell\epsilon}{3}\big(\frac{N}{N-1}\big)^{-N\bar J}\big)\\
\leq \frac{3}{\epsilon \ell}\big(\frac{N}{N-1}\big)^{N\bar J}\leq \frac{3e^{2\bar J}}{\epsilon \ell}.
\end{split}
\]
Therefore for some $\ell=\ell(\epsilon)$ large enough we have for all N:
\[
\Pm\big(\sup_{t\leq T\wedge \mathcal{S}_N}\frac{1}{N}\sum_{i=1}^N\Ind(B^{N,\ell,n,\epsilon,i}_t)\geq \frac{\epsilon}{3}\big)\leq \frac{\epsilon}{3}.
\]
Combining this with \eqref{eq:prob of exceeding bar J} and observing that $\sup_{t\leq T\wedge \tau^N_{\epsilon}}\frac{1}{N}\sum_{i=1}^N\Ind(B^{N,l,n,\epsilon,i}_t)\leq 1$ we have \eqref{eq:bounding B in Gc step}.

\subsubsection*{Step 3}

Having fixed $\ell=\ell(\epsilon)$ we may choose $n=n(\epsilon)$ large enough such that we have \begin{equation}
        \limsup_{N\ra\infty}\expE^{\Pm}\big[\frac{1}{N}\sum_{i=1}^N\Ind(A^{N,\ell,n,\epsilon}_i)\big]\leq \epsilon .
    \label{eq:bounding A in Gc step}
    \end{equation}
    \label{enum:bounding A in Gc step}

We define the initial enlargement:
\begin{equation}
\mathcal{F}^{v}_t=\mathcal{F}_t\vee\sigma(U^{\mathcal{T}_i^{N,\ell,n}(v')}_k,\; k\geq 0,\; v'\in \Pi^{\ell,n,v}),\quad t\geq 0
\label{eq:filtration Fv}
\end{equation}
whereby we write $\Pi^{\ell,n,v}$ for $\Pi^{\ell,n}$ with all descendents of v removed (we remove v itself). We then observe that:
\begin{enumerate}
    \item $\mathcal{T}^{N,\ell,n}_i(v)$ is $\mathcal{F}^{v}_0$ measurable.
    \item Conditional upon $\mathcal{T}^{N,\ell,n}_i(v')\neq \mathcal{T}^{N,\ell,n}_i(v)$ for $v'\in \Pi^{\ell,n,v}$, the jumps $U^{\mathcal{T}^{N,\ell,n}_i(v)}_k$ are chosen independently and uniformly at random at the times $\tau^{\mathcal{T}^{N,\ell,n}_i(v)}_k$.
    \item $W^{\mathcal{T}^{N,\ell,n}_i(v)}$ is an $(\mathcal{F}^{v}_t)_{t\geq 0}$-Brownian motion as with the argument that $W^{\hat{\alpha}^{N,\ell,n}_i(v)}$ is an $(\bar{ \mathcal{F}}_t)_{t\geq 0}$-Brownian motion in the proof of \eqref{eq:mass of rho dominated by eta}.
\end{enumerate}

We fix for the time being $1\leq i\leq N$ and now work on $(\Omega,\mathcal{F},(\mathcal{F}^{v}_t)_{t\geq 0},\Pm)$. We see that with probability at most $\frac{\lvert \Pi^{\ell,n,v}\rvert}{N-1}\leq \frac{(n+1)^\ell}{N-1}$, $\mathcal{T}^{N,\ell,n}_i(v)\neq \mathcal{T}^{N,\ell,n}_t(v')$ for all $v'\in \Pi^{\ell,n,v}$. Otherwise $W^{\mathcal{T}^{N,\ell,n}_i(v)}$ is an $\mathbb{F}^{v}$-Brownian motion and $U^{\mathcal{T}_i^{N,\ell,n}(v)}_k$ $(k\geq 1)$ are chosen independently and uniformly at random at time $\tau^{\mathcal{T}^{N,\ell,n}_i(v)}_k$, so that we can repeat the argument of the proof of Proposition \ref{prop:bound on number of jumps by particle} in order to obtain:
\[
\begin{split}
\Pm(J^{\mathcal{T}^{N,\ell,n}_i(v)}_{T\wedge\tau_{\epsilon}}\geq n+1)\leq \Pm(J^{\mathcal{T}^{N,\ell,n}_i(v)}_{T\wedge\tau_{\epsilon}}\geq n+1\lvert \mathcal{T}_i(v)\neq \mathcal{T}^{N,\ell,n}_t(v')\text{ for }  v'\in \Pi^{\ell,n,v})\\
+\Pm(\mathcal{T}_i(v)= \mathcal{T}^{N,\ell,n}_t(v')\text{ for some }  v'\in \Pi^{\ell,n,v})
\leq M_{\epsilon}p_{\epsilon}^{\lfloor \frac{n+1}{M_{\epsilon}}\rfloor}+\frac{\lvert \Pi^{\ell,n}\rvert}{N-1}
\end{split}
\]
for some $0<p_{\epsilon}<1$ and $M_{\epsilon}<\infty$. Whereas we may have established this using a new filtration, our probability space $(\Omega,\mathcal{F},\Pm)$ has been kept fixed. Therefore we have:
\[
\Pm(A^{N,\ell,n,\epsilon}_i)\leq \lvert \Pi^{\ell,n}\rvert M_{\epsilon}p_{\epsilon}^{\lfloor \frac{n+1}{M_{\epsilon}}\rfloor}+\frac{\lvert \Pi^{\ell,n}\rvert^2}{N-1}.
\]
Thus we have (using Tonelli's theorem and that $\lvert \Pi^{\ell,n}\rvert$ grows polynomially in $n$ for fixed $\ell$):
\[
\limsup_{N\ra\infty}\expE[\frac{1}{N}\sum_{i=1}^N\Ind(A_i^{N,\ell,n,\epsilon})]\leq \lvert \Pi^{\ell,n}\rvert M_{\epsilon}p_{\epsilon}^{\lfloor \frac{n+1}{M_{\epsilon}}\rfloor}\ra 0\quad\text{as}\quad n\ra\infty.
\]
Having fixed $\ell=\ell(\epsilon)$ we may therefore choose $n=n(\epsilon)$ such that $\limsup_{N\ra\infty}\expE[\frac{1}{N}\sum_{i=1}^N\Ind(A_i^{N,\ell,n,\epsilon})]<\epsilon$, we have \eqref{eq:bounding A in Gc step}.

From \eqref{eq:bounding B in Gc step}, \eqref{eq:bounding A in Gc step} and Proposition \ref{prop:bound on number of jumps by particle} we may conclude that for all $\epsilon>0$ there exists $\ell=\ell(\epsilon),\; n=n(\epsilon)$ such that:
    \begin{equation}
        \limsup_{N\ra\infty}\expE^{\Pm}\big[\sup_{t\leq T}\frac{1}{N}\sum_{i=1}^N\Ind((G^{\ell,n,i}_t)^c)\big]\leq \epsilon .
        \label{eq:prob of Gc can be made small}
    \end{equation}
This completes the proof of \eqref{S2Neps} and therefore of Part \ref{enum:char-almost sure abs cty of m} of Lemma \ref{lem:char-almost sure abs cty}.

\subsection{Proof of Part \ref{enum:char-almost sure abs cty of marginals} of Lemma \ref{lem:char-almost sure abs cty}}

We may observe that the proof of Part \ref{enum:char-almost sure abs cty of m} may be repeated with $\mathcal{A}$ replaced by $\{(a^1,b^1)\times\ldots\times (a^d,b^d):a^i,b^i\in \Qm\}$, and $\mathcal{R}$ adjusted accordingly to obtain a proof of Part \ref{enum:char-almost sure abs cty of marginals}.

We have now concluded our proof of Lemma \ref{lem:char-almost sure abs cty}.
\qed

\subsection{Proof of Lemma \ref{lem:tight after bounded times}}
We now prove Lemma \ref{lem:tight after bounded times} using the machinery we constructed to prove Lemma \ref{lem:char-almost sure abs cty}. We take $R<\infty$ to be determined and write $F_R=B(0,R)^c$. As with \eqref{mN2sums} we have:
\[
\begin{split}
m^N_t(F_{R}) \leq \frac{1}{N} \sum_{i=1}^N  \Ind(\mathcal{H}^{N,i,t}_t\in F_R) \Ind_{G^{\ell,n,i}_t} +\frac{1}{N} \sum_{i=1}^N \Ind((G^{\ell,n,i}_t)^C).
\end{split}
\]
We then use \eqref{eq:relationship between Gt and hat alpha} to see that:
\[
\sup_{t\leq T}m_t^N(F_{R})\leq \frac{1}{N}\sum_{i=1}^N\sum_{v\in \Pi^{\ell,n}}\Ind\big(\sup_{t\leq T}\lvert X^{\hat{\alpha}^{N,\ell,n}_i(v)}\rvert\geq R\big)+S_2^{N,\ell,n}
\]
whereby we replace $T$ with $1$ in the definition of $S_2^{N,\ell,n}$. We now fix $\ell=\ell(\epsilon)$ and $n=n(\epsilon)$ as in \eqref{S2Neps} so that $\limsup_{N\ra\infty}\expE[ S_2^{N,\ell,n} ] \leq \epsilon$.

These are then random variables on the filtered probability space $(\Omega,\mathcal{F},(\bar{\mathcal{F}}_t)_{t\geq 0},\Pm)$ defined in \eqref{eq:initially enlarged filtration by jumps} with respect to which $W^{\hat{\alpha}^{N,\ell,n}_i(v)}$ is an $(\bar{\mathcal{F}}_t)_{t\geq 0}$-Brownian motion, $X^{\hat{\alpha}^{N,\ell,n}_i(v)}$ is adapted and is a solution of the SDE \eqref{eq:SDE for Xalpha}:
\[
\begin{split}
dX^{\hat{\alpha}^{N,\ell,n}_i(v)}_t=b(m^N_t,X^{\hat{\alpha}^{N,\ell,n}_i(v)}_t)dt+dW^{\alpha}_t,\quad 0\leq t<\tau^{\hat{\alpha}^{N,\ell,n}_i(v)}=\inf\{t>0:X^{\hat{\alpha}^{N,\ell,n}_i(v)}_{t-}\in\partial U\}.
\end{split}
\]

Using \eqref{eq:tight ics then mass stays within ball} and the fact the drift is bounded, we have:
\[
\limsup_{R\ra\infty}\limsup_{N\ra\infty}\Big[\frac{1}{N}\sum_{i=1}^N\sum_{v\in \Pi^{\ell,n}}\Ind\big(\sup_{t\leq T}\lvert X^{\hat{\alpha}^{N,\ell,n}_i(v)}\rvert\geq R\big)\Big]=0.
\]
Therefore we have:
\[
\limsup_{R\ra \infty}\limsup_{N\ra\infty}\expE[\sup_{t\leq T}m_t^N(F_{R})]\leq \epsilon.
\]
Since $\epsilon>0$ is arbitrary, we are done.

\qed

\section{Coupling to a Particle System on a Large but Bounded Subdomain}\label{section:coupling}

We construct here a coupling which will allow us in Section \ref{Section:Hydrodynamic Limit Theorem} to establish our hydrodynamic limit theorem on unbounded domains. We prove the following lemma in the Appendix:
\begin{lem}
Let $U\subseteq \Rm^d$ be a non-empty open domain with $C^{\infty}$ boundary $\partial U$. Then for every $R> R_{\min}:=\inf\{R'>0:B(0,R')\cap U\neq \emptyset\}$ there exists a non-empty open bounded domain $U_R$ with $C^{\infty}$ boundary such that $U\cap B(0,R)\subseteq U_R\subseteq U$.
\label{lem:Boundary Truncation Lemma}
\end{lem}

For all $R> R_{\min}$ we let $U_R$ be such a subdomain of U. Since $U_R$ is a smooth bounded domain there exists $r_R>0$ such that $U_R$ satisfies the interior ball condition with radius $r>r_R>0$: for every $x\in U_R$ there exists $y\in U_R$ such that $x\in B(y,r_R)\subseteq U_R$. 

Given the Fleming-Viot particle system $\vec{X}^N$ with McKean-Vlasov dynamics constructed on the filtered probability space $(\Omega^N,\mathcal{F}^N,(\mathcal{F}^N_t)_{t\geq 0},\Pm^N)$ and associated empirical measure valued processes:
\[
m_t^N=\frac{1}{N}\sum_i\delta_{X^{N,i}_t}
\]
we couple $\vec{X}^N$ on the enlarged filtered probability space $(\Omega^{N,R},\mathcal{F}^{N,R},(\mathcal{F}^N_t)_{t\geq 0},\Pm^{N,R})$ with another Fleming-Viot particle system with general dynamics $\vec{X}^{N,R}$ on the subdomain $U_R$ with drift $b^{N,R,i}_t=b(m^N_t,X^{N,R,i}_t)$.

In particular we show the following:
\begin{prop}
For $R>R_{\min}+1$ we may couple $\vec{X}^N$ with a Fleming-Viot $N$-particle system $\vec{X}^{N,R}$ with generalised dynamics on the probability space $(\Omega^{N,R},\mathcal{F}^{N,R},(\mathcal{F}^{N,R}_t)_{t\geq 0},\Pm^{N,R})$ with drift $b^{N,R,i}_t=b(m^N_t,X^{N,R,i}_t)$ on the domain $U_R$. That is between jumps $\vec{X}^{N,R,i}_t$ is a solution of the SDE:
\[
dX^{N,R,i}_t=b(m^N_t,X^{N,R,i}_t)dt+dW^{N,R,i}_t.
\]
Moreover each particle is killed when it hits the boundary $\partial U_R$ and chooses another particle $X^{N,R,j}$ independently and uniformly at random. Furthermore $(\vec{X}^N,\vec{W}^N_t)_{0\leq t<\infty}$ remains a Fleming-Viot particle system with McKean-Vlasov dynamics and the same drift on the filtered probability space $(\Omega^{N,R},\mathcal{F}^{N,R},(\mathcal{F}^{N,R}_t)_{t\geq 0},\Pm^{N,R})$. If we now define the empirical measure valued processes:
\[
m_t^{N,R}=\frac{1}{N}\sum_i\delta_{X^{N,R,i}_t},
\]
the jump processes:
\[
J^{N}_t=\#\{\text{jumps upto time t by }\vec{X}^N\},\quad J^{N,R}_t=\#\{\text{jumps upto time t by }\vec{X}^{N,R}\}
\]
and assume that $\{\mathcal{L}(m^N_0)\}$ is tight in $\mathcal{P}(\mathcal{P}_{\Wah}(U))$ then we have the following:
\begin{enumerate}
\item
$\{\mathcal{L}(m^{N,R}_0):N\in\mathbb{N}\}$ is tight in $\mathcal{P}(\mathcal{P}_{\Wah}(U_R))$.
\item 
For $T<\infty$ we have:
\begin{equation}
    \limsup_{N\ra\infty}\expE\Big[\sup_{t\leq T}\lvert\lvert m^{N}_t-m^{N,R}_t\rvert\rvert_{\TV}+1\wedge\sup_{t\leq T}\lvert J^N_t-J^{N,R}_t\rvert \Big]\ra 0\quad\text{as}\quad R\ra\infty.
    \label{eq:controls for coupling}
\end{equation}
    \end{enumerate}
    \label{prop:coupling to system on bounded domain}
\end{prop}

We shall firstly construct the coupling before establishing that this coupling satisfies \eqref{eq:controls for coupling}.

\subsection{Construction of the Coupling}

Since $N$ is fixed in this construction, we neglect the $N$ superscript for the sake of notation. We fix a point $x^{\ast}\in U\cap B(0,R-1)$. We then take a filtered probability space $(\tilde{\Omega},\tilde{\mathcal{F}},(\tilde{\mathcal{F}}_t)_{t\geq 0},\tilde{\Pm})$ on which are defined the jointly independent Brownian motions $(\tilde{W}^i_t)_{t\geq 0}$ ($i=1,\ldots,N$) and whereby $(\tilde{\mathcal{F}}_t)_{t\geq 0}$ is the natural filtration of the Brownian motions $\tilde{W}^i$. We then define a probability space $(\Omega^V,\mathcal{F}^V,\Pm^V)$ on which the jointly independent uniform $\{1,\ldots,N\}\setminus\{i\}$-valued random variables $V^i_k$ ($i=1,\ldots,N;\; k\geq 1$) are defined. We shall firstly define our construction on the measrurable space (which we shall later equip with the appropriate filtration and probability measure):
\[
(\Omega^R,\mathcal{F}^R)=(\Omega\times\tilde{\Omega}\times \Omega^V,\mathcal{F}\otimes\tilde{\mathcal{F}}\otimes \sigma^V).
\]

We shall partition $\{1,\ldots,N\}$ into "blue" indices $\mathcal{B}_t$ and "red" indices $\mathcal{R}_t$ at each time t - we shall say the particle $X^{i}/X^{R,i}$ is blue/red at time t if $i\in \mathcal{B}_t/\mathcal{R}_t$. We shall refer to particles in the particle system $\vec{X}^R$ as "R-particles". We will define the times $(\tau^R_k)_{k=0}^{\infty}$ corresponding to the $k^{\thh}$ death time of any of the R-particles. Our coupling is constructed up to time $\tau^R_k$, inductively in k. We proceed as follows:

\underline{Step 1}

At time 0 we assign indices $i\in \{1,\ldots,N\}$ to be blue if $\lvert X^{i}_0\rvert \leq R-1$ and otherwise red:
    \begin{equation}
    \mathcal{B}_0:=\{i:\lvert X^i_0\rvert\leq R-1\},\quad \mathcal{R}_0:=\{i:\lvert X^i_0\rvert > R-1\}.
    \label{eq:initial assignment for B and R coupling}
    \end{equation}
    Our initial condition is given by:
    \begin{equation}
X^{i}_0=\begin{cases}
X^{R,i}_0,\quad i\in \mathcal{B}_0 \\
x^*,\quad i\notin \mathcal{S}^{R}_0
\end{cases}.
\label{eq:initial condition for XR coupling}
\end{equation}
We have therefore defined $\vec{X}^R_t$, $\mathcal{B}_t$ and $\mathcal{R}_t$ for $t\leq \tau^R_0=0$.

\underline{Step 2}

We then proceed inductively. Having defined $0=\tau^R_0<\ldots<\tau^R_k$ and $\vec{X}^R_t$, $\mathcal{B}_t$ and $\mathcal{R}_t$ for $t\leq \tau^R_k$ we define $\tau^R_{k+1}$ and $\vec{X}^R_t$, $\mathcal{B}_t$ and $\mathcal{R}_t$ for $t\leq \tau^R_{k+1}$. We define:
\[
X^{R,i}_t=\begin{cases}
X^{i}_t,\quad i\in \mathcal{B}_{\tau^R_k}\\
X^{R,i}_{\tau^R_k}+\tilde{W}^i_t-\tilde{W}^i_{\tau^R_k},\quad i\in \mathcal{R}_{\tau^R_k}
\end{cases},\quad \tau^R_k:\leq t<\tau^R_{k+1}=\inf\{t>\tau^R_k:X^{R,i}_{t-}\in \partial U_R\}.
\]
That is the R-particles $X^{R,i}$ which are blue at time $\tau^R_k$ track the corresponding $X^i$, whilst those which are red track the path of the corresponding $(\tilde{F}_t)_{t\geq 0}$-Brownian motion $\tilde{W}^i_t$, up to the next time $\tau^R_{k=1}$ one of the R-particles the boundary $\partial U_R$. We define blue R-particles to remain blue and red R-particles to remain red in between hitting times, so that:
\[
\mathcal{B}_t:=\mathcal{B}_{\tau^R_k},\quad \mathcal{R}_t:=\mathcal{R}_{\tau^R_k},\quad \tau^R_k\leq t<\tau^R_{k+1}.
\]

\underline{Step 3}

We now define the construction at time $\tau^R_{k+1}$. It may be the case that two of the R-particles hit the boundary $\partial U_R$ at the same time (when we equip our construction with a probability measure this will turn out to be a null event), if this is the case we halt our construction at the time we call $\tau^R_{\ST}:=\tau^R_{k+1}$. 

Otherwise there is only one R-particle which hits the boundary at time $\tau^R_{k+1}$, with unique index $\ell(k+1)\in \{1,\ldots,N\}$ such that $X^{R,\ell(k+1)}_{\tau^R_{k+1}-}\in \partial U_R$. There are three distinct possibilities:
\begin{enumerate}
    \item It could be that $\ell(k+1)$ is red immediately prior to the hitting time. In this case the index $V^{\ell(k+1)}_{k+1}$ is chosen and $X^{R,\ell(k+1)}$ jumps onto $X^{R,V^{\ell(k+1)}_{k+1}}$:
    \[
    X^{R,\ell(k+1)}_{\tau^R_{k+1}}:=X^{R,V^{\ell(k+1)}_{k+1}}_{\tau^R_{k+1}-}.
    \]
    In this case $\ell(k+1)$ remains red: $\ell(k+1)\in \mathcal{R}_{\tau^R_{k+1}}$, and none of the other indices change colour:
    \[
    \mathcal{R}_{\tau^R_k}:=\mathcal{R}_{\tau^R_k-},\quad \mathcal{B}_{\tau^R_k}:=\mathcal{B}_{\tau^R_k-}.
    \]
    \item It could be the case that $\ell(k+1)$ is blue immediately prior to the hitting time $\tau^R_{k+1}$, and $X^{R,\ell(k+1)}$ hits $\partial U_R\setminus U$ at this time. Thus $X^{R,\ell(k+1)}$ was tracking $X^{\ell(k+1)}$ up to the hitting time, but of course $X^{\ell(k+1)}$ cannot jump at this time as it did not hit the boundary of U. In this case only the R-particle jumps, choosing the index $V^{\ell(k+1)}_{k+1}$ to jump onto, and the index $\ell(k+1)$ switches to red (none of the other indices switch colour):
    \[
    \begin{split}
    X^{R,\ell(k+1)}_{\tau^R_{k+1}}:=X^{R,V^{\ell(k+1)}_{k+1}}_{\tau^R_{k+1}-},\quad 
    \mathcal{R}_{\tau^R_k}:=\mathcal{R}_{\tau^R_k-}\cup \{\ell(k+1)\},\quad \mathcal{B}_{\tau^R_k}:=\mathcal{B}_{\tau^R_k-}\setminus \{\ell(k+1)\}.
    \end{split}
    \]
    \item
    The final possibility is that $\ell(k+1)$ is blue immediately prior to time $\tau^R_{k+1}$, at which time $X^{R,\ell(k+1)}$ hits $\partial U_R\cap U$. If this is the case $X^{\ell(k+1)}$ and $X^{R,\ell(k+1)}$ hit $\partial U\cap \partial U_R$ together, in which case they jump onto the same particle and remain blue (none of the other indices change colour):
    \[
     \begin{split}
    X^{R,\ell(k+1)}_{\tau^R_{k+1}}:=X^{\ell(k+1)}_{\tau^R_{k+1}},\quad
    \mathcal{R}_{\tau^R_k}:=\mathcal{R}_{\tau^R_k-},\quad \mathcal{B}_{\tau^R_k}:=\mathcal{B}_{\tau^R_k-}.
    \end{split}
    \]
\end{enumerate}
In all three cases none of the other R-particles jump at the time $\tau^R_{k+1}$.

\underline{Step 4}

This is well-defined on $(\Omega^R,\mathcal{F}^R)$ up to the time $\tau^R_{\WD}:=\tau^R_{\infty}\wedge\tau^R_{\ST}$ whereby:
\[
\tau^R_{\infty}=\lim_{k\ra\infty}\tau^R_k,\quad \tau^R_{\ST}=\inf\{t>0:\exists\; j\neq k\;\text{such that}\; X^{R,j}_{t-},\; X^{R,k}_{t-}\in\partial U_R\}.
\]
We now equip our measurable space with a filtration and probability measure. We define the filtration:
\[
\mathcal{F}^R_t:=\mathcal{F}_t\otimes \tilde{F}_t\otimes \sigma(V^{\ell(k)}_k:\tau^R_k\leq t),\quad 0\leq t<\infty
\]
and the probability measure:
\[
\hat{\Pm}=\Pm\otimes \tilde{\Pm}\otimes \Pm^V.
\]
We see that on the filtered probability space $(\Omega^R,\mathcal{F}^R,(\mathcal{F}^R_t)_{t\geq 0},\hat{\Pm})$ our original $N$-particle system $(\vec{X}^N,\vec{W}^N)$ has the same distribution, and moreover $\vec{X}^{N,R}$ is a generalised Fleming-Viot particle system with drift:
\[
b^{R,i}_t=\Ind(i\in \mathcal{B}^R_t)b(m^N,X^{R,i}_t).
\]
In particular between jumps $X^{R,i}_t$ satisfies:
\[
dX^{R,i}_t=\Ind(t\in \mathcal{B}_t)\big(b(m^N_t,X^{R,i}_t)dt+dW^{i}_t\big)+\Ind(t\notin \mathcal{R}_t)d\tilde{W}^{i}_t.
\]
Therefore by Proposition \ref{prop:general weak solns are global}, $\hat{\Pm}(\tau_{\WD}=\infty)=1$. Since the Brownian motions $\{\tilde{W}^i\}$ are independent of the Brownian motions $\{W^i\}$, we may use Girsanov's theorem to tilt the probability measure $\hat{\Pm}$, obtaining a probability measure $\Pm^R$ under which both:
\[
W^{R,i}_t:=\int_0^t\Ind(i\in \mathcal{B}_s)ds+\int_0^t\Ind(i\in\mathcal{R}_s)(d\tilde{W}^{R,i}_s-b(m^N,X^{R,i}_s)ds),\quad 1\leq i\leq N
\]
and $W^i$ ($1\leq i\leq N$) are $\Pm^R$-Brownian motions. Since $\hat{\Pm}$ and $\Pm^R$ are equivalent, $\{\tau_{\WD}<\infty\}$ remains a null event. By considering the covariation, we see $\{X^i\}$ and $\{X^{R,i}\}$ both remain families of independent Brownian motions (though not independent of each other). We have therefore finished our construction of the coupling.

\subsubsection*{$\{\mathcal{L}(m^{N,R}_0):N\in\mathbb{N}\}$ is tight in $\mathcal{P}(\mathcal{P}_{\Wah}(U_R))$.}

We note that a family of random measures being tight in $\mathcal{P}(\mathcal{P}_{\Wah}(U_R))$ is equivalent to their mean measures being tight in $\mathcal{P}(U_R)$ \cite[Theorem 4.10]{Kallenberg2017}. Using \eqref{eq:initial assignment for B and R coupling} and \eqref{eq:initial condition for XR coupling} we can write $m^{R,N}(A)=m^N(A\cap \bar B(0,R-1))+m^N((\bar B(0,R-1))^c)\delta_{x^*}(A)$. Therefore the expected mean measures $E^{N,R}(A)=\expE[m^{N,R}(A)]$ and $E^N(A)=\expE[m^N(A)]$ satisfy:
\[
E^{N,R}(A)\leq E^N(A\cap B(0,R-1))+\delta_{x^*}(A).
\]
Since $\{\mathcal{L}(m^N)\}$ is tight in $\mathcal{P}(\mathcal{P}_{\Wah}(U))$, $\{E^N\}$ is tight in $\mathcal{P}(U)$, so that $\{(A\mapsto E^N(A\cap B(0,R-1)))\}$ is tight in $U\cap B(0,R-1)\subseteq U_R$. Therefore $\{E^{N,R}\}$ is tight in $\mathcal{P}(U_R)$ hence $\{\mathcal{L}(m^{N,R}_0)\}$ is tight in $\mathcal{P}(\mathcal{P}_{\Wah}(U_R))$.

\subsection{Proof the Coupling Satisfies \eqref{eq:controls for coupling}}

Our proof is structured as follows:
\begin{enumerate}
    \item We firstly control the number of red particles so that:
\begin{equation}
    \limsup_{R\ra\infty}\limsup_{N\ra\infty}\expE[\sup_{t\leq T}\lvert\lvert m^{N,R}-m^N\rvert\rvert_{\TV}]\leq \limsup_{R\ra\infty}\limsup_{N\ra\infty}\expE\big[\frac{\lvert \mathcal{R}_{T}\rvert }{N}\big]=0.
    \label{eq:convergence of mNR minus mN}
\end{equation}\label{enum:bound on number of red balls}
\item By observing that jumps of blue particles (which stay blue) in our original system coincide with jumps of blue R-particles, we deduce that:
\begin{equation}
\limsup_{N\ra\infty}\expE[1\wedge\sup_{t\leq T}\lvert J^{N,R}-J^N\rvert]\ra 0\quad \text{as}\quad R\ra\infty.
\label{eq:expected diff in jump number goes to 0}
\end{equation}\label{enum:number of jumps in coupling proof step}
\end{enumerate}

\underline{Step \ref{enum:bound on number of red balls}}

We write $\tau^i_k$ for the $k^{\thh}$ death time of particle $X^{N,i}$ - at which time it jumps onto the particle with index $U^i_k$ - and $\tau_k$ for the death time of any of the particles $X^{N,j}$ (for any j). We observe that in order for i to go from blue to red at at time $\tau^i_k$, it is necessary to have at least one of the following:
\[
U^i_k\in \mathcal{R}_{\tau^i_k},\quad \lvert X^{N,i}\rvert\geq \frac{R}{3},\quad\lvert W^{N,i}_{\tau^i_{k+1}}-W^{N,i}_{\tau^i_k}\rvert\geq \frac{R}{3}\quad  \text{or} \quad BT\geq \frac{R}{3}.
\]
We may without loss of generality assume $R>3BT$. We define the initial enlargement:
\[
\mathcal{F}^{W}_t:=\mathcal{F}^{R}_t\wedge \sigma((W^{N,i}_t)_{0\leq t\leq T}:\; i=1,\ldots,N).
\]
We note that under this filtration, the $U^i_k$ are still chosen independently and uniformly at the corresponding hitting times. We then define the $(\mathcal{F}^{W}_t)_{0\leq t<\infty}$-adapted processes:
\[
\begin{split}
\tilde{ \mathcal{B}}_t=\{i\in \mathcal{B}_t:\sup_{t',t''\leq T}\lvert W^{N,i}_{t'}-W^{N,i}_{t''}\rvert\leq \frac{R}{10}\;\text{and}\; \lvert X^{N,i}_{\tau^i_k}\rvert\leq \frac{R}{10}\;\text{for all}\;\tau^i_k\leq t\}\subseteq \mathcal{B}_t,\\ \bar{\mathcal{R}}_t=(\tilde{\mathcal{B}}_t)^c\supseteq \mathcal{R}_t.
\end{split}
\]
Note that $\Ind(i\in \mathcal{B}_t)$ is constant on $[\tau^i_k,\tau^i_{k+1})$ for all $k\geq 0$ and non-increasing on $[0,\infty)$. These processes are progressively measurable with respect to $(\mathcal{F}^{W})_{t\geq 0}$. Moreover if $i\in \tilde{\mathcal{B}}_{\tau^i_k\wedge T}$, then it must be the case that:
\[
\lvert X^{N,i}_{\tau^i_{k+1}}\rvert\leq \lvert X^{N,i}_{\tau^i_{k}}\rvert+ \lvert W^{N,i}_{\tau^i_{k+1}}-W^{N,i}_{\tau^i_{k}}\rvert+B(\tau^i_{k+1}-\tau^i_k)\leq \frac{3R}{10},
\]
so that if $i\in \bar{\mathcal{R}}_{\tau^i_{k+1}\wedge T}$ then $\lvert X^{N,i}_{\tau^i_{k+1}}\rvert >\frac{R}{10}$ or $U^i_{k+1}\in \mathcal{R}_{\tau^i_{k+1}}\subseteq \bar{\mathcal{R}}_{\tau^i_{k+1}}$. This means that the probability of $\bar{\mathcal{R}}_t$ increasing by 1 at each hitting time $\tau_k$ is at most $\frac{N}{N-1}m^N_{\tau_k-}(B(0,\frac{R}{10})^c)+\frac{\bar{\mathcal{R}}_{\tau_k-}}{N-1}$. We define for $\epsilon'>0$ and $\bar J<\infty$ to be determined:
\[
\hat{\tau}=\inf\{t>0:J^N_t\geq \bar J\;\text{or}\; m^N_t(B(0,\frac{R}{10}
)^c)\geq \epsilon'\}\wedge T,\quad 
E_k=\frac{\expE[\lvert \bar{\mathcal{R}}_{\tau_k\wedge \hat{\tau}}\rvert]}{N}
\]
so that we have $E_{k+1}-E_k\leq \frac{\epsilon'+E_k}{N-1}$ and therefore:
\[
E_{N\bar J}\leq \epsilon'+E_{N\bar J}\leq (1+\frac{1}{N-1})^{N\bar J}(\epsilon'+E_0).
\]
We obtain:
\[
\frac{1}{N}\expE[\lvert \mathcal{R}_T\rvert ]\leq  (1+\frac{1}{N-1})^{N\bar J}(\epsilon'+E_0)+\Pm(J^N_{T}\geq \bar J)+\Pm(\sup_{t\leq T}m^N_t(B(0,\frac{R}{10})^c)\geq \epsilon').
\]
Therefore we have:
\[
\begin{split}
\limsup_{R\ra\infty}\limsup_{N\ra\infty}\frac{1}{N}\expE[\vert\mathcal{R}_T\rvert ]
\leq e^{\bar J}\limsup_{R\ra\infty}\limsup_{N\ra\infty}(\epsilon'+E_0)\\
+\limsup_{R\ra\infty}\limsup_{N\ra\infty}\big(\Pm(J_T^N\geq \bar J)+\Pm(\sup_{t\leq T}m^N_t(B(0,\frac{R}{10})^c)\geq \epsilon')\big).
\end{split}
\]
Using Lemma \ref{lem:tight after bounded times} we see that $\limsup_{R\ra\infty}\limsup_{N\ra\infty}\Pm(\sup_{t\leq T}m^N_t(B(0,\frac{R}{10})^c))=0$. Using also \eqref{eq:initial assignment for B and R coupling} and the fact $\sup_{t',t''\leq T}\lvert W_{t'}-W_{t''}\rvert<\infty$ (for Brownian motions $W_t$), we see that $\limsup_{R\ra\infty}\limsup_{N\ra\infty}E_0=0$. From this we conclude:
\[
\begin{split}
\limsup_{R\ra\infty}\limsup_{N\ra\infty}\frac{1}{N}\expE[\vert\mathcal{R}_T\rvert ]
\leq e^{\bar J}\epsilon'
+\limsup_{N\ra\infty}\Pm(J_T^N\geq \bar J).
\end{split}
\]

We now fix $\epsilon>0$ and use Proposition \ref{prop:bound on number of jumps by particle} to take $\bar J<\infty$ such that $\limsup_{N\ra\infty}\Pm(J^N_T\geq \bar J)<\epsilon$. Then taking $0<\epsilon'<\frac{\epsilon}{e^{\bar J}}$ we have $\limsup_{R\ra\infty}\limsup_{N\ra\infty}\frac{1}{N}\expE[\lvert \mathcal{R}_T\rvert]\leq 2\epsilon$ for arbitrary $\epsilon>0$ hence we have \eqref{eq:convergence of mNR minus mN}.

\underline{Step \ref{enum:number of jumps in coupling proof step}}

We return to our original filtered probability space $(\Omega^{N,R},\mathcal{F}^{N,R},(\mathcal{F}^{N,R}_t)_{t\geq 0},\Pm^{N,R})$ and seek to establish \eqref{eq:expected diff in jump number goes to 0}. The idea is that while particle $X^{N,i}$ and $X^{N,R,i}$ are blue, jumps of one are jumps of the other. Therefore we only need to count the jumps once i turns red. For $i\in \{1,\ldots,N\}$ we define the stopping times at which a given index becomes red:
\[
\tau^{\mathcal{R}}_i:=\inf\{t>0:i\in \mathcal{R}_t\}
\]
so that we have:
\[
\sup_{t\leq T}\lvert J^{N,R}_t-J^N_t\rvert\leq \frac{1}{N}\sum_{i=1}^N\big(\lvert J^{N,R,i}_T-J^{N,R,i}_{\tau^{\mathcal{R}}_i\wedge T}\rvert+\lvert J^{N,i}_T-J^{N,i}_{\tau^{\mathcal{R}}_i\wedge T}\rvert\big).
\]

We fix $\epsilon>0$ and write for $c>0$ to be determined as in Proposition \ref{prop:bound on mass near bdy}:
\[
V_c=\{x\in U:d(x,\partial U)\geq c\},\quad V^R_{c}=\{x\in \partial U_R:d(x,\partial U_R)\geq c\}.
\]
We define the stopping times:
\[
\tau^{N}_{c}=\inf\{t>0:m^{N}(V_c)\leq \frac{1}{2}\},\quad\tau^{N,R}_{c}=\inf\{t>0:m^{N,R}(K^R_c)\leq \frac{1}{2}\},\quad \mathcal{T}_c=\tau^N_c\wedge \tau^{N,R}_c.
\]
Proposition \ref{prop:bound on mass near bdy} gives $\tilde{c}=\tilde{c}(\frac{1}{10},\epsilon)>0$ such that
\[
\limsup_{N\ra\infty}\Pm(\sup_{t\in [0,T]}m^N_t(V_{\tilde{c}}^c)\geq \frac{1}{10})<\epsilon.
\]
so that applying \eqref{eq:convergence of mNR minus mN} we have:
\[
\limsup_{R\ra\infty}\limsup_{N\ra\infty}\Pm(\sup_{t\in [0,T]}m^{N,R}_t((V_{\tilde{c}})^c)\geq \frac{2}{10})\leq 2\epsilon.
\]
Therefore $\limsup_{R\ra\infty}\limsup_{N\ra\infty}\Pm(\mathcal{T}_{\tilde{c}}< T)\leq 3\epsilon$. We therefore have:
\begin{equation}
\begin{split}
\expE\big[1\wedge\sup_{t\leq T}\lvert J^{N,R}_t-J^N_t\rvert\big]\leq \Pm(\mathcal{T}_{\tilde{c}}<T)
\\
+\frac{1}{N}\sum_{i=1}^N\expE\big[\lvert J^{N,R,i}_{T\wedge \mathcal{T}_{\tilde{c}}}-J^{N,R,i}_{\tau^{\mathcal{R}}_i\wedge T\wedge \mathcal{T}_{\tilde{c}}}\rvert+\lvert J^{N,i}_{T\wedge \mathcal{T}_{\tilde{c}}}-J^{N,i}_{\tau^{\mathcal{R}}_i\wedge T\wedge \mathcal{T}_{\tilde{c}}}\rvert\big\lvert \tau^{\mathcal{R}}_i<T\wedge\mathcal{T}_{\tilde{c}}\big]\Pm(\tau^{\mathcal{R}}_i<T\wedge\mathcal{T}_{\tilde{c}}).
\end{split}
\label{eq:decom of number of jumps based on condition taui}
\end{equation}

We note that $(\vec{X}^N_{\tau^{\mathcal{R}}_i+t})_{t\geq 0}$ and $(\vec{X}^{N,R}_{\tau^{\mathcal{R}}_i+t})_{t\geq 0}$ are Fleming-Viot particle systems with generalised dynamics. We recall from \eqref{eq:formula for stopping time} that the stopping time $\tau^N_{\epsilon}$ given in Proposition \ref{prop:bound on number of jumps by particle} is given by $\inf\{t>0:m^N_t(V_{\tilde{c}(\frac{1}{2},\epsilon)})< \frac{1}{2}\}$. Moreover the constants $M_{\epsilon}$ and $p_{\epsilon}$ obtained in that proof were dependent only upon the upper bound on the drift $B<\infty$, and the value of $\tilde{c}(\frac{1}{2},\epsilon)$. We therefore see that we may apply (the proof of) Proposition \ref{prop:bound on number of jumps by particle} to see that there exists $C_{\epsilon}<\infty$ dependent only upon $\tilde{c}(\frac{1}{10},\epsilon)$ such that:
\[
\expE\big[\lvert J^{N,R,i}_{T\wedge \mathcal{T}_{\tilde{c}}}-J^{N,R,i}_{\tau^{\mathcal{R}}_i\wedge T\wedge \mathcal{T}_{\tilde{c}}}\rvert+\lvert J^{N,i}_{T\wedge \mathcal{T}_{\tilde{c}}}-J^{N,i}_{\tau^{\mathcal{R}}_i\wedge T\wedge \mathcal{T}_{\tilde{c}}}\rvert\big\lvert \tau^{\mathcal{R}}_i<T\wedge\mathcal{T}_{\tilde{c}}\big]\leq C_{\epsilon}.
\]
We therefore have:
\[
\begin{split}
\expE\big[1\wedge\sup_{t\leq T}\lvert J^{N,R}_t-J^N_t\rvert\big]\leq \frac{C_{\epsilon}}{N}\expE[\mathcal{R}_T]+\Pm(\mathcal{T}_{\tilde{c}}<T).
\end{split}
\]
Taking $\limsup_{R\ra\infty}\limsup_{N\ra\infty}$ of both sides, using \eqref{eq:convergence of mNR minus mN} and noting $\epsilon>0$ was arbitrary we are done.

\section{Hydrodynamic Limit Theorem}\label{Section:Hydrodynamic Limit Theorem}

In this section we shall establish Theorem \ref{theo:N ra infty theorem}. We shall then prove the uniqueness in law of weak solutions to the McKean-Vlasov SDE \eqref{eq:MKV sde}, before combining this with Theorem \ref{theo:N ra infty theorem} to prove Theorem \ref{theo:Hydrodynamic Limit Theorem} along with the existence part of Proposition \ref{prop:Properties of the McKean-Vlasov Process} - completing its proof. 

However the proof of Theorem \ref{theo:N ra infty theorem} relies on Lemma \ref{lem:MKV Solns compact for compact set of ics}, which provides a partial result for Proposition \ref{prop:Properties of the McKean-Vlasov Process} along with compactness for families of global weak solutions to the McKean-Vlasov SDE \eqref{eq:MKV sde} whose initial conditions belong to a compact set. This lemma will also be used in Section \ref{section:Further Estimates}. Therefore we firstly prove Lemma \ref{lem:MKV Solns compact for compact set of ics}.

Throughout this section we assume Condition \ref{cond:blip}. For $\kappa\subseteq \mathcal{P}_{\Wah}(U)$ we define:
\begin{equation}
\begin{split}
\Xi(\kappa)=\{(\mathcal{L}(X_t\lvert \tau>t), -\ln\Pm(\tau>t))_{0\leq t<\infty}\in C([0,\infty);\mathcal{P}_{\Wah}(U)\times \Rm_{\geq 0}):\\
 (X,\tau,W)\text{ is a global weak solution of }\eqref{eq:MKV sde}\text{ with initial condition }\mathcal{L}(X_0)\in\kappa\}
\end{split}
\label{eq:xiT(kappa)}
\end{equation}
which we equip with the metric $d^{D}$. Therefore \eqref{eq:xiinfty} is given by $\Xi=\Xi(\mathcal{P}_{\Wah}(U))$.

\begin{lem}
Every weak solution $(X,\tau,W)$ to \eqref{eq:MKV sde} is a global weak solution such that:
\[
(\mathcal{L}(X_t\lvert \tau>t))_{0\leq t<\infty}\in C([0,\infty);\mathcal{P}_{\Wah}(U))\quad\text{and}\quad(\Pm(\tau>t))_{0\leq t<\infty}\in C([0,\infty);\Rm_{> 0}).
\]
Moreover $\Xi(\kappa)$ is a compact subset of $(C([0,\infty);\mathcal{P}_{\Wah}(U)\times \Rm_{\geq 0}),d^{\infty})$ for $\kappa\subseteq \mathcal{P}_{\Wah}(U)$ compact. 
\label{lem:MKV Solns compact for compact set of ics}
\end{lem}

Note that Lemma \ref{lem:MKV Solns compact for compact set of ics} allows for the possibility that $\Xi(\kappa)$ is the compact set $\emptyset$.

\subsection{Proof of Lemma \ref{lem:MKV Solns compact for compact set of ics}}\label{Section:Properties of the McKean-Vlasov Process}

We begin by showing that every weak solution $(X,\tau,W)$ to \eqref{eq:MKV sde} is a global weak solution with $\Wah$-continuous in time laws. We suppose $(X_t,W_t)_{0\leq t\leq\tau}$ is a weak solution to \eqref{eq:MKV sde}. Lemma \ref{lem:upper and lower bounds on density of diffusions}, which we establish in the appendix (Lemma \ref{lem:upper and lower bounds on density of diffusions}), automatically implies that $(X,\tau,W)$ is a global weak solution.

We now turn to proving that global weak solutions to \eqref{eq:MKV sde} have $\Wah$-continuous in time conditional laws. We fix some weak solution $(X,\tau,W)$ of \eqref{eq:MKV sde} (which is a global weak solution) and $\phi\in C_b(U)$. Corollary \ref{cor:hitting time Corollary} (established in the appendix) gives that $\Pm(\tau=t)=0$ for $t\geq 0$. Therefore we have:

\[
\lim_{t'\uparrow t}\phi(X_{t'})\Ind(\tau>t')=\lim_{t'\downarrow t}\phi(X_{t'})\Ind(\tau>t')=\phi(X_t)\Ind(\tau>t)\quad\text{almost surely.}
\]
Thus $\mathcal{L}(X_{t'})\ra \mathcal{L}(X_{t})$ in $\mathcal{M}(U)$ as $t'\ra t$. Moreover since $\Pm(\tau=t)=0$, $\Pm(\tau>t')\ra \Pm(\tau>t)$ as $t'\ra t$. Therefore we have:
\[
\mathcal{L}(X_t\lvert\tau>t)\in C([0,\infty);\mathcal{P}(U)),\quad \Pm(\tau>t)\in C([0,\infty);\Rm_{> 0})
\]
where $\mathcal{P}(U)$ is equipped with the topology of weak convergence of probability measures. Since $\Wah$ generates this same topology, we have $(\mathcal{L}(X_t\lvert \tau>t))_{0\leq t<\infty}\in C([0,\infty);\mathcal{P}_{\Wah}(U))$.

\subsubsection*{Compactness of $\Xi(\kappa)$}

We now turn to establishing that $\Xi(\kappa)$ is a compact subset of $(C([0,\infty);\mathcal{P}_{\Wah}(U)\times \Rm_{\geq 0},d^{\infty})$ for $\kappa\subseteq \mathcal{P}_{\Wah}(U)$ compact. Since the empty set is compact, we may assume without loss of generality that $\Xi(\kappa)\neq\emptyset$. We take $(X^k,\tau^k,W^k)$ a sequence of global weak solutions to \eqref{eq:MKV sde} with initial conditions $\mathcal{L}(X^k_0)\in \kappa$.

\begin{enumerate}
    \item 
    We define:
    \[
    F^k_t:=\int_0^tb(\mathcal{L}(X^k_s\lvert \tau^k>s),X_s^k)ds
    \]
    and the metric of uniform convergence on compact intervals of time:
    \[
    \begin{split}
    \bar d_{\infty}((x^1_t,f^1_t,w^1_t)_{0\leq t<\infty},(x^2_t,f^2_t,w^2_t)_{0\leq t<\infty})\\= \sum_{n=1}^{\infty}2^{-n}(\sup_{t\leq n}(\lvert x^1_t-x^2_t\rvert+\lvert f^1_t-f^2_t\rvert+\lvert w^1_t-w^2_t\rvert)\wedge 1).
    \end{split}
    \]
    We establish that $\{\mathcal{L}((X^k_{t\wedge \tau^k},F^k_{t\wedge \tau^k},W^k_{t\wedge \tau^k})_{0\leq t<\infty})\}$ is tight in $\mathcal{P}((C([0,\infty);\bar U\times \Rm^d\times \Rm^d),\bar d_{\infty}))$. 
    \label{enum:xi compact X W tight}
    \item 
    We equip $[0,\infty]$ with the topology given by the one-point compactification of $[0,\infty)$, metrised with the metric $d_{[0,\infty]}(x,y)=\lvert \frac{1}{x+1}-\frac{1}{y+1}\rvert$. Then $\{\mathcal{L}(\tau^k)\}$ must be tight after compactification, hence the joint laws are tight, so that $\{\mathcal{L}(((X^k_{t\wedge \tau^k},F^k_{t\wedge \tau^k},W^k_{t})_{0\leq t<\infty},\tau^k))\}$ is tight in $\mathcal{P}((C([0,\infty);\bar U\times\Rm^d\times \Rm^d),\bar d_{\infty})\times ([0,\infty],d_{[0,\infty]}))$. 
    
    We consider any convergent in distribution subsequential limit:
    \[
    ((X_{t\wedge \tau},F_{t\wedge \tau},W_{t\wedge \tau})_{0\leq t<\infty},\tau)
    \]
    so that on some new probability space $(\Omega',\mathcal{F}',\Pm')$ we have $\Pm'$-almost sure convergence in $\bar d^{\infty}$ by Skorokhod's representation theorem. Having almost sure convergence (rather than convergence in distribution) shall become useful in Step \ref{enum:xi compact convergence of conditional laws step ptwise}. We equip $(\Omega',\mathcal{F}',\Pm')$ with the filtration $\mathcal{F}'_t:=\cap_{h>0}\sigma((X_{s\wedge \tau},F_{s\wedge \tau},W_{s})_{0\leq s\leq t+h},\tau\wedge (t+h))$. We see that $(W_t)_{t\geq 0}$ must be an $(\mathcal{F}'_t)_{t\geq 0}$-Brownian motion and $\tau$ an $(\mathcal{F}'_t)_{t\geq 0}$-stopping time. It is now sufficient to show that:
    \begin{enumerate}
        \item $(\mathcal{L}(X^k_t\lvert \tau^k>t),-\ln\Pm(\tau^k>t))_{0\leq t<\infty}\ra (\mathcal{L}(X_t\lvert \tau>t),-\ln\Pm(\tau>t))_{0\leq t<\infty}$ in $d^{\infty}$.
        \item $(X,W,\tau)$ is a global weak solution of \eqref{eq:MKV sde}.
        \end{enumerate}
    \item    \label{enum:xi compact tau=time to hit partial U}
    We establish that there exists an $\mathbb{F}'$-adapted and uniformly bounded process $b_t$ such that:
    \[
    dX_t=b_tdt+dW_t,\quad 0\leq t<\tau=\inf\{t>0:X_{t-}\in \partial U\}.
    \]
    Corollary \ref{cor:hitting time Corollary} then gives that $\Pm(\tau=t)=0$ whilst Lemma \ref{lem:upper and lower bounds on density of diffusions} gives that $\Pm(\tau>t)$ for all $t\geq 0$.
    \item
    We establish that $\mathcal{L}(X^k_t\lvert \tau^k>t)\ra \mathcal{L}(X_t\lvert \tau>t)$ in $\Wah$ and $\Pm(\tau^k>t)\ra \Pm(\tau>t)$ pointwise in t.
    \label{enum:xi compact convergence of conditional laws step ptwise}
    \item
    We now establish that $(\mathcal{L}(X^k_t\lvert \tau^k>t),-\ln\Pm(\tau^k>t))_{0\leq t<\infty}\ra (\mathcal{L}(X_t\lvert \tau>t)-\ln\Pm(\tau>t))_{0\leq t<\infty}$ in $d^{\infty}$.
    \label{enum:xi compact convergence of conditional laws step}
    \item
    By considering the martingale problem, we see that $b_t=b(\mathcal{L}(X_t\lvert \tau>t),X_t)$ for $t<\tau$ hence $(X,\tau,W)$ must be a global weak solution of \eqref{eq:MKV sde}. Since $\kappa\ni\mathcal{L}(X^k_0)\overset{\Wah}{\ra} \mathcal{L}(X_0)$ and $\kappa$ is compact in $\mathcal{P}_{\Wah}(U)$, $\mathcal{L}(X_0)\in \kappa$. Thus $(\mathcal{L}(X_t\lvert \tau>t),-\ln\Pm(\tau>t))_{0\leq t<\infty}\in \Xi(\kappa)$. Using Step \ref{enum:xi compact convergence of conditional laws step} we have established that for any sequence $(\mathcal{L}(X^k_t\lvert \tau^k>t),-\ln\Pm(\tau^k>t))_{0\leq t<\infty}$ in $\Xi(\kappa)$ there is a further subsequence converging in $d^{\infty}$ to an element $(\mathcal{L}(X_t\lvert \tau>t),-\ln\Pm(\tau>t))_{0\leq t<\infty}$ of $\Xi(\kappa)$. Thus we have Lemma \ref{lem:MKV Solns compact for compact set of ics}.
\end{enumerate}

\underline{Step \ref{enum:xi compact X W tight}}

We note that Aldous' condition \cite[Theorem 1]{Aldous1978} gives that $\{\mathcal{L}((X^k_{t\wedge \tau^k},F^k_{t\wedge \tau^k},W^k_{t\wedge \tau^k})_{0\leq t\leq T})\}$ is tight in $\mathcal{P}(D([0,T];\bar U\times \Rm^d\times \Rm^d))$ hence in $\mathcal{P}(C([0,T];\bar U\times \Rm^d\times \Rm^d))$ (equipped with the uniform metric) for any $T<\infty$.

We now fix $\epsilon>0$. Then there exists for each $T\in\mathbb{N}$ some $K_T\subseteq C([0,T];\bar U\times \Rm^d\times \Rm^d)$ compact such that $\Pm((X^k_{t\wedge \tau^k},F^k_{t\wedge \tau^k},W^k_{t\wedge \tau^k})_{0\leq t\leq T}\notin K_T)<\epsilon 2^{-T}$. We therefore define:
\[
\mathcal{K}=\{f\in (C([0,\infty);\bar U\times \Rm^d\times \Rm^d),d^{\infty}):(f_t)_{0\leq t\leq T}\in K_T\quad\text{for all}\quad T\in\mathbb{N}\}.
\]
\sloppy We see that $\mathcal{K}$ is clearly compact in $(C([0,\infty);\bar U\times \Rm^d\times \Rm^d),\bar d^{\infty})$, and moreover $\Pm((X^k_{t\wedge \tau^k},F^k_{t\wedge \tau^k},W^k_{t\wedge \tau^k})_{0\leq t<\infty}\notin \mathcal{K})\leq \sum_{T=1}^{\infty}\epsilon 2^{-T}\leq \epsilon$. Therefore we are done.

\underline{Step \ref{enum:xi compact tau=time to hit partial U}}

We note that $(X,W,\tau,F)$ $\Pm'$-almost surely satisfies:
\begin{equation}
dX_t=dF_t+dW_t,\quad t\leq \tau
\label{eq:Xt=Bt+dWt sde for tau hitting time}
\end{equation}
whereby $W_t$ is a Brownian motion up to time $\tau$ and moreover F has B-Lipschitz paths. We now define $b_t=\lim_{h\ra 0}\frac{F_{t+h}-F_{t}}{h}\in [-B,B]$ when the limit exists and $b_t=0$ otherwise. Since $F_t$ is Lipschitz, Rademacher's theorem allows us to see that:
\[
dX_t=b_tdt+dW_t,\quad t\leq \tau.
\]

We now seek to show that $\Pm'(\tau=\inf\{t>0:X_{t-}\in\partial U\})=1$. We let $\tau'=\inf\{t:X_{t-}\in \partial U\}$. Clearly we must have $X_{\tau-}\in \partial U$ if $\tau<\infty$ hence it is sufficient to show $\Pm'(\tau'<\tau)=0$. We must have $X_t\in \bar U$ for every $t\leq \tau$. Since $\partial U$ is smooth and $\Pm'$-almost surely $X_t$ satisfies \eqref{eq:Xt=Bt+dWt sde for tau hitting time}, if $\tau'<\tau$ then $\Pm'$-almost surely there exists $\tau''\in (\tau',\tau)$ such that $X_{\tau''}\notin \bar U$. This is impossible, thus $\Pm'(\tau'<\tau)=0$. 

\underline{Step \ref{enum:xi compact convergence of conditional laws step ptwise}}

Since $\Pm(\tau=t)=0$, $\Pm(\tau^k>t)\ra \Pm(\tau>t)$. We now take $\phi\in C_b(U)$ and extend $\phi$ to a $\bar U$ by setting $\phi(x)=0$ for $x\in \partial U$. We have $X^k_{t\wedge \tau^k}\ra X_{t\wedge \tau}$ $\Pm'$-almost surely. Unless $\tau^k>t$ for arbitrarily large k and $\tau\leq t$ we must have $\phi(X^k_{t\wedge \tau^k})\ra \phi(X_{t\wedge \tau})$. However since $\tau^k\ra\tau$, $\Pm'(\limsup_{k\ra\infty}\tau^k\geq t\geq \tau)\leq \Pm'(\tau=t)=0$ hence $\phi(X^k_{t\wedge \tau^k})\ra \phi(X_{t\wedge \tau})$ $\Pm'$-almost surely. Therefore we have $\mathcal{L}(X^k_t)\ra \mathcal{L}(X_t)$ in $\mathcal{M}(U)$ hence $\mathcal{L}(X^k_t\lvert \tau^k>t)\ra \mathcal{L}(X_t\lvert \tau>t)$ in $\Wah$.

\underline{Step \ref{enum:xi compact convergence of conditional laws step}}

We begin by establishing that for all $T<\infty$ we have:
\begin{equation}
\sup_{t\leq T}\lvert -\ln \Pm(\tau^k>t)+\ln \Pm(\tau>t)\rvert\ra 0\quad\text{as}\quad k\ra\infty.
\label{eq:conv of p tauk>t on compacts}
\end{equation}
We will then establish that for all $T<\infty$ we have:
\begin{equation}
\sup_{t\leq T}\Wah(\Law(X^k_t\lvert \tau^k>t),\Law(X_t\lvert \tau>t))\ra 0\quad\text{as}\quad k\ra\infty.
\label{eq:conv of conditional dist of Xk on compacts tightness proof}
\end{equation}
These would then imply $(\mathcal{L}(X^k_t\lvert \tau^k>t),-\ln\Pm(\tau^k>t))_{0\leq t<\infty}\ra (\mathcal{L}(X_t\lvert \tau>t)-\ln\Pm(\tau>t))_{0\leq t<\infty}$ in $d^{\infty}$.

$\Pm(\tau^k>t)$ and $\Pm(\tau>t)$ are continuous, non-negative, non-increasing in t, and uniformly (in $k\in\mathbb{N}$, $t\leq T$) bounded away from 0. This and Step \ref{enum:xi compact convergence of conditional laws step ptwise} imply \eqref{eq:conv of p tauk>t on compacts} and that
\begin{equation}
\sup_{\substack{0\leq t\leq t+h\leq T\\ k\in\mathbb{N}}}\Pm(t<\tau^k\leq t+h)\ra 0\quad\text{as}\quad h\ra 0
\label{eq:uniform cty of cumulative dist fn of tau k}
\end{equation}
by elementary analysis. We now turn to establishing \eqref{eq:conv of conditional dist of Xk on compacts tightness proof}. We calculate:
\begin{equation}
\begin{split}
\Wah(\mathcal{L}(X^k_{t+h}\lvert \tau^k>t+h),\mathcal{L}(X^k_{t}\lvert \tau^k>t))\\
\leq \Wah(\mathcal{L}(X^k_{t+h}\lvert \tau^k>t+h),\mathcal{L}(X^k_{t}\lvert \tau^k>t+h))
+\Wah(\mathcal{L}(X^k_{t}\lvert \tau^k>t+h),\mathcal{L}(X^k_{t}\lvert \tau^k>t)).
\end{split}
\label{eq:Step 6 of tightness proof decomp}
\end{equation}
We begin by bounding $\Wah(\mathcal{L}(X^k_{t}\lvert \tau^k>t+h),\mathcal{L}(X^k_{t}\lvert \tau^k>t))$. We observe that:
\[
\begin{split}
\mathcal{L}(X^k_{t})
=\mathcal{L}(X_t^k\lvert \tau^k>t+h)\underbrace{\Pm(\tau^k>t+h)}_{=\Pm(\tau^k>t)-\Pm(t<\tau^k\leq t+h)}\\
+\mathcal{L}(X_t^k\lvert t<\tau^k\leq t+h)\Pm(t<\tau^k\leq t+h)
=\mathcal{L}(X^k_t\lvert \tau^k> t+h)\Pm(\tau^k>t)
\\+\big(\mathcal{L}(X_t^k\lvert t<\tau^k\leq t+h)-\mathcal{L}(X_t^k\lvert \tau^k>t+h)\big)\Pm(t<\tau^k\leq t+h).
\end{split}
\]
Therefore we have:
\[
\begin{split}
\mathcal{L}(X^k_{t}\lvert \tau^k>t)=\mathcal{L}(X^k_t\lvert \tau^k> t+h)\\
+\big(\mathcal{L}(X_t^k\lvert t<\tau_k\leq t+h)-\mathcal{L}(X_t^k\lvert \tau^k>t+h)\big)\frac{\Pm(t<\tau^k\leq t+h)}{\Pm(\tau^k>t)}
\end{split}
\]
so that using \eqref{eq:uniform cty of cumulative dist fn of tau k} we have:
\[
\begin{split}
\Wah(\mathcal{L}(X^k_{t}\lvert \tau^k>t),\mathcal{L}(X^k_t\lvert \tau^k> t+h))\leq \lvert\lvert \mathcal{L}(X^k_{t}\lvert \tau^k>t)-\mathcal{L}(X^k_t\lvert \tau^k> t+h)\rvert\rvert_{\TV}\\
\leq \frac{\Pm(t<\tau^k\leq t+h)}{\Pm(\tau^k>t)}.
\end{split}
\]

We have $\Wah(\mathcal{L}(X^k_{t}\lvert \tau^k>t+h),\mathcal{L}(X^k_{t+h}\lvert \tau^k>t+h))\leq \frac{\expE[\lvert X_{(t+h)\wedge \tau^k}^k-X_{t\wedge \tau^k}^k\rvert]}{\Pm(\tau^k>t+h)}$ so that using \eqref{eq:Step 6 of tightness proof decomp} we have:
\[
\begin{split}
\sup_{\substack{k\in\mathbb{N}\\0\leq t\leq t'\leq t+h\leq T}}\Wah(\mathcal{L}(X^k_{t+h}\lvert \tau^k>t+h),\mathcal{L}(X^k_{t}\lvert \tau^k>t))\\ \leq \sup_{\substack{k\in\mathbb{N}\\0\leq t\leq t'\leq t+h\leq T}}\Big(\frac{\Pm(t<\tau^k\leq t+h)}{\Pm(\tau^k>t)}+\frac{\expE[\lvert X_{(t+h)\wedge \tau^k}^k-X_{t\wedge \tau^k}^k\rvert]}{\Pm(\tau^k>t+h)}\Big)\ra 0\quad\text{as}\quad h\ra 0.
\end{split}
\]

Therefore using Step \ref{enum:xi compact convergence of conditional laws step ptwise} we have \eqref{eq:conv of conditional dist of Xk on compacts tightness proof}. This concludes our proof of Lemma \ref{lem:MKV Solns compact for compact set of ics}.

\qed

\subsection{Proof of Theorem \ref{theo:N ra infty theorem}}
Our goal is to establish tightness of $\{\mathcal{L}((m^N_t,J^N_t)_{0\leq t<\infty})\}$ and characterise the limit distributions as being supported on $\Xi$ - the set of flows of laws of a stochastic process.

To characterise subsequential limits the strategy we would like to employ is to use martingale methods to chararacterise subsequential limits as being supported on the solution set of a nonlinear Fokker-Planck PDE, then to show that these PDE solutions correspond to global weak solutions of \eqref{eq:MKV sde}. 

Formally speaking subsequential limits should correspond to weak solutions of the nonlinear Fokker-Planck equation:
\[
\partial_t u + \nabla \cdot \left(b\left(\frac{u}{\lvert u\rvert_*},x\right) u\right) = \frac{1}{2} \Delta u, \quad \quad u \big \vert_{\partial U} = 0
\]
renormalised to have mass 1. We may rigorously show that subsequential limits of $\{(m_t^N,J_t^N)_{0\leq t<\infty}\}$ correspond to weak solutions of this PDE. However on unbounded domains we can't directly show these PDE solutions correspond to solutions of the McKean-Vlasov SDE \eqref{eq:MKV sde} as we need to make use of a uniqueness theorem \cite[Theorem 1.1]{Porretta2015a} for solutions of the linear Fokker-Planck equation which requires boundedness of the domain. 

We will instead consider a notion of solution which satisfies a certain approximation condition upon truncation of the domain to a large but bounded subdomain $U_R$ of U. Proposition \ref{prop:coupling to system on bounded domain} allows us to couple our $N$-particle system $\vec{X}^N$ to an $N$-particle system $\vec{X}^{N,R}$ on $U_R$ and obtain uniform controls on the difference between the two $N$-particle systems. Thus we show subsequential limits satisfy this approximation condition, and by martingale methods are solutions of our PDE.

We then show that such approximable PDE solutions correspond to solutions of the McKean-Vlasov SDE \eqref{eq:MKV sde}.

\subsubsection*{Overview}

For $R > R_{\min}+1$, we take $\vec{X}^{N,R}$ to be the particle system on the subdomain $U_R\subseteq U$ whose existence is guaranteed by Proposition \ref{prop:coupling to system on bounded domain} with associated empirical measure valued process and jump process respectively given by:
\[
m^{N,R}_t=\frac{1}{N}\sum_{i=1}^N\delta_{X^{N,R,i}_t},\quad J^{N,R}_t=\#\{\text{jumps up to time t by }\vec{X}^{N,R}\}.
\]

We define for $1+R_{\min}<R<\infty$ the following test functions:
\begin{equation}
\begin{split}
C_{0}^{\infty}(\bar U_R\times [0,\infty))=\{\varphi\in C^{\infty}_c(\bar U_R\times [0,\infty)):\varphi_{\lvert_{\partial U_R\times [0,\infty)}}\equiv 0\}
\end{split}
\label{eq:test fns with time variable}
\end{equation}
and define $C_{0}^{\infty}(\bar U\times [0,\infty))$ in the same manner, with $U_R$ replaced with $U$.

We define the following:
\begin{equation}
\begin{split}
M^{\varphi,N,R}_t:=\big(1-\frac{1}{N}\big)^{NJ^{N,R}_t}\langle m^{N,R}_t(.),\varphi(.,t)\rangle-\langle m^{N,R}_0(.,0),\varphi(.,0)\rangle
\\-\int_0^t\big(1-\frac{1}{N}\big)^{NJ^{N,R}_s}\langle m^{N,R}_s(.),\partial_s\varphi(.,s)+b(m^{N}_s,.)\cdot \nabla\varphi(.,s)\\+\frac{1}{2}\Delta \varphi(.,s)\rangle ds,\quad 0\leq t\leq T,\quad \varphi\in C_0^{\infty}(\bar U_R\times [0,\infty)),
\end{split}
\label{eq:equation for Martingales MphiNR}
\end{equation}
and define $M^{\varphi,N}_t$ in the same manner, with $U_R$ replaced with $U$ and $m^{N,R}$ replaced with $m^N$.

By showing these are martingales and using the Martingale Central Limit Theorem \cite[Theorem 2.1]{Whitt2007} we establish the following proposition:
\begin{prop}
For $R>1+R_{\min}$, $T<\infty$ and for fixed test function $\varphi\in C_0^{\infty}(\bar U_R\times [0,\infty))$, $(M^{\varphi,N,R}_t)_{0\leq t\leq T}$ (and similarly $M^{\varphi,N}_t)_{0\leq t\leq T}$ for $\varphi\in C_0^{\infty}(\bar U\times [0,\infty)$) converges to zero uniformly in probability:
\begin{equation}
\begin{split}
    \sup_{t\leq T}\lvert M_t^{\varphi,N,R}\rvert\ra 0\quad\text{in probability}\quad\text{as}\quad N\ra\infty.
\end{split}
\end{equation}
\label{prop: MN martingale convergence to 0}
\end{prop}

We then establish tightness of $\{\mathcal{L}((m_t^{N,R},J_t^{N,R})_{0\leq t\leq T})\}$ by combing Proposition \ref{prop: MN martingale convergence to 0} with the estimates of Section \ref{section:N-particle Estimates} (which prevent mass accumulating on the boundary):
\begin{prop}
We show for $R>1+R_{\min}$ and $T<\infty$ that $\{\mathcal{L}((m_t^{N,R},J_t^{N,R})_{0\leq t\leq T})\}$ (similarly $\{\mathcal{L}((m_t^{N},J_t^{N})_{0\leq t\leq T})\}$) is tight in $\mathcal{P}(D([0,T];\mathcal{P}_{\Wah}(U_R)\times \Rm_{\geq 0}))$ (respectively $\mathcal{P}(D([0,T];\mathcal{P}_{\Wah}(U)\times \Rm_{\geq 0}))$) with almost surely continuous limit distributions.
\label{prop:Tightness of (mR,JR) fixed time horizon}
\end{prop}

It is then simple to use Proposition \ref{prop:Tightness of (mR,JR) fixed time horizon} to establish that:
\begin{prop}
$\{\mathcal{L}((m_t^{N},J_t^{N})_{0\leq t<\infty})\}$ is tight in $\mathcal{P}((D([0,\infty);\mathcal{P}_{\Wah}(U)\times \Rm_{\geq 0}),d^{D}))$ with almost surely continuous limit distributions.
\label{prop: tightness of mJ infinite time horizon}
\end{prop}

Along subsequential limits we have $(1-\frac{1}{N})^{NJ^{N}_t}\ra e^{-J_t}$ and $m^N_t\ra m_t$ so that Proposition \ref{prop: MN martingale convergence to 0} gives us that $y_t:=e^{-J_t}m_t$ almost surely corresponds to a weak solution of:
\[
\begin{split}
\langle \varphi, y_t\rangle -\langle \varphi,y_0\rangle = 
\int_0^{t}\langle y_s(.),\partial_s\varphi(.,s)
+b(y_s,.)\cdot \nabla \varphi(.,s) +\frac{1}{2}\Delta \varphi(.,s)\rangle ds=0,\\
0\leq t\leq T,\quad \varphi\in C_0^{\infty}(\bar U\times [0,\infty)).
\end{split}
\]
We would then like to show that such a PDE solution corresponds to a solution of the McKean-Vlasov SDE \eqref{eq:MKV sde} by constructing a diffusion killed at the boundary $\partial U$ with drift $b(y_t,X_t)$ and showing that $\mathcal{L}(X_t)=y_t$. This final step requires a uniqueness result of Porretta \cite[Theorem 1.1]{Porretta2015a} for weak solutions of the linear Fokker-Planck PDE (both $y_t$ and the $\mathcal{L}(X_t)$ satisfy the same linear Fokker-Planck PDE with fixed drift $b(y_t,.)$). Availing ourselves of this uniqueness theorem, however, requires the following:
\begin{enumerate}
    \item We require $y=y_t\otimes dt$ to have a density with respect to $\text{Leb}_{U\times [0,\infty)}$. Lemma \ref{lem:char-almost sure abs cty} allows us to see that this is the case.
    \item We require $y_0$ to have a density with respect to $\text{Leb}_U$. Lemma \ref{lem:char-almost sure abs cty} allows us to see that this is the case after arbitrarily small time intervals. This issue may be overcome by arguing after a small time interval $t_0>0$, showing that $(y_{t_0+t})_{t\geq 0}$ corresponds to a McKean-Vlasov solution, then taking a limit as $t_0\ra 0$ using Lemma \ref{lem:MKV Solns compact for compact set of ics}.
    \item We require $U$ to be bounded, whereas we wish to include the case where $U$ is unbounded. To address this issue, we employ the coupling of Section \ref{section:coupling}. Since $U_R$ is bounded, we may apply the above strategy to the coupled particle system $\vec{X}^{N,R}$. By then employing the uniform controls of Proposition \ref{prop:coupling to system on bounded domain} and changing our notion of PDE solution, we are able to circumvent this problem.
\end{enumerate}

We now introduce our notion of PDE solution. Given $y\in C([0,\infty);\mathcal{P}_{\Wah}(U))$ and $R>1+R_{\min}$ we define:
\begin{equation}
\begin{split}
    \mathcal{H}_{R,T}(y)=\{z\in C([0,T];\mathcal{M}(U_R))\cap L^1(U_R\times [0,T]):z_t\in L^1(U_R)\quad \text{for all}\quad t \in \Qm_{>0}\quad
\text{and}\\ \langle z_t(.),\varphi(.,t)\rangle- \langle z_0(.),\varphi(.,0)\rangle 
-\int_0^{t}\langle z_s(.),\partial_s\varphi(.,s)
+b(y_s,.)\cdot \nabla \varphi(.,s) +\frac{1}{2}\Delta \varphi(.,s)\rangle ds=0,\\
\quad 0\leq t\leq T,\quad \varphi\in C^{\infty}_0(\bar U_R\times [0,\infty))\}.
\end{split}
\end{equation}
This is the solution set of the linear Fokker-Planck equation on the truncated domain and truncated time interval with drift given by $b(y_s,.)$. We now define the following notion of approximable PDE solution for the nonlinear Fokker-Planck equation:
\begin{equation}
\begin{split}
\mathcal{S}=\{(y,f)\in C([0,\infty);\mathcal{P}_{\Wah}(U)\times \Rm_{\geq 0}):\;\text{for every }\epsilon>0\;\text{and}\; T<\infty\;\text{we have for}\; R<\infty\\ \text{arbitrarily large that there exists}\; z\in \mathcal{H}_{R,T}(y)\text{ with }\sup_{t\leq T}\lvert\lvert y_te^{-f_t}-z_t\rvert\rvert_{\TV} \leq \epsilon\}.
\end{split}
\end{equation}

Note that at this point, we have not established existence of either PDE solutions or McKean-Vlasov solutions for given initial data. We will combine Proposition \ref{prop: MN martingale convergence to 0} with Lemma \ref{lem:char-almost sure abs cty} to show that any subsequential limit of our Fleming-Viot particle system must meet the criteria pathwise to being a PDE solution:
\begin{prop}
Suppose that some subsequence of $\{(m_t^N,J_t^N)_{0\leq t<\infty}\}$ converges in $(D([0,\infty);\mathcal{P}_{\Wah}(U)\times \Rm_{\geq 0}),d^{D})$ in distribution to $(m_t,J_t)_{0\leq t<\infty}$. Then $(m_t,J_t)_{0\leq t<\infty}\in\mathcal{S}$ almost surely.
\label{prop:limits given by PDE solutions}
\end{prop}
We then show that any such PDE solution must correspond to a solution of our McKean-Vlasov SDE \eqref{eq:MKV sde}:
\begin{prop}
\begin{equation}
\mathcal{S}\subseteq \Xi\cap C((0,\infty);L^1(U)).
\end{equation}
\label{prop:PDE solns are MKV solns}
\end{prop}

Taken together, these give Theorem \ref{theo:N ra infty theorem}.

\subsubsection*{Proof of Proposition \ref{prop: MN martingale convergence to 0}}

We provide here the proof for $M^{\varphi,N,R}_t$. The proof for $M^{\varphi,N}_t$ is identical with $U_R$, $m^{N,R}$, $C_{0}^{\infty}(\bar U_R\times [0,\infty))$ and $\tau^{N,R}_k$ replaced with $U$, $m^N$, $C_{0}^{\infty}(\bar U\times [0,\infty))$ and $\tau^N$ respectively.

We fix $\varphi\in C_{0}^{\infty}(\bar U_R\times [0,\infty))$, $1+R_{\min}<R\leq \infty$ and establish $(M^{\varphi,N,R}_t)_{0\leq t\leq T}$ converges to zero in distribution.

It is trivial that $M^{\varphi,N,R}_t$ is integrable for all t. We recall that $\tau^{N,R}_k$ is the $k^{\thh}$ death time of any particle in the coupled system with $\tau^{N,R}_0:=0$. Inducting on k, we shall establish that $M^{\varphi,N,R,k}_t:=M^{\varphi,N,R}_{t\wedge \tau_{k}}$ is a martingale. This is trivially true for $k=0$.

We note that $J^{N,R,k}_t$ is constant on $[\tau_k,\tau_{k+1})$, and moreover the infinitesimal generator of $\langle m^{N,R}_t(.),\varphi(.,t)\rangle$ is $\langle m^{N,R}_t(.),\partial_t\varphi(.,t)+b(m^N_t,.)\cdot \nabla\varphi(.,t)+\frac{1}{2}\Delta \varphi(.,t)\rangle $. Therefore we have $M^{\varphi,N,R,k+1}_{t\wedge\tau^{N,R}_{k+1}-}=\Ind(t<\tau_{k+1})M^{\varphi,N,R,k+1}_t+\Ind(t\geq \tau^{N,R}_{k+1})M^{\varphi,N,R,k+1}_{\tau^{N,R}_{k+1}-}$ is a martingale.

At time $\tau^{N,R}_{k+1}$, the particle which dies (let's say particle i) jumps to a uniformly chosen different particle (let's say particle j). Since $\varphi$ vanishes on the boundary $\partial U_R$, the value of $\varphi(X^{R,i}_t,t)$ jumps from 0 to $\varphi(X^{R,j}_{\tau_{k+1}-},\tau_{k+1}-)$, the expected value of which must be 
\[
\frac{1}{N-1}\sum_{j\neq i}\varphi(X^j_t,t)=\frac{N}{N-1}\langle m^{N,R}_{\tau_{k+1}-},\varphi(.,\tau_{k+1}-)\rangle.
\]
Thus we have:
\[
\begin{split}
\expE[\langle m_{\tau_{k+1}}^{N,R}(.),\varphi(.,\tau_{k+1})\rangle\lvert  \mathcal{F}_{\tau_{k+1}-}]=\frac{1}{N}\Big[\frac{N}{N-1}\langle m^{N,R}_{\tau_{k+1}-},\varphi(.,\tau_{k+1}-)\rangle+N\langle m^{N,R}_{\tau_{k+1}-},\varphi(.,\tau_{k+1}-)\rangle\Big]\\
=\Big(1-\frac{1}{N}\Big)^{-1}\langle m^{N,R}_{\tau_{k+1}-},\varphi(.,\tau_{k+1}-)\rangle.
\end{split}
\]

Thus $\expE[M^{k+1}_{\tau^{N,R}_{k+1}}\lvert \mathcal{F}_{\tau_{k+1}-}]=M^{k+1}_{\tau^{N,R}_{k+1}-}$. Therefore we have $M^{\varphi,N,R,k+1}_t$ is a martingale. Thus $M^{\varphi,N,R}_t$ is a martingale.

We shall now employ the Martingale Central Limit Theorem \cite[Theorem 2.1]{Whitt2007} to obtain convergence to 0 in probability as $N\ra\infty$. Between times $\tau^{N,R}_k$ and $\tau^{N,R}_{k+1}$, we have:
\[
dM_t^{\varphi,N,R}=\Big(1-\frac{1}{N}\Big)^{NJ_t^{N,R}}\frac{1}{N}\sum_{i=1}^N\nabla\varphi(X_t^{R,i},t)\cdot dW_t^{R,i}+\text{drift terms.}
\]
Hence we have $[M^{N,R,\varphi}]_{T\wedge\tau_{k+1}-}-[M^{N,R,\varphi}]_{T\wedge\tau_{k}}\leq \frac{1}{N}\lvert \lvert \nabla\varphi\rvert\rvert_{\infty}^2 (\tau^{N,R}_{k+1}-\tau^{N,R}_k)$. Moreover, at each jump time, the jumps of $M^{\varphi,N,R}$ are bounded by:
\[
\begin{split}
\lvert M^{\varphi,N,R}_{T\wedge \tau^{N,R}_{k+1}}-M^{\varphi,N,R}_{T\wedge \tau^{N,R}_{k+1}-}\rvert\leq \big(1-\frac{1}{N}\big)^k\big\lvert \big\langle m^{N,R}_{T\wedge \tau^{N,R}_{k+1}}(.)-m^{N,R}_{T\wedge \tau^{N,R}_{k+1}-}(.), \varphi(.,T\wedge \tau^{N,R}_{k+1})\big\rangle\big\rvert\\
+\big(1-\frac{1}{N}\big)^k\big\lvert \big\langle (1-\frac{1}{N})m^{N,R}_{T\wedge \tau^{N,R}_{k+1}}(.)- m^{N,R}_{T\wedge \tau^{N,R}_{k+1}}(.), \varphi(.,T\wedge \tau^{N,R}_{k+1})\big\rangle\big\rvert\leq 3\big(1-\frac{1}{N}\big)^k\frac{\lvert\lvert \varphi\rvert\rvert_{\infty}^2}{N}.
\end{split}
\]

Therefore the jumps of $[M^{\varphi,N,R}]_t$ are bounded by:
\[
[M^{\varphi,N,R}]_{T\wedge \tau^{N,R}_{k+1}}-[M^{\varphi,N,R}]_{T\wedge \tau^{N,R}_{k+1}-} \leq \Big(1-\frac{1}{N}\Big)^{2k}\frac{9\lvert\lvert \varphi\rvert\rvert_{\infty}^2}{N^2}\leq \Big(1-\frac{1}{N}\Big)^{k}\frac{9\lvert\lvert \varphi\rvert\rvert_{\infty}^2}{N^2}.
\]
Therefore summing the geometric series we have:
\[
[M^{N,R,\varphi}]_T\leq \frac{1}{N}\lvert \lvert \nabla\varphi\rvert\rvert_{\infty}^2 T+\frac{9\lvert\lvert \varphi\rvert\rvert_{\infty}^2}{N}\ra 0\quad\text{as}\quad N\ra\infty.
\]
Thus we have $[M^{N,R,\varphi}]_T$ converges to zero in probability as $N\ra\infty$. Moreover it is trivial that $\expE[\sup_{t\leq T}\lvert M^{N,R,\varphi}_{t}-M^{N,R,\varphi}_{t-}\rvert]\ra 0$ in probability as $N\ra\infty$. Thus using the Martingale Central Limit Theorem \cite[Theorem 2.1]{Whitt2007} we have $(M^{\varphi,N,R}_t)_{0\leq t\leq T}\ra 0$ uniformly in probability.

\qed

\subsubsection*{Proof of Proposition \ref{prop:Tightness of (mR,JR) fixed time horizon}}
We provide here the proof for $\{\mathcal{L}((m^{N,R}_t,J^{N,R}_t)_{0\leq t\leq T})\}$. The proof for $\{\mathcal{L}((m^{N}_t,J^{N}_t)_{0\leq t\leq T})\}$ is identical, but with $U_R$, $m^{N,R}$, $C_{0}^{\infty}(\bar U_R\times [0,\infty))$ and $\tau^{N,R}_k$ replaced with $U$, $m^N$, $C_{0}^{\infty}(\bar U\times [0,\infty))$ and $\tau^N$, respectively, aside from two places where Lemma \ref{lem:tight after bounded times} must be invoked.

The proof can be broken down into the following steps:
\begin{enumerate}
\item \label{enum:J tightness}
We begin by establishing that $\{\mathcal{L}((J^{N,R}_t)_{0\leq t\leq T})\}$ is tight in $\mathcal{P}(D([0,T];\Rm_{\geq 0}))$, and moreover any limit distribution is supported on the space of continuous functions.
\item \label{enum:tightness of m open u}
We then show $\{\mathcal{L}((m_t^{N,R})_{0\leq t\leq T})\}$ is tight in $\mathcal{P}(D([0,T];\mathcal{P}_{\Wah}(U_R)))$.
\item
Having shown that $\{\mathcal{L}(((m^{N,R}_t)_{0\leq t\leq T},(J^{N,R}_t)_{0\leq t\leq T}))\}$ is tight in $\mathcal{P}(D([0,T];\mathcal{P}_{\Wah}(U_R))\times D([0,T];\Rm_{\geq 0}))$ with limit distributions supported on $\mathcal{P}(D([0,T];\mathcal{P}_{\Wah}(U_R))\times C([0,T];\Rm_{\geq 0}))$ we establish $\{\mathcal{L}((m_t^{N,R},J^{N,R}_t)_{0\leq t\leq T})\}$ is tight in $\mathcal{P}(D([0,T];\mathcal{P}_{\Wah}(U_R)\times \Rm_{\geq 0}))$ with almost surely continuous limit distributions.
\label{enum:tightness of (mR,JR) in Skorokhod over fixed time interval}
\end{enumerate}

\underline{Step \ref{enum:J tightness}}

Markov's inequality and Proposition \ref{prop:bound on number of jumps by particle} give that $\{\mathcal{L}(J^{N,R}_T:N\in\mathbb{N})\}$ is tight. Thus it is enough to show the set of laws of $\varsigma^N_t:=\big(1-\frac{1}{N}\big)^{NJ_t^{N,R}}$ is tight in $\mathcal{P}(D([0,T];\Rm))$ with limit distributions supported on $C([0,T];\Rm)$. We will employ Aldous' condition \cite[Theorem 1]{Aldous1978}. Since we have $0\leq \big(1-\frac{1}{N}\big)^{NJ_t^{N,R}}\leq 1$ then $\{\mathcal{L}(\varsigma_t^N)\}$ must be tight for each fixed t. We therefore need to establish \cite[Condition A]{Aldous1978}.

We fix $\epsilon,\delta>0$. As in Part \ref{enum:bound on mass near bdy} of Proposition \ref{prop:bound on mass near bdy} we take $\hat{K}_{\frac{\epsilon}{2},\frac{\delta}{2}}=V_{\hat{c}(\frac{\epsilon}{2},\frac{\delta}{2})}\subseteq U_R$ such that we have $\limsup_{N\ra\infty}\Pm(\sup_{t\leq T}m^{N,R}_t(\hat{K}_{\frac{\epsilon}{2},\frac{\delta}{2}}^c)\geq \frac{\epsilon}{2})\leq \frac{\delta}{2}$. Since $U_R$ is bounded, $\hat{K}_{\frac{\epsilon}{2},\frac{\delta}{2}}$ is compact. Here the proof for $\{\mathcal{L}((m^{N}_t,J^{N}_t)_{0\leq t\leq T})\}$ diverges from the present proof as $U$ is not necessarily bounded. In this case we use Lemma \ref{lem:tight after bounded times} to obtain $R'_{\epsilon,\delta}<\infty$ such that $\limsup_{N\ra\infty}\Pm(\sup_{t\leq T}m^{N,R}_t( B(0,R'_{\epsilon,\delta})^c)\geq \frac{\epsilon}{2})\leq \frac{\delta}{2}$. In either case we obtain $\tilde{K}_{\epsilon,\delta}=\hat{K}_{\frac{\epsilon}{2},\frac{\delta}{2}}\cap \bar B(0,R'_{\epsilon,\delta})\subseteq U_R$ compact such that:
\[
\limsup_{N\ra\infty}\Pm(\sup_{t\leq T}m^{N,R}_t(\tilde{K}_{\epsilon,\delta}^c)\geq \epsilon)\leq \delta.
\]

We now take $\varphi_{\epsilon,\delta}\in C_{c}^{\infty}(U_R)$ such that $\Ind_{\hat{K}_{\epsilon,\delta}}\leq \varphi_{\epsilon,\delta}\leq 1$. Thus we have:
\begin{equation}
\limsup_{N\ra\infty}\Pm\Big(\sup_{0\leq t\leq T}\lvert 1-\langle m^{N,R}_{t},\varphi_{\epsilon,\delta}(.)\rangle\rvert\leq \epsilon\Big)\leq \delta.
\label{eq:bound on 1 minus (m,varphi)}
\end{equation}
We then take $M_{t}^{\varphi_{\epsilon,\delta},N,R}$ as in \eqref{eq:equation for Martingales MphiNR} and observe:
\[
\begin{split}
\varsigma^N_{t+h}-\varsigma^N_t=(\varsigma^N_{t+h}-\varsigma^N_{t+h}\langle m^{N,R}_{t+h},\varphi_{\epsilon,\delta}(.)\rangle)
-(\varsigma^N_{t}-\varsigma^N_{t}\langle m_t,\varphi_{\epsilon,\delta}(.)\rangle)+(M^{\varphi_{\epsilon,\delta},N,R}_{t+h}-M^{\varphi_{\epsilon,\delta},N,R}_t)\\
+\int_{t}^{t+h}\varsigma^N_s\langle m^{N,R}_s(.),b(m^N_s,.)\cdot \nabla\varphi_{\epsilon,\delta}+\frac{1}{2}\Delta\varphi_{\epsilon,\delta}\rangle ds.
\end{split}
\]
We bound the first two terms on the right hand side using \eqref{eq:bound on 1 minus (m,varphi)}, the third term converges to zero in probability using Proposition \ref{prop: MN martingale convergence to 0} whilst the integrand in the fourth term is bounded (by $C_{\epsilon,\delta}<\infty$ say). Therefore we have:
\[
\liminf_{N\ra\infty}\Pm\big(\sup_{\substack{0\leq t\leq t+h'\\\leq t+h\leq T}}\lvert \varsigma^N_{t+h'}-\varsigma^N_t\rvert\leq 3\epsilon+C_{\epsilon,\delta}h\big)\geq 1-2\delta.
\]

This establishes \cite[Condition A]{Aldous1978}. Moreover for any subsequential limit in distribution $\varsigma^{\infty}$ and $\epsilon,\delta>0$ there exists some $h_{\epsilon,\delta}=\frac{\epsilon}{C_{\epsilon,\delta}}>0$ such that:
\[
\Pm(\sup_{\substack{h'\leq h_{\epsilon,\delta}\\ 0\leq t\leq T-h'}}\lvert\varsigma^{\infty}_{t+h'}-\varsigma^{\infty}_t\rvert\geq 5\epsilon)\leq 2\delta.
\]
Thus as $\delta>0$ is arbitrary there exists some random $h(\epsilon)>0$ such that 
\[
\sup_{\substack{h'\leq h(\epsilon)\\ 0\leq t\leq T-h'}}\lvert\varsigma^{\infty}_{t+h'}-\varsigma^{\infty}_t\rvert\leq \epsilon \quad\text{almost surely.}
\]
Since $\epsilon>0$ is arbitrary, $\zeta^{\infty}\in C([0,T];\Rm)$ almost surely. 

\underline{Step \ref{enum:tightness of m open u}}

We show $\{\mathcal{L}((m_t^{N,R})_{0\leq t\leq T})\}$ is tight in $\mathcal{P}(D([0,T];\mathcal{P}_{\Wah}(\bar U_R)))$, then extend this to showing $\{\mathcal{L}((m_t^{N,R})_{0\leq t\leq T})\}$ is tight in $\mathcal{P}(D([0,T];\mathcal{P}_{\Wah}(U_R)))$.

Since $U_R$ is bounded, \cite[Theorem 2.1]{Gorostiza1990} gives us that:
\begin{lem}[\cite{Gorostiza1990}]
We suppose that for every $\varphi\in C_c^{\infty}(\Rm^d)$ the laws of $\varrho^N_t:=\langle \varphi(.),m_t^{N,R}(.)\rangle$ are tight in $\mathcal{P}(D([0,T];\Rm))$. Then $\{\mathcal{L}((m_t^{N,R})_{0\leq t\leq T})\}$ must be tight in $\mathcal{P}(D([0,T];\mathcal{P}_{\Wah}(\bar U)))$.
\label{lem:enough for test fns tightness}
\end{lem}
Here the proof for $\{\mathcal{L}((m^{N}_t,J^{N}_t)_{0\leq t\leq T})\}$ diverges from the present proof as $U$ is not necessarily bounded. In this case we obtain Lemma \ref{lem:enough for test fns tightness} by combining \cite[Theorem 2.1]{Gorostiza1990} with Lemma \ref{lem:tight after bounded times}.

We now verify the assumptions of Lemma \ref{lem:enough for test fns tightness}. We fix $\varphi\in C_c^{\infty}(\Rm^d)$ and establish that $\{\mathcal{L}(\varrho^N_t)\}$ is tight in $\mathcal{P}(D([0,T];\Rm))$ by way of Aldous' criterion \cite[Theorem 1]{Aldous1978}. 

Since $\varphi$ is bounded $\{\mathcal{L}(\varrho^N_t)\}$ is tight on the line for fixed t, so it is sufficient to check \cite[Condition A]{Aldous1978}. We let $\tau^N$ be a sequence of stopping times and $\delta_N$ a sequence of constants as defined in \cite[Condition 1]{Aldous1978}. We write $\varrho^N=\varrho^{N,C}+\varrho^{N,J}$ whereby $\varrho^{N,C}$ is continuous and $\varrho^{N,J}_{t}=\sum_{t'\leq t}\varrho^{N}_{t'}-\varrho^N_{t'-}$. Then $\varrho^{N,C}$ is a diffusion process with uniformly bounded drift and diffusivity hence we have:
\[
\varrho^{N,C}_{\tau_N+\delta_N}-\varrho^{N,C}_{\tau_N}\overset{p}{\ra} 0.
\]
We note that the jumps of $\varrho$ are of magnitude bounded by $\frac{C}{N}$ for some $C<\infty$. Therefore to verify
\[
\varrho^{N,J}_{\tau_N+\delta_N}-\varrho^{N,J}_{\tau_N}\overset{p}{\ra} 0
\]
it is enough to check:
\[
J^{N,R}_{\tau_N+\delta_N}-J^{N,R}_{\tau_N}\overset{p}{\ra} 0.
\]
We have this since $\{\mathcal{L}((J^{N,R}_t)_{0\leq t\leq T})\}$ is tight in $\mathcal{P}(D([0,T];\Rm))$ with limit distributions supported on $C([0,T];\Rm)$ (Step \ref{enum:J tightness}). Thus we have verified \cite[Condition A]{Aldous1978}:
\[
\varrho^{N}_{\tau_N+\delta_N}-\varrho^{N}_{\tau_N}\overset{p}{\ra} 0
\]
and hence have verified the assumption of Lemma \ref{lem:enough for test fns tightness}.

Having established $\{\mathcal{L}((m_t^{N,R})_{0\leq t\leq T})\}$ is tight in $\mathcal{P}(D([0,T];\mathcal{P}_{\Wah}(\bar U_R)))$, we now show it is tight in $\mathcal{P}(D([0,T];\mathcal{P}_{\Wah}(U_R)))$. Using Skorokhod's representation theorem, we consider along any subsequence a further subsequence converging on a possibly different probability $(\Omega',\mathcal{F}',\Pm')$ space in $D([0,T];\mathcal{P}_{\Wah}(\bar U_R))$ $\Pm'$-almost surely to $(m^R_t)_{0\leq t\leq T}$. It is sufficient to show $(m^R_t)_{0\leq t\leq T}\in D([0,T];\mathcal{P}_{\Wah}(U_R))$ $\Pm'$-almost surely.

For each $\epsilon, T_0>0$, Part \ref{enum:bound on mass near bdy after small time} of Proposition \ref{prop:bound on mass near bdy} implies that $m^R_t(K_{\epsilon,T_0}^c)<\epsilon$ for every $T_0\leq t\leq T $ $\Pm'$-almost surely. Therefore $m^R_t(\partial U)=0$ for every $T_0 \leq t\leq T$ $\Pm'$-almost surely. Since $T_0$ can be made arbitrarily small and $\{\mathcal{L}(m^{N,R}_0)\}$ is tight in $\mathcal{P}(\mathcal{P}_{\Wah}(U_R))$ we have $m^R_t(\partial U)=0$ for all $0\leq t\leq T$ $\Pm'$-almost surely.

\underline{Step \ref{enum:tightness of (mR,JR) in Skorokhod over fixed time interval}}

{
It is sufficient to consider some subsequence on which $((m^{N,R}_t)_{0\leq t\leq T},(J^{N,R}_t)_{0\leq t\leq T})$ converges in $D([0,T];\mathcal{P}_{\Wah}(U_R))\times D([0,T];\Rm_{\geq 0})$ in distribution, then establish along this subsequence convergence in $D([0,T];\mathcal{P}_{\Wah}(U_R)\times \Rm_{\geq 0})$ in distribution with limit distributions supported on $C([0,T];\mathcal{P}_{\Wah}(U_R)\times \Rm_{\geq 0})$.

Indeed by the Skorokhod Representation Theorem on a possibly different probability space $(\Omega',\mathcal{F}',\Pm')$ we have along this subsequence $\Pm'$-almost sure convergence of $((m^{N,R}_t)_{0\leq t\leq T},(J^{N,R}_t)_{0\leq t\leq T})$ to a limit we call $((m^R_t)_{0\leq t\leq T},(J^R_t)_{0\leq t\leq T})$. By Step \ref{enum:J tightness} we have $J^R_t$ is continuous, and hence $\Pm'$-almost surely $J^{N,R}_t$ converges uniformly to $J^R_t$.

From the definition of the Skorokhod metric \cite[Equation (12.13), Page 124]{Billingsley1999} it is trivial that this implies $(m^{N,R}_t,J_t^{N,R})_{0\leq t\leq T}$ converges to $(m^R_t,J^R_t)_{0\leq t\leq T}$ $\Pm'$-almost surely in $D([0,T];\mathcal{P}_{\Wah}(U_R)\times \Rm_{\geq 0})$. Now we have:
\[
\begin{split}
\Big\lvert M^{N,R,\varphi}_{t+h}-M^{N,R,\varphi}_t-\big(\langle (1-\frac{1}{N})^{NJ^{N,R}_{t+h}}m^{N,R}_{t+h}(.),\varphi(.)\rangle-\langle (1-\frac{1}{N})^{NJ^{N,R}_t}m^{N,R}_{t}(.),\varphi(.)\rangle\big)\Big\rvert\\
\leq C_{\varphi}h,\quad \varphi\in C^{\infty}_c(U)  
\end{split}
\]
where $C_{\varphi}$ is a constant dependent only upon $\varphi$. Note that we are viewing $\varphi$ both as a function in $C_c^{\infty}(U_R)$ and a function in $C_0^{\infty}(\bar U_R\times [0,\infty))$ which is constant in time up to time $T$ by abuse of notation. Proposition \ref{prop: MN martingale convergence to 0} then implies that almost surely $(m^R_t)_{0\leq t\leq T}$ satisfies for all $0\leq t\leq t+h\leq T$:
\[
\lvert \langle e^{-J^R_{t+h}}m^R_{t+h}(.),\varphi(.)\rangle-
    \langle e^{-J^R_{t}}m^R_t(.),\varphi(.)\rangle\rvert\leq C_{\varphi}h,\quad \varphi\in C_c^{\infty}(U_R\times [0,T]).
\]
We know $m^R_te^{-J^R_t}\in D([0,T];\mathcal{M}(U_R))$, so that we have:
\[
\lvert \langle e^{-J^R_{t}}m^R_{t}(.),\varphi(.)\rangle-
    \langle e^{-J^R_{t-}}m^R_{t-}(.),\varphi(.)\rangle\rvert=0,\quad \varphi\in C_c^{\infty}(U_R),\quad 0\leq t\leq T.
\]
This implies $e^{-J^R_{t}}m^R_{t}=e^{-J^R_{t-}}m^R_{t-}$ for all $t\leq T$ hence $e^{-J^R_{t}}m^R_{t}\in C([0,T];\mathcal{M}(U_R))$. Thus almost surely $m^R\in C([0,T];\mathcal{P}(U_R))$. Since $\Wah$ metrises the topology of weak convergence of probability measures, we are done.
\qed
\subsubsection*{Proof of Proposition \ref{prop: tightness of mJ infinite time horizon}}

We fix $\epsilon>0$. Then by Proposition \ref{prop:Tightness of (mR,JR) fixed time horizon} there exists for each $T\in\mathbb{N}$ some $K_T\subseteq D([0,T];\mathcal{P}_{\Wah}(U)\times \Rm_{\geq 0})$ compact such that $\Pm((m^N_t,J^N_t)_{0\leq t\leq T}\notin K_T)<\epsilon 2^{-T}$. We therefore define:
\[
\mathcal{K}=\{f\in D([0,\infty);\mathcal{P}_{\Wah}(U)\times \Rm_{\geq 0}),d^{D}):(f_t)_{0\leq t\leq T}\in K_T\quad\text{for all}\quad T\in\mathbb{N}\}.
\]
We see that $\mathcal{K}$ is clearly compact in $(D([0,\infty);\mathcal{P}_{\Wah}(U)\times \Rm_{\geq 0}),d^{D})$, and moreover $\Pm((m^N_t,J^N_t)_{0\leq t<\infty}\notin \kappa)\leq \sum_T\epsilon 2^{-T}\leq \epsilon$.

\qed

\subsubsection*{Proof of Proposition \ref{prop:limits given by PDE solutions}}

We write $(\Omega',\mathcal{F}',\Pm')$ for the probability space on which our subsequential limit $(m_t,J_t)_{0\leq t<\infty}$ is defined. We define:
\begin{equation}
\begin{split}
\mathcal{S}_{\epsilon,R,T}=\{(y,f)\in C([0,\infty);\mathcal{P}_{\Wah}(U)\times \Rm_{\geq 0}):\;\text{there exists}\\ z\in \mathcal{H}_{R,T}(y)\text{ with }\sup_{t\leq T}\lvert\lvert y_te^{-f_t}-z_t\rvert\rvert_{\TV} \leq \epsilon\}.
\end{split}
\end{equation}

We claim that for all $\epsilon>0$ and $T<\infty$ fixed:
\begin{equation}
\Pm'((m_t,J_t)_{0\leq t<\infty}\in \mathcal{S}_{\epsilon,R,T})\ra 0\quad \text{as}\quad R\ra \infty.
\label{eq:supported on SeRT with large probability}
\end{equation}

We fix $R<\infty$ for the time being. We take, on the probability space $(\Omega^{N,R},\mathcal{F}^{N,R},\Pm^{N,R})$, the particle system $\vec{X}^{N,R}$ on $U_R$ coupled to $\vec{X}^N$ whose existence is guaranteed by Proposition \ref{prop:coupling to system on bounded domain}. We have by propositions \ref{prop:Tightness of (mR,JR) fixed time horizon} and \ref{prop: tightness of mJ infinite time horizon} that $\{\mathcal{L}(((m_t^N,J_t^N)_{0\leq t<\infty},(m_t^{N,R},J_t^{N,R})_{0\leq t\leq T}))\}$ is tight in $\mathcal{P}((D([0,\infty);\mathcal{P}_{\Wah}(U)\times \Rm_{\geq 0}),d^{D})\times D([0,T];\mathcal{P}_{\Wah}(U)\times \Rm_{\geq 0}))$ with limit distributions supported on $C([0,\infty);\mathcal{P}_{\Wah}(U)\times \Rm_{\geq 0})\times C([0,T];\mathcal{P}_{\Wah}(U)\times \Rm_{\geq 0})$. We may therefore take a further subsequence along which $\{((m_t^N,J_t^N)_{0\leq t<\infty},(m_t^{N,R},J_t^{N,R})_{0\leq t\leq T})\}$ is convergent in distribution. Using Skorokhod's representation theorem, these may be supported on a probability space $(\Omega'',\mathcal{F}'',\Pm'')$ along which $\{((m_t^N,J_t^N)_{0\leq t<\infty},(m_t^{N,R},J_t^{N,R})_{0\leq t\leq T})\}$ is $\Pm'$-almost surely convergent, to a limit we call $((m_t,J_t)_{0\leq t<\infty},(m_t^{R},J_t^{R})_{0\leq t\leq T})$.

Note that we are abusing notation here, writing $(m_t,J_t)_{0\leq t<\infty}$ both for a random variable on $(\Omega',\mathcal{F}',\Pm')$ and for a random variable on $(\Omega'',\mathcal{F}'',\Pm'')$. Nevertheless, by construction, they have the same law, hence:
\[
\Pm'((m_t,J_t)_{0\leq t<\infty}\in \mathcal{S}_{\epsilon,R,T})=\Pm''((m_t,J_t)_{0\leq t<\infty}\in \mathcal{S}_{\epsilon,R,T}).
\]

For $t\in \Qm_{>0}$ we have by Lemma \ref{lem:char-almost sure abs cty} that $m^{R}_t\in L^1(U_R)$ $\Pm''$-almost surely. Therefore $m^{R}_t\in L^1(U_R)$ for all $t\in \Qm_{>0}$, $\Pm''$-almost surely. Moreover, Lemma \ref{lem:char-almost sure abs cty} gives that $m^{R}=m^{R}_t\otimes dt$ satisfies $m^{R}\in L^1(U_R\times [0,T])$, $\Pm''$-almost surely. Therefore, by Proposition \ref{prop: MN martingale convergence to 0} we have 
\[
(m^{R}_te^{-J^{R}_t})_{0\leq t\leq T}\in \mathcal{H}_{R,T}(m)\quad \Pm''\text{-almost surely}.
\]

Since convergence in Skorokhod space to a continuous function implies uniform convergence, $(m_t^N,J_t^N)_{0\leq t\leq T}\ra (m_t,J_t)_{0\leq t\leq T}$ and $(m_t^{N,R},J_t^{N,R})_{0\leq t\leq T}\ra (m_t^{R},J_t^{R})_{0\leq t\leq T}$ in $d^{\infty}_{[0,T]}$ $\Pm''$-almost surely. Therefore we have:
\[
\sup_{t\leq T}\lvert\lvert m^{R}_t-m_t \rvert\rvert_{\TV}\leq \liminf_{N\ra\infty}\sup_{t\leq T}\lvert\lvert m^{N,R}_t-m^N_t \rvert\rvert_{\TV}\quad\Pm''\text{-almost surely.}
\]
Therefore we have:
\[
\begin{split}
\expE^{\Pm''}\Big[\sup_{t\leq T}\lvert\lvert m^{R}_te^{-J^{R}_t}-m_te^{-J_t}\rvert\rvert_{\TV}\Big]\leq \expE^{\Pm''}\Big[\sup_{t\leq T}\lvert\lvert m^{R}_t-m_t \rvert\rvert_{\TV}+1\wedge \sup_{t\leq T}\lvert J^{R}_t-J_t\rvert\Big]\\
\leq \expE^{\Pm''}\Big[\liminf_{N\ra\infty}\big(\sup_{t\leq T}\lvert\lvert m^{N,R}_t-m^N_t \rvert\rvert_{\TV}+1\wedge \sup_{t\leq T}\lvert J^{N,R}_t-J^{N}_t\rvert\big)\Big]\\
\overset{\text{Fatou's Lemma}}{\leq} \liminf_{N\ra\infty}\expE^{\Pm''}\Big[\sup_{t\leq T}\lvert\lvert m^N_t-m^{N,R}_t\rvert\rvert_{\TV}+1\wedge \sup_{t\leq T}\lvert J^N_t-J^{N,R}_t\rvert\Big]\\
=\liminf_{N\ra\infty}\expE^{\Pm^{N,R}}\Big[\sup_{t\leq T}\lvert\lvert m^N_t-m^{N,R}_t\rvert\rvert_{\TV}+1\wedge \sup_{t\leq T}\lvert J^N_t-J^{N,R}_t\rvert\Big]\ra 0 \quad\text{as}\quad R\ra \infty
\end{split}
\]
by Proposition \ref{prop:coupling to system on bounded domain}. Therefore using Markov's inequality we have:
\[
\begin{split}
\Pm'((m_t,J_t)_{0\leq t<\infty}\in \mathcal{S}_{\epsilon,R,T})=\Pm''((m_t,J_t)_{0\leq t<\infty}\in \mathcal{S}_{\epsilon,R,T})\\
\leq \frac{1}{\epsilon}\expE^{\Pm''}\Big[\sup_{t\leq T}\lvert\lvert m^{R}_te^{-J^{R}_t}-m_te^{-J_t}\rvert\rvert_{\TV}\Big]\ra 0\quad\text{as}\quad R\ra\infty.
\end{split}
\]

Therefore we have \eqref{eq:supported on SeRT with large probability} so that for all $R_0<\infty$:
\[
(m_t,J_t)_{0\leq t<\infty}\in \cup_{R\geq R_0}\mathcal{S}_{\epsilon, R,T}\quad \Pm'\text{-almost surely.}
\]
Therefore we have:
\[
(m_t,J_t)_{0\leq t<\infty}\in \cap_{\epsilon >0}\cap_{T\in\mathbb{N}}\cap_{R_0\in \mathbb{N}}\cup_{R\geq R_0}\mathcal{S}_{\epsilon, R,T}=\mathcal{S}\quad \Pm'\text{-almost surely.}
\]

\qed

\subsubsection*{Proof of Proposition \ref{prop:PDE solns are MKV solns}}

\underline{Step 1}

We fix deterministic $(m_t,J_t)_{0\leq t<\infty}\in\mathcal{S}$ and use Girsanov's theorem to construct $(X,\tau,W)$ a global weak solution of the SDE:
\begin{equation}
dX_t=b(m_t,X_t)dt+dW_t,\quad 0\leq t\leq \tau=\inf\{t>0:X_{t}\in \partial U\}.
\label{eq:SDE for whole domain deterministic m}
\end{equation}

\underline{Step 2}

For the time being we fix $R,T<\infty$ and assume that there exists $z\in \mathcal{H}_{R,T}(m)$ such that $z_0\in L^1(U_R)$ and $(z_t)_{0\leq t\leq T}$ is a solution of:
\begin{equation}
\begin{split}
\langle z_t(.),\varphi(.,t)\rangle- \langle z_0(.),\varphi(.,0)\rangle 
-\int_0^{t}\langle z_s(.),\partial_s\varphi(.,s)
+b(m_s,.)\cdot \nabla \varphi(.,s)\\ +\frac{1}{2}\Delta \varphi(.,s)\rangle ds=0,
\quad 0\leq t\leq T,\quad \varphi\in C^{\infty}_0(\bar U_R\times [0,\infty)).
\end{split}
\label{eq:lin PDE for fixed (m,J) in proof PDE solns are MKV solns}
\end{equation}

Then defining $\tau^R=\inf\{t>0:X_{t}\in \partial U_R\}\leq \tau$ and $(X^R_t)_{0\leq t\leq \tau^R}:=(X_t)_{0\leq t\leq \tau}$ we obtain $(X^R,\tau^R,W)$ a weak solution of the SDE:
\begin{equation}
    dX^R_t=b(m_t,X^R_t)dt+dW_t,\quad 0\leq t\leq \tau^R=\inf\{t>0:X^R_{t}\in \partial U_R\}
    \label{eq:SDE for approximate SDE deterministic m}
\end{equation}
such that:
\begin{equation}
    \sup_{t\leq T}\lvert\lvert \mathcal{L}(X_t)-\mathcal{L}(X^R_t)\rvert\rvert_{\TV}\leq \Pm(\tau^R<\tau\wedge T).
\end{equation}

We now establish that:
\begin{equation}\label{eq:zt=approx flow of laws}
    \mathcal{L}(X^R_t)=z_t\quad \text{for}\quad t\leq T,\quad
    (z_t)_{0\leq t\leq T}\in C([0,T);L^1(U_R)).
\end{equation}

Indeed we observe that the following is a martingale for every $\varphi\in C_0^{\infty}(\bar U\times [0,\infty))$:
\[
\varphi(X^R_{t\wedge\tau^R},t\wedge\tau^R)-\varphi(X^R_0,0)-\int_0^{t\wedge \tau^R}(\partial_s+b(m_s,X^R_s)\cdot\nabla+\frac{1}{2}\Delta)\varphi(X^R_s,s)ds,\quad 0\leq t\leq T.
\]
Taking expectation, we see that $\mathcal{L}(X_t)$ - must satisfy the PDE \eqref{eq:lin PDE for fixed (m,J) in proof PDE solns are MKV solns}. Moreover we have $z_0,\mathcal{L}(X_0)\in L^1(U_R)$ and $z_t\otimes dt,\mathcal{L}(X_t)\otimes dt\in L^1(U_R\times [0,T))$. We therefore have $\mathcal{L}(X^R_t)=z_t$ by the uniqueness results of  \cite[Theorem 1.1]{Porretta2015a}, and by \cite[Theorem 3.6]{Porretta2015a} we also have $(z_t)_{0\leq t\leq T}\in C([0,T);L^1(U_R))$.

\underline{Step 3}

We suppose that for all $\epsilon>0$ and $T<\infty$ there exists $R<\infty$ arbitrarily large such that there exists $z\in \mathcal{H}_{R,T}(m)$ with $\sup_{t\leq T}\lvert\lvert z_t-m_te^{-J_t}\rvert\rvert_{\TV} \leq \epsilon$ and $z_0\in L^1(U_R)$. Then we claim:
\begin{equation}\label{eq:(m,J) in Xi and cts wrt L1 step 3}
    (m_t,J_t)_{0\leq t<\infty}\in \Xi\cap C([0,\infty);L^1(U)\times \Rm_{>0}).
\end{equation}

We have from Step $2$ the sequence of solutions $(X^{R_n},\tau^{R_n},W)$ to \eqref{eq:SDE for approximate SDE deterministic m} on the domains $U_{R_n}$ with $R_n\ra \infty$ as $n\ra\infty$ such that:
\[
\begin{split}
    \sup_{t\leq T}\lvert\lvert \mathcal{L}(X_t)-m_te^{-J_t}\rvert\rvert_{\TV}\leq\sup_{t\leq T}\lvert\lvert \mathcal{L}(X^{R_n}_t)-\mathcal{L}(X_t)\rvert\rvert_{\TV}+\sup_{t\leq T}\lvert\lvert \mathcal{L}(X^{R_n}_t)-m_te^{-J_t}\rvert\rvert_{\TV}\\
    \leq \frac{1}{n}+\Pm(\tau^{R_n}<\tau\wedge T).
\end{split}
\]

Since $U\cap B(0,R_n)=U_{R_n}\cap B(0,R_n)$, $\Pm(\tau^{R_n}<\tau\wedge T)\ra 0$ as $n\ra \infty$ hence:
\[
\sup_{t\leq T}\lvert\lvert m_te^{-J_t}-\mathcal{L}(X_t\lvert \tau>t)e^{-\ln\Pm(\tau>t)}\rvert\rvert_{\TV}= 0.
\]
Thus $(X,\tau,W)$ is a global weak solution of \eqref{eq:MKV sde} and therefore $(m_t,J_t)\in \Xi$. Moreover since $(\mathcal{L}(X^{R_n}_t))_{0\leq t\leq T}\in C([0,T);L^1(U))$ for all $n$, $(\mathcal{L}(X_t))_{0\leq t\leq T}\in C([0,T);L^1(U))$. We have established \eqref{eq:(m,J) in Xi and cts wrt L1 step 3}.

\underline{Step 4}

We therefore have that if $(m_t,J_t)_{0\leq t<\infty}\in\mathcal{S}$ then $(m_{t_0+t},J_{t_0+t})_{0\leq t<\infty}\in \Xi$ for all $t_0\in \Qm_{>0}$. We have that:
\[
(m_{t_0+t},J_{t_0+t})_{0\leq t<\infty}\ra (m_t,J_t)_{0\leq t<\infty}\quad\text{in}\quad d^{\infty}\quad\text{as}\quad t_0\ra\infty.
\]
Since $m_{t_0}\ra m_0$ in $\Wah$, Lemma \ref{lem:MKV Solns compact for compact set of ics} allows us to extract a subsequence converging to an element of $\Xi$, hence $(m_t,J_t)_{0\leq t<\infty}\in \Xi$. Moreover since $(m_{t_0+t},J_{t_0+t})_{0\leq t<\infty}\in C([0,\infty);L^1(U)\times \Rm_{\geq 0})$ for all $t_0\in \Qm_{>0}$ we have $(m_t,J_t)_{0\leq t<\infty}\in C((0,\infty);L^1(U)\times \Rm_{\geq 0})$.

\qed

\subsection{Uniqueness in Law of Weak Solutions to \eqref{eq:MKV sde}}
We implement a strategy similar to the proof of \cite[Proposition C.1]{Campi2017}. We fix $\nu\in\mathcal{P}(U)$ and firstly seek to show:
\begin{equation}
\begin{split}    
    (\mathcal{L}^{\nu}(X_t\lvert \tau>t))_{0\leq t<\infty}\quad\text{is unique  amongst}\\ \text{all weak solutions to }\eqref{eq:MKV sde}
    \text{ with initial condition }X_0\sim \nu.
    \label{eq:flow of laws unique for proof}
\end{split}
\end{equation}

We take weak solutions to \eqref{eq:MKV sde} $(X^{1},W^{1},\tau^{1})$ and $(X^2,W^2,\tau^2)$ of \eqref{eq:MKV sde} on the possibly different probability spaces $(\Omega^{1},\mathcal{F}^{1},\Pm^{1})$ and $(\Omega^{2},\mathcal{F}^{2},\Pm^{2})$. We note by our earlier result that these must be global weak solutions. We then define $\mathcal{L}(X^k_t)=u^k_t$ for $k=1,2$ and $t<\infty$.

We recall that $b$ is uniformly Lipschitz in the measure argument with respect to the $\Wah$ metric. Since this metric is dominated by the Total Variation metric (up to a constant), $b$ is uniformly Lipschitz in the measure argument with respect to the Total Variation metric.

By abuse of notation we write:
\[
b:\mathcal{M}_+(U)\times U\ni(u,x)\mapsto b(\frac{u}{\lvert u\rvert_{*}},x)\in\Rm^d,\quad \lvert u\rvert_{*}=u(U)
\]
where $\lvert u\rvert_*$ is the mass of u on $U$. 

Therefore since $\lvert u^1_t\rvert,\lvert u^2_t\rvert\geq \lvert u^1_1\rvert\wedge \lvert u^1_1\rvert>0$ for $0\leq t\leq 1$ there exists $C_{\text{Lip}}<\infty$ such that:
\[
\lvert b(u^1_t,x)-b(u^2_t,x)\rvert\leq C_{\text{Lip}}\lvert\lvert u^1_t-u^2_t\rvert\rvert_{\TV},\quad x\in U,\; t\leq 1.
\]

We now define 
\[
d_t=\sup_{s\leq t}\lvert\lvert u^{2}_s-u^{1}_s\rvert\rvert_{TV}
\]
and drifts 
\[
b^{1}(x,t)=b(u^{1}_t,x)\text{ and }b^2(x,t)=b(u^{2}_t,x).
\]

We consider weak solutions of the following SDE:
\begin{equation}
\begin{split}
dX_t=b^1(X_t,t)dt+dW_t,\quad 0\leq t\leq\tau=\inf\{t:X_{t}\in\partial U\},\quad
X_0\sim\nu.
\end{split}
\label{eq:fv lin PDE ast m fixed}
\end{equation}
Weak solutions to \eqref{eq:fv lin PDE ast m fixed} are unique in law by the same change of measure argument giving that weak solutions to SDEs without killing with bounded measurable coefficients are unique in law; see \cite[Proposition 3.10]{Karatzas1991}.

Clearly $(X^1,W^1,\tau^1)$ on $(\Omega^1,\mathcal{F}^1,\Pm^1)$ is a weak solution of \eqref{eq:fv lin PDE ast m fixed}. We have by Girsanov's theorem (since $b^{1},b^2$ are bounded Novikov's condition is satisfied):
\[
W'_t=W^2_t-\int_0^t(b^{1}(X^{2}_s,s)-b^2(X^2_s,s))ds
\]
is a $\Pm'$-Brownian motion whereby:
\[
\begin{split}
Z_t=\frac{d\Pm'}{d\Pm^{2}}_{\lvert_{\mathcal{F}^{2}_t}}=\varepsilon(Y)_t,\quad
Y_t=\int_0^tb^{1}(X^2_s,s)-b^2(X^2_s,s)dW^2_s.
\end{split}
\]
Therefore we have:
\[
dX^2=b^2(X^2,t)dt+dW^2_t=b^1(X^2,t)dt+dW'_t
\]
so that $(X^2_t,W'_t,\tau^2)$ on $(\Omega^2,\mathcal{F}^2,\Pm')$ is also a weak solution of \eqref{eq:fv lin PDE ast m fixed}. By uniqueness in law of weak solutions to \eqref{eq:fv lin PDE ast m fixed} we have:
\begin{equation}
\mathcal{L}^{\Pm^{1}}(X^{1}_t)=\mathcal{L}^{\Pm'}(X^{2}_t),\quad t\leq 1.
\label{eq:X1 under P1 and X2 under P' have same dist}
\end{equation}

We now fix some measurable set $A\subseteq \Rm$ and see that:
\[
\begin{split}
\lvert u_t^{1}(A)-u^{2}_t(A)\rvert=\lvert \Pm^{1}(X_t^{1}\in A)-\Pm^{2}(X^{2}_t\in A)\rvert=\big\lvert \Pm'(X^{2}_t\in A)-\Pm^{2}(X^{2}_t\in A)\big\lvert\\
=\big\lvert\expE^{\Pm^{2}}\big[\Ind_{X_t\in A}\big(Z_t-1\big)\big]\big\rvert \underbrace{\leq}_{\substack{\text{Holder's}\\\text{inequality}}} \lvert\lvert Z_t-1\rvert\rvert_{L^2(\Pm^{2})}\sqrt{u_1(A)}\leq \lvert\lvert Z_t-1\rvert\rvert_{L^2(\Pm^{2})}.
\end{split}
\]
Taking the supremum over measurable sets $A\subseteq \Rm$ we have:
\[
\lvert\lvert u^{1}_t-u^{2}_t\rvert\rvert_{TV}^2\leq \expE^{\Pm^{2}}[(Z_t-1)^2]=\expE^{\Pm^{2}}[Z_t^2]-2\underbrace{\expE^{\Pm^{2}}[Z_t]}_{\substack{=1\text{ as }Z_t\text{ is a}\\\Pm^{2}\text{-martingale}}}+1=\expE^{\Pm^{2}}[Z_t^2]-1.
\]
We calculate the first term on the right using Ito's formula:
\[
\expE^{\Pm^{2}}[Z^2_t]=1+\int_0^t\expE^{\Pm^{2}}[(b^1-b^2)^2(X_s,s)Z_s^2]ds\leq 1+\int_0^t(C_{\text{Lip}}d_s)^2\expE^{\Pm^{2}}[Z^2_s]ds.
\]
By Gronwall's inequality, using that $d_t\leq 1$ and $e^{rt}\leq 1+rte^r$ for $0\leq t\leq 1$ we have:
\[
\expE^{\Pm^{2}}[Z^2_t]\underbrace{\leq}_{\text{Gronwall}} e^{\int_0^t(C_{\text{Lip}}d_s)^2}\leq e^{C_{\text{Lip}}^2d_t^2t}\leq 
1+C_{\text{Lip}}^2d_t^2e^{C_{\text{Lip}}^2} t\quad\text{for}\quad 0\leq t\leq 1.
\]
Therefore we have:
\[
\lvert\lvert u_t^{1}-u_t^2\rvert\rvert^2_{TV}\leq C_{\text{Lip}}^2d_t^2e^{C_{\text{Lip}}^2} t\quad\text{for}\quad 0\leq t\leq 1.
\]
Therefore for some $C<\infty$ we have:
\[
d_t\leq C\sqrt{t}d_t.
\]

Thus for $t<\frac{1}{2C^2}\wedge 1$ we have $u^{2}_t=u^{1}_t$. By iteration we have $u^1_t=u^2_t$ for $t\leq 1$. Repeating inductively we have $u^1_t=u^2_t$ for all $t<\infty$. This implies \eqref{eq:flow of laws unique for proof}.

This then implies uniqueness in law. Indeed \eqref{eq:flow of laws unique for proof} implies that both $(X^{1},W^{1},\tau^{1})$ and $(X^2,W^2,\tau^2)$ are weak solutions to \eqref{eq:fv lin PDE ast m fixed} and hence are equal in law.

\qed
\subsection{Proof of Proposition \ref{prop:Properties of the McKean-Vlasov Process} and Theorem \ref{theo:Hydrodynamic Limit Theorem}}\label{subsection: completing MKV well-posedness proof}

Given $\nu \in \mathcal{P}_{\Wah}(U)$, let $\{\vec{X}_t^N\}_{N \geq 2}$ be any sequence of weak solutions to \eqref{eq:N-particle system sde} with initial conditions $\vec{X}^N_0$ such that the (random) empirical measures $m^N_0 = \vartheta^N_{\#}\vec{X}^N_0$ converge in $\mathcal{P}_{\Wah}(U)$ to $\nu$, in probability as $N \to \infty$. This can be achieved, for example, by taking $\vec{X}^N_0 \sim \nu^{\otimes N}$.

Next, we define $m^N_t$ and $J^N_t$ as in \eqref{eq:defin mN} and \eqref{eq:defin JN}:
\[
m^N_t=\vartheta^N(\vec{X}^N_t),\quad J^N_t = \frac{1}{N} \sup \{ k \in \mathbb{N}\;|\; \tau_k \leq t \}.
\]
Theorem \ref{theo:N ra infty theorem} and the fact that $m^N_0\ra \nu$ in probability imply that the laws of $(m^N_t,J^N_t)_{0\leq t<\infty}$ are tight in $\mathcal{P}((D([0,\infty);\mathcal{P}_{\Wah}(U)\times \Rm_{\geq 0})),d^{D}))$ and every limit distribution of this family is supported on $\Xi(\{\nu\})\cap C((0,\infty);L^1(U)\times \Rm_{\geq 0})$. In particular, $\Xi(\{\nu\})\cap C((0,\infty);L^1(U)\times \Rm_{\geq 0})$ is non-empty. We have already proved uniqueness in law of weak solutions to \eqref{eq:MKV sde}; therefore this limit distribution is uniquely determined. Taken together with Lemma \ref{lem:MKV Solns compact for compact set of ics}, this establishes Proposition \ref{prop:Properties of the McKean-Vlasov Process}

The fact that the limit distribution is unique implies convergence along the entire sequence $N \to \infty$ in probability to the same element of $\Xi(\{\nu\})\cap C((0,\infty);L^1(U)\times \Rm_{\geq 0})$. Furthermore, since convergence in $d^D$ to a continuous function implies convergence in $d^{\infty}$, we have convergence in $d^{\infty}$ in probability. This proves Theorem \ref{theo:Hydrodynamic Limit Theorem}.

}

\qed

\section{Properties of the Semigroup $G_t$ - Proposition \ref{prop:basic properties of Gt}}\label{section:Further Estimates}

Our goal in this section is to establish Proposition \ref{prop:basic properties of Gt}. We begin with a proof of \eqref{eq:joint continuity of semigroup}. We take $(t_n,\nu_n)\ra (t,\nu)$ and $T>\sup_nt_n$. Then Lemma \ref{lem:MKV Solns compact for compact set of ics} and Proposition \ref{prop:Properties of the McKean-Vlasov Process} imply that $(G_t(\nu_n))_{0\leq t\leq T}\ra (G_t(\nu))_{0\leq t\leq T}$ in $d^{\infty}_{[0,T]}$ as $n\ra\infty$. Therefore:
\[
\Wah(G_{t_n}(\nu_n),G_t(\nu))\leq \Wah(G_{t_n}(\nu_n),G_{t_n}(\nu))+\Wah(G_{t_n}(\nu),G_t(\nu))\ra 0\quad\text{as}\quad n\ra \infty.
\]
We have thus established \eqref{eq:joint continuity of semigroup}.

We now assume $U$ is bounded, fix $t_0>0$ and combine the estimates on the $N$-particle system we established in Part \ref{enum:bound on mass near bdy after small time} of Proposition \ref{prop:bound on mass near bdy} with the hydrodynamic convergence theorem (Theorem \ref{theo:Hydrodynamic Limit Theorem}) to prove that $\text{Image}(G_{t_0})\subset\subset \mathcal{P}_{\Wah}(U)$.

Let $(\vec{X}^N_t:0\leq t<\infty)=((X^{N,1}_t,\ldots,X^{N,N}_t):0\leq t<\infty)$ be a sequence of weak solutions to \eqref{eq:N-particle system sde} with initial conditions $\vec{X}^N_0\sim \nu^{\otimes N}$. We define $m^N_t=\vartheta^N(\vec{X}^N_t)$ as in \ref{eq:defin mN}. Therefore Part \ref{enum:bound on mass near bdy after small time} of Proposition \ref{prop:bound on mass near bdy} gives that for all $\epsilon>0$ there exists $c=c(\epsilon,t_0)$ dependent only upon $t_0$, $\epsilon>0$, the upper bound on the drift $B<\infty$ and the constant of the interior ball condition $r>0$ such that the compact set $K_{\epsilon,t_0}=V_{c(\epsilon,t_0)}$ satisfies:
\[
\lim_{N\ra\infty}\Pm(m^N_{t_0}(K_{\epsilon,t_0}^c)\geq \epsilon)=0.
\]
Therefore by our hydrodynamic convergence theorem (Theorem \ref{theo:Hydrodynamic Limit Theorem}) we have:
\[
G_{t_0}(\nu)(K_{\epsilon,t_0}^c)\leq \epsilon.
\]

Since $K_{\epsilon,t_0}$ was not dependent upon $\nu$, $G_{t_0}(\nu)(K_{\epsilon,t_0}^c)\leq \epsilon$ for all $\nu\in\mathcal{P}(U)$. Therefore:
\[
G_{t_0}(\nu)\in \{\mu\in\mathcal{P}(U):\mu(K_{2^{-n},t_0}^c)\leq 2^{-n},\quad n\in\mathbb{N}\},
\]
which is a tight family of measures on U.

\qed

\section{Existence and Properties of QSDs - Proposition \ref{prop:properties of QSD}}\label{section:QSDs and their properties}

\subsection*{Parts \ref{enum:QSD QLD equivalent} and \ref{enum:tau exp distributed} of Proposition \ref{prop:properties of QSD}}

We firstly establish $\ref{enum:QSD}\Leftrightarrow \ref{enum:QLD}$. It is trivial to see that a QSD is a QLD. In the opposite direction we consider a QLD $\pi$ with $G_t(\nu)\ra \pi$ in $\Wah$ as $t\ra\infty$. We define the following continuous map:
\[
p:(C([0,\infty);\mathcal{P}_{\Wah}(U)\times \Rm_{\geq 0}),d^{\infty})\ni (f,g)\mapsto f\in C([0,1];\mathcal{P}_{\Wah}(U)).
\]
We further define:
\[
\begin{split}
\zeta_1(\kappa) := \{G_t(\mu)_{0\leq t\leq 1}:\mu\in\kappa\}=p(\Xi(\kappa)).
\end{split}
\]

We have by Lemma \ref{lem:MKV Solns compact for compact set of ics} that $\Xi(\kappa)$ is compact in $(C([0,\infty);\mathcal{P}_{\Wah}(U)\times \Rm_{\geq 0}),d^{\infty})$ for compact $\kappa$, hence $\zeta_1(\kappa)$ is compact in $C([0,1];\mathcal{P}_{\Wah}(U))$ as it is the continuous image of a compact set. We now take $\kappa=\{G_n(\nu):n\in\mathbb{N}\}\cup\{\pi\}$ which is compact in $\mathcal{P}_{\Wah}(U)$. Thus we have 
\[
\zeta_1(\kappa)=\{(G_{n+t}(\nu))_{0\leq t\leq 1}:n\in\mathbb{N}\}\cup\{G_t(\pi)_{0\leq t\leq 1}\}
\]
is compact in $C([0,1];\mathcal{P}_{\Wah}(U))$. We note that:
\[
(G_{n+t}(\nu))_{0\leq t\leq 1}\ra (\pi)_{0\leq t\leq 1}\quad\text{in}\quad C([0,1];\mathcal{P}_{\Wah}(U))\quad \text{as}\quad n\ra\infty.
\]
Thus we must have:
\[
(\pi)_{0\leq t\leq 1}\in \zeta_1(\kappa).
\]
Therefore $G_t(\pi)=\pi$ for $0\leq t\leq 1$. Thus $\pi$ is a QSD.

We now establish $\ref{enum:QSD}\Rightarrow\ref{enum:PDE for QSD}$ along with Part \ref{enum:tau exp distributed} of Proposition \ref{prop:properties of QSD}. We take $\pi$ a QSD, $(X,\tau,W)$ a global weak solution to \eqref{eq:MKV sde} with initial condition $X_0\sim \pi$ and $(m_t,J_t)_{0\leq t\leq 1}=(\mathcal{L}(X_t\lvert \tau>t),-\ln\Pm(\tau>t))_{0\leq t<\infty}\in \Xi(\{\pi\})$. By considering the martingale problem we see that $m_te^{-J_t}=\pi e^{-J_t}=\mathcal{L}(X_t)$ satisfies:
\begin{equation}
\begin{split}
e^{-J_t}\langle \pi(.),\varphi(.)\rangle-\langle \pi(.),\varphi(.)\rangle=\int_0^t e^{-J_s}\langle \pi, b(\pi,.)\cdot \nabla \varphi\\+\frac{1}{2}\Delta\varphi\rangle ds,\quad 0\leq t<\infty, \quad \varphi\in C_0^{\infty}(\bar U).
\end{split}
\label{eq:Fokker-Planck with const in time varphi to prove Prop QSD}
\end{equation}
Clearly the right hand side is differentiable in time, so the left hand side must be also and so we have:
\[
-\langle \pi(.),\varphi(.)\rangle e^{-J_t}\frac{d}{dt}J_t= e^{-J_t}\langle \pi, b(\pi,.)\cdot \nabla \varphi+\frac{1}{2}\Delta\varphi\rangle,\quad \varphi\in C_0^{\infty}(\bar U).
\]
Thus $\frac{d}{dt}J_t$ must be constant and so equal to some $\lambda\geq 0$. Since we can't have $\Pm(\tau>t)=1$ for all t>0, we must have $\lambda>0$. Moreover we must have $\mathcal{L}_{\pi}(X_1)\in L^1(U)$ since $\mathcal{L}_{\pi}(X_1)$ can be related to the distribution at time 1 of Brownian motion killed at the boundary by a Girsanov transformation - thus we must have $\pi\in L^1(U)$. Thus $(\pi,\lambda)$ satisfies \eqref{eq:equation for det stat solns PDE} and hence $\ref{enum:QSD}\Rightarrow\ref{enum:PDE for QSD}$. Moreover $e^{-\lambda t}=e^{-J_t}=\lvert \mathcal{L}(X_t)\rvert=\Pm(\tau>t)$ so that $\tau\sim \text{exp}(\lambda)$ and hence we have Part \ref{enum:tau exp distributed} of Proposition \ref{prop:properties of QSD}.

We now establish $\ref{enum:PDE for QSD}\Rightarrow\ref{enum:QSD}$. We take $(\pi,\lambda)\in L^1(U)\times (0,\infty)$ a solution of \eqref{eq:equation for det stat solns PDE} and take $(X,\tau,W)$ a weak solution of the SDE (which exists by Girsanov's theorem):
\[
\begin{split}
dX_t=b(\pi,X_t)dt+dW_t,\quad 0\leq t\leq\tau=\inf\{t:X_{t-}\in\partial U\},\quad 
X_0\sim \pi.
\end{split}
\]

We have both $\mathcal{L}_{\pi}(X_t)=\mathcal{L}_{\pi}(X_t\lvert \tau>t)\Pm(\tau>t)$ and $\pi e^{-\lambda t}$ must be $L^1_{\text{loc}}(\bar U\times [0,\infty))$ solutions to the PDE (for every $T<\infty$):
\[
\begin{split}
\langle y_t(.),\varphi(.,t)\rangle-\langle \pi(.),\varphi(.,0)\rangle=\int_0^t \langle y_s(.),\partial_s \varphi(.,s)\\ +b(\pi,.)\cdot \nabla \varphi(.,s)+\frac{1}{2}\Delta\varphi(.,s)\rangle ds, 
\quad 0\leq t<\infty,\quad \varphi\in C_0^{\infty}(\bar U\times [0,\infty)).
\end{split}
\]
Therefore by \cite[Theorem 1.1]{Porretta2015a} we have $\pi e^{-\lambda t}=\mathcal{L}(X_t)$, thus $\mathcal{L}(X_t\lvert \tau>t)=\pi$ and hence $(X,\tau,W)$ satisfies \eqref{eq:MKV sde}. Thus $\pi$ is a QSD.

\subsection*{Part \ref{enum:Pi compact} of Proposition \ref{prop:properties of QSD}}
We define $\Pi_n=\{\pi\in\mathcal{P}(U):G_{2^{-n}}(\pi)=\pi\}$. We recall that Proposition \ref{prop:basic properties of Gt} gives that $G_{2^{-n}}:\mathcal{P}(U)\ra\mathcal{P}(U)$ is continuous with tight image. Since the convex hull of a tight family of measures is tight, the closed convex hull $F_{2^{-n}}:=\overline{\text{Conv}}(\text{Image}(G_{t_0}))$ is compact in $\mathcal{P}(U)$. Therefore $\Pi_n$ corresponds to the fixed points of the following map:
\[
G_{2^{-n}}:F_{2^{-n}}\ra F_{2^{-n}}
\]
which is a continuous map from a compact convex subset of a locally convex topological vector space ($\mathcal{M}(U)$) to itself. Thus Schauder's fixed point theorem implies $\Pi_n$ is a non-empty compact subset of $\mathcal{P}(U)$. It is therefore sufficient to prove:
\begin{equation}
\Pi=\cap_n\Pi_n
\label{eq:Pi=intersection of Pins}
\end{equation}
as the intersection of a descending sequence of non-empty compact sets must be non-empty and compact.

We clearly have that $\Pi\subseteq \Pi_n$ for all n, so it is sufficient to establish $\cap_n\Pi_n\subseteq \Pi$. We suppose $\pi\in\cap_n\Pi_n$ and fix $t>0$. We take a sequence of dyadic rationals $t_n\ra t$. We have $\pi=G_{t_n}(\pi)\ra G_t(\pi)$ by Proposition \ref{prop:basic properties of Gt} so that we have $G_t(\pi)=\pi$. Since t is arbitrary $\pi\in\Pi$.

\qed

\section{QSDs as Limits of the Fleming-Viot Particle System}

The goal of this section is to establish that QSDs may be obtained as limits of the $N$-particle system. In Theorem \ref{theo:main t theo conv of stat dists} we show that the stationary distributions of the $N$-particle system converge to the set of QSDs. In Theorem \ref{theo:main t theo jt conv} we then establish under an additional assumption on the semigroup $G_t$ convergence as $N$ and $T$ go to infinity together.

\subsection*{Proof of Theorem \ref{theo:main t theo conv of stat dists}}
We take the $N$-particle stationary distributions $\psi^N$, associated to which are the corresponding stationary empirical measures $\Psi^N=\vartheta^N_{\#}\psi^N$ as in \eqref{eq:defin PsiN} and $\mathcal{P}_{\Wah}(U)$-valued random variables $\pi^N\sim \Psi^N$. We consider a sequence of stationary solutions $\vec{X}^N$ to \eqref{eq:N-particle system sde} with initial distributions $\psi^N$. We write $m^N_t:=\vartheta^N(\vec{X}^N_t)$ and $J^N_t = \frac{1}{N} \sup \{ k \in \mathbb{N}\;|\; \tau_k \leq t \}$ as in \eqref{eq:defin mN} and \eqref{eq:defin JN}. 

Since $m^N_1\sim \Psi^N$, Proposition \ref{prop:laws at positive times constrained to compact set} gives that $\{\Psi^N\}$ are a tight family of random measures. We may then use Theorem \ref{theo:N ra infty theorem} to establish that $\{\mathcal{L}((m^N,J^N_t)_{0\leq t<\infty})\}$ is tight in $\mathcal{P}((D([0,\infty);\mathcal{P}_{\Wah}(U)\times \Rm_{\geq 0}),d^{D}))$. We then consider an arbitrary convergent subsequence, along which $\Psi^N\ra \Psi^{\infty}$ in $\mathcal{P}(\mathcal{P}_{\Wah}(U))$ and $(m^N,J^N_t)_{0\leq t<\infty}\ra (m^{\infty},J^{\infty}_t)_{0\leq t<\infty}$ in distribution, which must satisfy $(m^{\infty},J^{\infty}_t)_{0\leq t<\infty}\in \Xi$ almost surely by Theorem \ref{theo:N ra infty theorem}. We take a random variable $\pi^{\infty}\sim \Psi^{\infty}$ so that we have:
\[
\Psi^{\infty}\sim m^{\infty}_t=G_t(m^{\infty}_0),\quad t\geq 0,
\]
so that in particular $\Psi^{\infty}$ is an invariant measure for the semigroup $G_t$. We calculate:
\[
\begin{split}
\expE[\Wah(\pi^{\infty},\Pi)]=\expE[\Wah(G_t(\pi^{\infty}),\Pi)]\ra 0 \quad \text{as}\quad t\ra\infty\\
\text{by Lebesgue's Dominated Convergence Theorem and our assumption }\eqref{eq:convergence to non-unique quasi-equil}.
\end{split}
\]
Thus $\pi^{\infty}\in\Pi$ almost surely, so $\Psi^{\infty}$ is supported on $\Pi$. Thus along every subsequence, there is a further subsequence along which $W(\pi^N,\Pi)\ra 0$ in probability, hence we have convergence in probability along the original sequence.

\qed

\subsection*{Proof of Theorem \ref{theo:main t theo jt conv}}

We take an arbitrary sequence $t_N\ra \infty$ and fix $\epsilon>0$. We take (using the assumption \eqref{eq:uniform convergence to unique equilibria}) $T<\infty$ such that $\Wah(G_T(\nu),\pi)\leq \epsilon$ for all $\nu\in \mathcal{P}_{\Wah}(U)$. 

Then by Proposition \ref{prop:laws at positive times constrained to compact set}, $\{\mathcal{L}(m^N_{t_N-T})\}$ is tight in $\mathcal{P}(\mathcal{P}_{\Wah}(U))$ and hence by Theorem \ref{theo:N ra infty theorem} and Skorokhod's representation theorem we may take a subsequence and possibly different probability space $(\Omega',\mathcal{F}',\Pm')$ on which $(m^N_{t_N-T+t},J^N_{t_N-T+t})_{0\leq t<\infty}$ converges in $d^{\infty}$ $\Pm'$-almost surely to $(m_t,J_t)_{0\leq t<\infty}\in \Xi\subseteq C([0,\infty);\mathcal{P}_{\Wah}(U)\times \Rm_{\geq 0})$. Then $m_T=G_T(m_0)$ so that on this subsequence:
\[
\limsup_{N\ra\infty}\Wah(m^N_{t_N},\pi)\leq \epsilon\quad \Pm'\text{-almost surely.}
\]
This subsequence was arbitrary as was $\epsilon>0$, so we have $\Wah(m^N_{t_0},\pi)\ra 0$ in probability as $N\wedge t_0\ra \infty$. Using Theorem \ref{theo:Hydrodynamic Limit Theorem} and Proposition \ref{prop:properties of QSD} we are done.

\qed

\section{Appendix}
Here we prove various technical lemmas, whose proofs we have deferred to this appendix.

\subsection{Proof of Lemma \ref{lem:condition to prove for sets in R for abs cty lemma}}

We define:
\[
G_{\epsilon}=0\vee \sup_{\rho \in \mathcal{R}} \big( m(\rho) - C_{\epsilon}(T_{\min}(\rho))\text{Leb}(\rho) \big).
\]
Since we have bounded $m$ on sets in $\mathcal{R}$ in terms of $\text{Leb}$, we may bound the corresponding outer measure $m^{\ast}$ in terms of the outer measure $\text{Leb}^{\ast}$. Specifically:
\[
m^*(E)\leq C_{\epsilon}(T_{\min}(E))\text{Leb}^*(E)+G_{\epsilon}\quad\text{for all}\quad E\in\mathcal{B}(U\times [0,T])\setminus\{\emptyset\}
\]
holds almost surely. Since $m\leq m^{\ast}$ and $\text{Leb}=\text{Leb}^{\ast}$, this implies that
\begin{equation}
m(E)\leq C_{\epsilon}(T_{\min}(E))\text{Leb}(E)+G_{\epsilon}\quad\text{for all}\quad E\in\mathcal{B}(U\times [0,T])\setminus\{\emptyset\}
\label{eq:bound for m in terms of CLeb and G}
\end{equation}
holds $\Pm^{\epsilon}$-almost surely. We define for $\delta,T_0\geq 0$:
\[
\begin{split}
\mathcal{N}_{\delta,T_0}=\{\mu\in \mathcal{P}(U\times[0,T]):
\mu(N)\leq \delta\quad\text{for all}\quad N\in\mathcal{B}( U\times (T_0,T])\\\text{such that }\text{Leb}_{U\times (0,T]}(N)=0\}.
\end{split}
\]
Then \eqref{eq:bound for m in terms of CLeb and G} implies that for $\delta,T>0$ we have:
\[
\begin{split}
\Pm(m\in \mathcal{N}_{\delta,T_0})\geq \Pm(G_{\epsilon}\leq \delta)\geq 1-\frac{\epsilon}{\delta}\text{ by Markov's inequality.}
\end{split}
\]
Since $\epsilon>0$ is arbitrary, we have $\Pm(m\in\mathcal{N}_{\delta,T_0})=1$ for all $\delta,T>0$. We now note that $\mathcal{N}_{0,0}=\cap_{T_0>0}\cap_{\delta>0}\mathcal{N}_{\delta,T_0}$ so that:
\[
\Pm(m\ll\text{Leb}_{U\times (0,T]})=\Pm(m\in\mathcal{N}_{0,0})=1.
\]
Moreover we have $\Pm(m(U\times\{0\})=0)=1$. Therefore $\Pm(m\ll\text{Leb}_{U\times [0,T]})=1$.

\qed

\subsection{Proof of Lemma \ref{lem:Boundary Truncation Lemma}}

We fix $\varphi$ a positive mollifier supported on $B(0,1)$ and take $\rho\in C^{\infty}(\Rm^d)$ such that:
\[
U=\{\rho>0\},\quad \partial U=\{\rho=0\},\quad \nabla \rho\neq 0\text{ on }\partial U.
\]
We define $g=(\varphi\ast \Ind_{B(0,R+3)})\rho \in C^{\infty}_c(\Rm^d)$ which we note satisfies:
\begin{enumerate}
    \item $U\cap \bar B(0,R+2)=\{g>0\}\cap \bar B(0,R+2)$.
    \item $\partial U\cap \bar B(0,R+2)=\{g=0\}\cap \bar B(0,R+2)$.
    \item $\nabla g\neq 0$ on $\partial U\cap \bar B(0,R+2)$. 
\end{enumerate}

We then define $h=\varphi\ast \Ind_{B(0,R+1)^c}\in C_c^{\infty}(\Rm^d)$ which we note satisfies:
\begin{enumerate}
    \item $h\equiv 0$ on $\bar B(0,R)$.
    \item $0\leq h\leq 1$ on $\bar B(0,R+2)\setminus \bar B(0,R)$.
    \item $h\equiv 1$ on $B(0,R+2)^c$.
\end{enumerate}

We then define $f=g-\epsilon h\in C^{\infty}(\Rm^d)$ for $\epsilon>0$ to be determined and claim that for some $\epsilon>0$ small enough $U_R:=\{f>0\}$ gives a domain with our desired values. We firstly observe that for all $\epsilon>0$, 
\[
U\cap \bar B(0,R)\subseteq \{f>0\}\subseteq U\cap \bar B(0,R+4).
\]
Therefore by the implicit function theorem it is sufficient to show that for some $\epsilon>0$ small enough:
\begin{equation}
f=0\;\Rightarrow \;\nabla f\neq 0.
\end{equation}
Sard's theorem allows us to take $\epsilon_n\downarrow 0$ such that
\[
g=\epsilon_n\Rightarrow \nabla g\neq 0.
\]
Therefore $f_n=g-\epsilon_n h$ satisfies:
\[
f_n(x)=0\quad \text{and} \quad\lvert x\rvert\geq R+2\;\Rightarrow\; \nabla f_n(x)\neq 0.
\]
We now assume for contradiction there exists for all n $\lvert x_n\rvert \leq R+2$ such that $f_n(x_n)=0$ and $\nabla f_n(x_n)=0$. We take a convergent subsequence $x_{n_k}\ra x\in \bar B(0,R+2)$, so that $0=f_{n_k}(x_{n_k})\ra g(x)$ and $0=\nabla f_{n_k}(x_{n_k})\ra \nabla g(x)$. This is a contradiction, hence we may choose $\epsilon_{n}$ such that:
\[
g-\epsilon_n h=0\Rightarrow \nabla(g-\epsilon_nh)\neq 0.
\]
\qed

\subsection{Controls on the Density and Hitting Time of Generic Diffusions}
\begin{lem}\label{lem:upper and lower bounds on density of diffusions}
Let $(X,\tau,W)$ be a weak solution of the following SDE:
\begin{equation}
dX_t=b_tdt+dW_t,\quad 0\leq t\leq \tau=\inf\{t>0:X_{t}\in \partial U\}
\end{equation}
on the domain $U\subseteq \Rm^d$ and filtered probability space $(\Omega,\mathcal{F},(\mathcal{F}_t)_{t\geq 0},\Pm)$ where $b_t$ is $(\mathcal{F}_t)_{t\geq 0}$-adapted and uniformly bounded $\lvert b_t\rvert\leq B$. For $\vec{x}=(x_1,\ldots,x_{d})\in U\subseteq \Rm^d$ and $h>0$ we define the open cube:
\[
D_{h}(\vec{x})=\{\vec{y}=(y_1,\ldots,y_d)\in U:\lvert y_i-x_i\rvert<h\}.
\]
Throughout we write $\mathcal{L}(X_t)=\mathcal{L}(X_t\lvert \tau>t)\Pm(\tau>t)$ for the law of the killed process restricted to U. Then we have the following:
\begin{enumerate}
    \item There exists a non-increasing function $C:(0,\infty)\ra\Rm_{>0}$ such that:
    \begin{equation}
        \mathcal{L}(X_t)(.)\leq C_t\text{Leb}_U(.),\quad 0<t<\infty.
        \label{eq:upper bound on density for general diffusion}
    \end{equation}
    \item If $h, t>0$ and $D_{5h}(\vec{x})\subseteq U$ there exists $c>0$ dependent only upon the upper bound on the drift $B<\infty$, $t>0$ and $h>0$ such that:
    \begin{equation}
        \mathcal{L}(X_t)_{\lvert_{D_h(\vec{x})}}(.)\geq c\Pm(X_t\in D_h(\vec{x}))\text{Leb}_{\lvert_{D_h(\vec{x})}}(.).
        \label{eq:lower bound on density for general diffusion}
    \end{equation}
\end{enumerate}
\end{lem}

We obtain from this the following corollary:
\begin{cor}
For every $t\geq 0$ we have $\Pm(\tau=t)=0$.
\label{cor:hitting time Corollary}
\end{cor}

\begin{proof}[Proof of Lemma \ref{lem:upper and lower bounds on density of diffusions}]
We firstly establish \eqref{eq:upper bound on density for general diffusion}. We may apply Lemma \ref{lem:dominated family of Ys} to the family of processes $\{X_t-\vec{x}:\vec{x}\in U\}$ to see that:
\[
\Pm(X_t\in D_h(\vec{x}))\leq C_t\text{Leb}(D_h(\vec{x}))
\]
where $C_t$ is the function given by Lemma \ref{lem:dominated family of Ys}. By considering the outer measure generated by the open cubes, we see that:
\[
\mathcal{L}(X_t)(.)\leq \mathcal{L}(X_t)^{\ast}(.)\leq C_t\text{Leb}^{\ast}(.)=C_t\text{Leb}(.),\quad 0<t<\infty.
\]

We now turn to establishing \eqref{eq:lower bound on density for general diffusion}. We consider on the probability space $(\Omega,\mathcal{F},(\mathcal{F}_t)_{t\geq 0},\Pm)$ a family of weak solutions $(X^{\gamma},W^{\gamma})$ ($\gamma\in\Gamma$) on the domains $U^{\gamma}\supseteq D_{4h}(\vec{0})$ to:
\[
dX^{\gamma}_t=b^{\gamma}_tdt+dW^{\gamma}_t,\quad 0\leq t\leq\tau^{X,\gamma}=\inf\{t>0:X^{\gamma}_{t}\in \partial U^{\gamma}\}
\]
where $\{b^{\gamma}\}$ are bounded $\lvert b^{\gamma}_t\rvert\leq B$ and $(\mathcal{F}_t)_{t\geq 0}$-adapted processes. We take $\vec{h}=(h,\ldots,h)$ and write $\vec{n}(\vec{x})$ for the inward normal of the positive orthant $\Rm^d_{>0}$. If we repeat the proof of Lemma \ref{lem:dominated family of Ys} (on page \pageref{proof:dominated family of Ys}) with strong solutions of the SDE \eqref{eq:equation for Y} replaced with strong solutions of the 1-dimensional SDE (which exists by \cite[Theorem 1.3]{Andres}):
\[
dZ_t=d\tilde{W}_t+Bdt+dL^{Z}_t,\quad Z_0=2h,
\]
we obtain for each $\gamma$ a strong solution $(Z^{\gamma},\tilde{W}^{\gamma})$ of the d-dimensional SDE:
\[
dZ^{\gamma}_t=(B,\ldots,B)dt+\vec{n}(Z^{\gamma}_t)dL^{Z^{\gamma}}_t,\quad 0\leq t\leq\tau^{Z^{\gamma}}=\inf\{t>0:Z^{\gamma}_t\in \partial D_{4h}(\vec{0})\},\quad Z^{\gamma}_0=2\vec{h}
\]
where $L^{Z^{\gamma}}_t$ is the local time of $Z^{\gamma}_t$ with the boundary $\partial \Rm^d_{>0}$ and which satisfies:
\begin{equation}
X^{\gamma}_0\in D_{2h}(\vec{0})\text{ and }Z_t^{\gamma}\in D_{\epsilon}(\vec{0})\Rightarrow X^{\gamma}_t \in D_{\epsilon}(\vec{0}),\quad 0<\epsilon< h.
\end{equation} 

Moreover we may take $c>0$ such that for all $h>\epsilon>0$ and $\gamma\in\Gamma$ we have $\Pm(Z^{\gamma}_t\in D_{\epsilon}(\vec{0}))\geq c\text{Leb}(D_{\epsilon}(\vec{0}))$. Therefore by considering the processes $\{X_t-\vec{y}:\vec{y}\in D_{h}(\vec{x})\}$ we see that:
\[
\Pm(X_t\in D_{\epsilon}(\vec{y}))\geq \Pm(X_{0}\in D_{2h}(\vec{y}))c\text{Leb}(D_{\epsilon}(\vec{0}))\geq \Pm(X_{0}\in D_{h}(\vec{x}))c\text{Leb}(D_{\epsilon}(\vec{0})).
\]
Therefore by considering the inner measure generated by the open cubes, we see that:
\[
\mathcal{L}(X_t)(.)\geq \mathcal{L}(X_t)_{\ast}(.)\geq c\text{Leb}_{\ast}(.)=c\text{Leb}(.)
\]

\end{proof}

\begin{proof}[Proof of Corollary \ref{cor:hitting time Corollary}]
The $t=0$ case is trivial, so we assume $t>0$. Indeed by the continuity of the paths of $X_t$, we have for all $R<\infty$ and $\epsilon>0$:
\[
\Pm(\tau=t)\leq \limsup_{h\ra\infty}\Pm(d(X_{t-h},\partial U)\leq \epsilon\quad\text{and}\quad \lvert X_{t-h}\rvert \leq R+1\})+\Pm(\lvert X_{t}\rvert \geq R).
\]

Applying \eqref{eq:upper bound on density for general diffusion} we have:
\[
\Pm(\tau=t)\leq C_{\frac{t}{2}}\text{Leb}(\{x:d(x,\partial U)\leq \epsilon\quad \text{and}\quad\lvert x\rvert\leq R+1)+\Pm(\lvert X_{t}\rvert \geq R).
\]

Taking $\limsup_{R\ra\infty}\limsup_{\epsilon\ra 0}$ of both sides we are done.
\end{proof}

\bibliography{fvlib}
\bibliographystyle{plain}

\end{document}